%% file: wigner-strichartz-arXiv-v2-corrected.tex
\documentclass[reqno,12pt]{gen-j-l}
\usepackage{amssymb,amsmath,amsbsy,amsthm,esint,setspace,color}
\usepackage{hyperref}
\usepackage{amsrefs}
\usepackage{enumitem}
\usepackage[normalem]{ulem}

\newtheorem{theorem}{Theorem}[section]
\newtheorem{lemma}[theorem]{Lemma}
\newtheorem{proposition}[theorem]{Proposition}

\newtheorem{conjecture}{Conjecture}[section]
\newtheorem{remark}{Remark}[section]

\numberwithin{equation}{section}

\let\oldtocsection=\tocsection
 
\let\oldtocsubsection=\tocsubsection
 
\let\oldtocsubsubsection=\tocsubsubsection
 
\renewcommand{\tocsection}[2]{\hspace{0em}\oldtocsection{#1}{#2}}
\renewcommand{\tocsubsection}[2]{\hspace{1em}\oldtocsubsection{#1}{#2}}
\renewcommand{\tocsubsubsection}[2]{\hspace{2em}\oldtocsubsubsection{#1}{#2}}

\begin{document}

\title[Spacetime estimates for Boltzmann]{Small data global well-posedness for a Boltzmann equation via bilinear spacetime estimates}

\author{Thomas Chen}
\address{T. Chen,  
Department of Mathematics, University of Texas at Austin.}
\email{tc@math.utexas.edu}
\author{ Ryan Denlinger}
\address{R. Denlinger,  
Department of Mathematics, University of Texas at Austin.}
\email{denlinger@math.utexas.edu}
\author{Nata\v{s}a Pavlovi\'c}
\address{N. Pavlovi\'c,  
Department of Mathematics, University of Texas at Austin.}
\email{natasa@math.utexas.edu}

\begin{abstract}
We provide a new analysis of the Boltzmann equation with constant collision kernel in two
space dimensions. The scaling-critical Lebesgue space is $L^2_{x,v}$; we prove global well-posedness
and a version of scattering, assuming that the data $f_0$ is sufficiently
 smooth and localized, and the
$L^2_{x,v}$ norm of $f_0$ is sufficiently small. The proof relies upon a new
scaling-critical
 bilinear spacetime estimate for the collision ``gain'' term in
Boltzmann's equation, combined with a novel application of the Kaniel-Shinbrot
iteration.
\end{abstract}

\maketitle

\tableofcontents

\section{Introduction and main results}
\label{sec:intro}

\subsection{Background.}
\label{ssec:background}

Boltzmann's equation describes the time-evolution of the
 phase-space density $f(t,x,v)$ of a dilute gas, accounting for both dispersion under the
free flow and dissipation as the result of collisions. We will be interested in the
Boltzmann equation with \emph{constant collision kernel} in the plane,
$\mathbb{R}^2_x \times \mathbb{R}^2_v$, which is written as follows:
\begin{equation}
\label{eq:MMboltzd2000}
\begin{aligned}
& \left( \partial_t + v \cdot \nabla_x \right) f (t,x,v) = \\
& \qquad \qquad =\int_{\mathbb{S}^1} d\omega \int_{\mathbb{R}^2}
du \left\{ f(t,x,v^*) f(t,x,u^*) - f (t,x,v) f(t,x,u) \right\} 
\end{aligned}
\end{equation}
with prescribed initial data $f (0,x,v) = f_0 (x,v)$, and
$(t,x,v) \in [0,\infty) \times \mathbb{R}^2 \times \mathbb{R}^2$.
Here the symbols $u^*, v^*$ are defined by the \emph{collisional change of variables}
$$
u^* = u + (\omega \cdot (v - u) ) \omega
$$
$$
v^* = v -  ( \omega \cdot (v-u) ) \omega
$$
and $\omega \in \mathbb{S}^1 \subset \mathbb{R}^2$ is a unit vector.
We may also write
\begin{equation} \label{eqB-Q}
\left( \partial_t + v \cdot \nabla_x \right) f = Q(f,f) = Q^+ (f,f) - Q^- (f,f)
\end{equation}
where
\begin{equation} \label{eqQ+} 
Q^+ (f,g) (x,v) = \int_{\mathbb{S}^1} d\omega \int_{\mathbb{R}^2} du
f (x,v^*) g (x,u^*)
\end{equation} 

\begin{equation} \label{eqQ-} 
Q^- (f,g) (x,v) = 2\pi f (x,v) \rho_g (x)
\end{equation} 
and
$$
\rho_f (x) = \int_{\mathbb{R}^2} dv f(x,v)
$$
The PDE (\ref{eq:MMboltzd2000}) is scaling-critical, independently in $x$ and $v$, for the
$L^2 \left( \mathbb{R}^2_x \times \mathbb{R}^2_v \right)$ norm of $f_0$.

The Cauchy problem for (\ref{eq:MMboltzd2000}), specifically with the constant collision kernel,
 is by now a mature subject and many different techniques are available. One of the oldest known
techniques is the Kaniel-Shinbrot iteration \cite{KS1984}, which will be explained in detail in
Section \ref{sec:techprelim}; this is a monotonicity-based
 technique for producing a non-negative solution of
Boltzmann's equation. 
Strichartz estimates have been used in \cite{Ar2011} to solve equations related to (\ref{eq:MMboltzd2000}) but containing
a cut-off in the interaction at large velocities. Scattering was subsequently addressed in
\cite{HeJiang2017}, again using Strichartz estimates.
 Global well-posedness has been proven
near equilibrium by a variety of techniques \cites{GS2011,AMUXY2011,Uk1974,Gu2003}, 
all of which rely somehow on a notion of Dirichlet
form (and sometimes requiring the long-range version of (\ref{eq:MMboltzd2000}), 
e.g.,
\emph{true Maxwell molecules}). 
For more background on Boltzmann's equation we refer the reader to \cite{CIP1994}. Weaker notions
of solution are available globally in time due to DiPerna and Lions \cite{DPL1989}, but uniqueness
remains an open problem for such solutions.

The difficulty with solving (\ref{eq:MMboltzd2000}) at critical regularity
is actually more challenging than appears to be customarily acknowledged, because though the two terms on the
right hand side 
(known as ``gain'' $Q^+$ and ``loss'' $Q^-$ respectively)
 both scale the same way, they do \emph{not} share the
same estimates. In fact, the gain term exhibits a convolutive effect (similar
to $f *_v g$) which is not observed with the loss term.
This problem was acknowledged in \cite{Ar2011} and dealt with by introducing
a cutoff in the collision kernel at large velocities, thereby breaking the
scale-invariance of the problem.

In the
present work, we take the point of view that the data $f_0$ should be \emph{sufficiently localized and regular enough} (in the
sense of weighted $L^2$-based Sobolev spaces) to makes sense of both ``gain'' 
and ``loss'' terms, \emph{but}
that the theorem should 
only depend on the smallness of the critical norm, in this case $L^2$. The advantage
of this approach is that the local iteration relies purely upon energy estimates in $L^2$-based spaces.
In particular, we will prove a bilinear estimate of the form
$$
L^2_{x,v} \times L^2_{x,v} \rightarrow L^1_{t \in \mathbb{R}} L^2_{x,v}
$$
for the $Q^+$ operator (acting on the free flow), which is new to the best of our knowledge.
Once this bilinear estimate is in hand, any
 space of \emph{mixed} integrability in $x,v$, e.g. $L^p_x L^r_v$ with $p \neq r$, arises only
as the result of Sobolev embedding applied to an $L^2$-based Sobolev norm. 
 
In our analysis, we will invoke the approach that we introduced in \cite{CDP2017,CDP-DCDS2019}, based on the Wigner transform of the Boltzmann equation, which makes the problem naturally accessible to a combination of techniques from both kinetic theory, and  dispersive nonlinear PDEs.

\subsection{Summary of the present work.}
\label{ssec:summary}

The subject of this paper is a new treatment of the Boltzmann equation with constant collision
kernel in $d=2$, which is scaling-critical for the space $L^2_{x,v}$. We prove global well-posedness and
scattering for solutions with small norm in the critical space $L^2_{x,v}$, whenever 
$\left\Vert \left< v \right>^{\frac{1}{2}+}
 \left< \nabla_x \right>^{\frac{1}{2}+}
 f_0 \right\Vert_{L^2_{x,v}}$ is \emph{finite} but
not necessarily small. 

Our proof relies on the Kaniel-Shinbrot iteration,
as recommended in the introduction to \cite{Ar2011}. 
As far as we are aware, this is the first time that the Kaniel-Shinbrot iteration
has been implemented outside Maxwellian-weighted $L^\infty$ spaces.
Moreover, a uniqueness result will be proven which
does not require either non-negativity or Sobolev regularity of solutions. 
Therefore, the existence of a non-negative solution from
Kaniel-Shinbrot will imply that any other local solution in the correct \emph{integrability}
 class is automatically
non-negative and coincides with the Kaniel-Shinbrot solution. From there, the extra regularity
is propagated \emph{a posteriori}, 
globally in time (with possibly large growth rate), by constructing sufficiently regular
local solutions and employing standard commutation rules.

Our proof relies on the Wigner transform
and endpoint Strichartz estimates due to Keel-Tao 
\cite{KT1998} for hyperbolic Schr{\" o}dinger equations in the 
\emph{doubled} dimension $2 d = 4$. We point out 
that endpoint kinetic Strichartz estimates are \emph{false} \cite{Bennettetal2014}
in all dimensions. For this reason, there is no obvious analogue of our proof which employs the
kinetic picture exclusively.

\subsection{Main results.}
\label{ssec:mainres}

Our main results are summarized in the following:
\begin{theorem}
\label{thm:MainResults}
There exists a number $\eta_0 > 0$ such that all the following is simultaneously true:

Suppose $f_0 ( x,v ) : \mathbb{R}^2 \times \mathbb{R}^2 \rightarrow \mathbb{R}$ is
a \emph{non-negative}, measurable,  locally integrable function such that
\begin{equation}
\label{eq:Halphadata}
\Big\Vert \left< v \right>^{\frac{1}{2}+}
 \left< \nabla_x \right>^{\frac{1}{2}+}
f_0 (x,v) \Big\Vert_{L^2 \left( \mathbb{R}^2 \times \mathbb{R}^2 \right)}
< \infty
\end{equation}
and
\begin{equation}
\Big\Vert f_0 (x,v) \Big\Vert_{L^2 \left( \mathbb{R}^2 \times \mathbb{R}^2 \right)} 
< \eta_0
\end{equation}
Then there exists a globally defined (for $t \geq 0$) non-negative
mild solution $f \in C \left( [0,\infty), L^2_{x,v} \right)$
of Boltzmann's equation
\begin{equation} \label{eq:MMboltzd200}
\left( \partial_t + v \cdot \nabla_x \right) f (t,x,v) = Q(f,f)
\end{equation}
where 
$$
Q(f,f) = Q^+ (f,f) - Q^- (f,f),
$$
with $Q^+$ and $Q^-$ given respectively in \eqref{eqQ+} and \eqref{eqQ-},  
such that $f(0) = f_0$ and the following
bounds \eqref{eq:MRbd1},\eqref{eq:MRbd2},\eqref{eq:MRbd3}
 hold for any $T \in \left( 0, \infty \right]$ (noting that $T = +\infty$ is included):
\begin{equation}
\label{eq:MRbd1}
\left< v \right>^{\frac{1}{2}+} Q^+ (f,f) \in 
L^1_{t \in [0,T]} L^2_{x,v} 
\end{equation}
\begin{equation}
\label{eq:MRbd2}
\rho_f \in 
L^2_{t \in [0,T]} L^\infty_x \bigcap L^2_{t \in [0,T]} L^4_x 
\end{equation}
\begin{equation}
\label{eq:MRbd3}
\left< v \right>^{\frac{1}{2}+} f \in 
 L^\infty_{t \in [0,T]} L^2_{x,v} \bigcap L^\infty_{t \in [0,T]} L^4_x L^2_v 
\end{equation}
The solution $f(t)$ is unique in the class of all mild solutions,
with the same initial data, satisfying all the
bounds \eqref{eq:MRbd1},\eqref{eq:MRbd2},\eqref{eq:MRbd3})  for each $T \in (0,\infty)$.
In particular, any mild solution with data $f_0$
 satisfying \eqref{eq:MRbd1},\eqref{eq:MRbd2},\eqref{eq:MRbd3} 
is automatically non-negative (since it is equal to $f$). 

The solution $f(t)$ also satisfies:
\begin{equation}
\label{eq:L2longtimebd}
\left\Vert f \right\Vert^2_{L^\infty_{t \geq 0} L^2_{x,v}}
+ \left\Vert Q^+ (f,f) \right\Vert_{L^1_{t \geq 0} L^2_{x,v}} \leq C
\left\Vert f_0 \right\Vert_{L^2_{x,v}}^2
\end{equation}
Moreover, $f(t)$ scatters in $L^2_{x,v}$ as $t \rightarrow +\infty$; equivalently, 
$f_{+\infty} = \lim_{t \rightarrow +\infty} T(-t) f(t)$ exists in the norm topology
in $L^2_{x,v}$. 

Finally, $f(t)$ carries (a posteriori) the same regularity as the initial data:
\begin{equation}
\label{eq:propreg}
\forall T > 0,\qquad \Big\Vert \left< v \right>^{\frac{1}{2}+}
\left< \nabla_x \right>^{\frac{1}{2}+}
 f (t) \Big\Vert_{L^\infty_{t \in [0,T]} L^2_{x,v}}
< \infty
\end{equation}
\end{theorem}

\begin{remark}
We note that no claim is made regarding the injectivity or
non-injectivity for the map
$f_0 \mapsto f_{+\infty}$.
Moreover, no claim is made as to whether or not the bound in 
(\ref{eq:propreg}) is uniform
as $T \rightarrow \infty$.
\end{remark}

\begin{remark}
The constant $C$ appearing in (\ref{eq:L2longtimebd}) is absolute, requiring only
the imposed condition that $\left\Vert f_0 \right\Vert_{L^2} < \eta_0$ for another
absolute constant $\eta_0$.
The existence of such an absolute
 $C$ indicates that the behavior of Boltzmann's equation
is effectively linear on long timescales if the $L^2_{x,v}$ norm of $f_0$ is 
sufficiently small. Note that the bound (\ref{eq:L2longtimebd}) appears to be new.
\end{remark}

\begin{remark}
It is an easy consequence of the $Q^+ (f,f)$ estimate
 (\ref{eq:L2longtimebd}), of Duhamel's formula, and Minkowski's inequality,
along with the \emph{homogeneous} Strichartz estimates, that the
solution of (\ref{eq:MMboltzd200}) 
satisfies $f \in L^q_t L^r_x L^p_v \left( [0,\infty) \times \mathbb{R}^2 \times \mathbb{R}^2 \right)$,
whenever $p,r \geq 1$,
 $q > 2$, $\frac{1}{r}+\frac{1}{p} = 1$, and $\frac{1}{q} = \frac{1}{p} - \frac{1}{r}$. This is
the full range of homogeneous Strichartz estimates expected for $L^2$ solutions of the free transport
equation in $d=2$. We do not mention estimates of this form in Theorem \ref{thm:MainResults} because
they are not relevant to the method of the proof.
\end{remark}

\subsection{The local well-posedness theorem.}

We will also prove the following local well-posedness theorem, 
following a similar line of reasoning. We point out that while
the data is required to have $\frac{1}{2}+$ regularity, the 
\emph{time of existence} depends only on regularity at the
$s$ level for an arbitrary $s \in \left( 0, \frac{1}{2}
\right)$. We are not aware of any analogous theorem in the literature
which works at arbitrarily small fractional (but non-zero) regularities for
any Boltzmann equation; the proof relies on a novel interpolation
strategy which would be difficult to implement in the usual framework
of inhomogeneous Strichartz estimates. We also remark that the
theorem is optimal because $s = 0$ is scaling critical, so we
cannot expect a
 local theorem depending only on the size of the $L^2$
norm of the data.

\begin{theorem}
\label{thm:lwp}
Fix a number $s \in \left( 0, \frac{1}{2}\right)$;
then there exists a function 
$\lambda_s ( \cdot ) : \mathbb{R}^{\geq 0} \rightarrow
\mathbb{R}^{\geq 0}$ such that all the following is true:

Suppose $f_0 : \mathbb{R}^2 \times \mathbb{R}^2 \rightarrow
\mathbb{R}$ is a \emph{non-negative}, locally integrable function
such that
\begin{equation}
\left\Vert \left< v \right>^{\frac{1}{2}+}
\left< \nabla_x \right>^{\frac{1}{2}+} f_0 (x,v)
\right\Vert_{L^2 \left( \mathbb{R}^2 \times \mathbb{R}^2 \right)}
< \infty
\end{equation}
Then for some $T_0$ satisfying
\begin{equation}
T_0 > \lambda_s \left( 
\left\Vert \left< v \right>^s
\left< \nabla_x \right>^s f_0 
\right\Vert_{L^2 \left( \mathbb{R}^2 \times \mathbb{R}^2 \right)}
\right)
\end{equation}
 there exists a  non-negative
mild solution $f \in C 
\left( [0,T_0), L^2_{x,v} \right)$
of Boltzmann's equation
\begin{equation} \label{eq:MMboltzd2001}
\left( \partial_t + v \cdot \nabla_x \right) f (t,x,v) = Q(f,f)
\end{equation}
where 
$$
Q(f,f) = Q^+ (f,f) - Q^- (f,f),
$$
with $Q^+$ and $Q^-$ given respectively in \eqref{eqQ+} and \eqref{eqQ-},  
such that $f(0) = f_0$ and the following
bounds \eqref{eq:MRbd11},\eqref{eq:MRbd21},\eqref{eq:MRbd31}
hold for any $T \in \left( 0, T_0 \right)$:
\begin{equation}
\label{eq:MRbd11}
\left< v \right>^{\frac{1}{2}+} Q^+ (f,f) \in 
L^1_{t \in [0,T]} L^2_{x,v} 
\end{equation}
\begin{equation}
\label{eq:MRbd21}
\rho_f \in 
L^2_{t \in [0,T]} L^\infty_x \bigcap L^2_{t \in [0,T]} L^4_x 
\end{equation}
\begin{equation}
\label{eq:MRbd31}
\left< v \right>^{\frac{1}{2}+} f \in 
 L^\infty_{t \in [0,T]} L^2_{x,v} \bigcap L^\infty_{t \in [0,T]} L^4_x L^2_v 
\end{equation}
The solution $f(t)$ is unique in the class of all mild solutions,
with the same initial data, satisfying all the
bounds (\ref{eq:MRbd11},\ref{eq:MRbd21},\ref{eq:MRbd31})  for each $T \in (0,T_0)$.
In particular, any mild solution with data $f_0$
 satisfying (\ref{eq:MRbd11},\ref{eq:MRbd21},\ref{eq:MRbd31}) 
is automatically non-negative (since it is equal to $f$). 
\end{theorem}

We are not able to show that the $\frac{1}{2}+$ regularity
assumed at $t=0$ is propagated, but we expect this to be true
and state it is a conjecture:
\begin{conjecture}
In the notation of Theorem \ref{thm:lwp}, the local solution
$f(t)$ carries the regularity of the data up to time $T_0$.
More precisely, for any $T \in \left( 0, T_0 \right)$, there
holds
\begin{equation}
\left\Vert \left< v \right>^{\frac{1}{2}+}
\left< \nabla_x \right>^{\frac{1}{2}+}
f(t) \right\Vert_{L^\infty_{t \in [0,T]} L^2_{x,v}} < \infty
\end{equation}
\end{conjecture}

\begin{remark}
It is possible to show that the regularity is propagated for
a time that depends on the size of the $\frac{1}{2}+$ norm
at time $t=0$. The point of the conjecture is that the 
$\frac{1}{2}+$ regularity persists for a time depending on
a \emph{lower} regularity norm, namely the $s$ norm.
\end{remark}

\begin{remark}
In view of Theorem \ref{thm:lwp}, where the time of existence
depends on a norm which is very close to $L^2$, it is natural 
to ask whether it is possible to prove local well-posedness in
a space like $L^2$ or $L^2 \cap L^1$ (note that the $L^1$ norm is
conserved for Boltzmann's equation). Since $L^2$ is a critical norm
for the Boltzmann equation with constant collision kernel, the best we can
hope for is a local well-posedness time which depends on the profile
of the initial data. Unfortunately, so far we have not been able to extract
such a result using our method, though there is no obvious obstruction.
Several \emph{a priori} estimates are available in complete generality
for $L^2$ solutions on a short time interval
 (assuming that a certain spacetime integral is finite in which case it is
bounded quantatively),
and they are presented in Appendix \ref{app:apriori}.
\end{remark}

\section*{Acknowledgements}
\label{sec:acknowledgements}

T.C. gratefully acknowledges support by the NSF through grants DMS-1151414 (CAREER) and DMS-1716198.
R.D. gratefully acknowledges support from a postdoctoral fellowship
at the University of Texas at Austin. N.P. gratefully acknowledges support from NSF grant DMS-1516228.
\\

\section{Technical preliminary: The Kaniel-Shinbrot iteration}
\label{sec:techprelim}

%{\color{red}{
In this section, we present a brief review of the Kaniel-Shinbrot iteration method (see \cite{KS1984}) 
for proving existence of solutions for Boltzmann equations, and describe its typical use. Then we 
give a short preview of the new approach based on of the Kaniel-Shinbrot iteration method that we introduce in this paper. 
%}}

\subsection{The method of Kaniel and Shinbrot in a nutshell} 
The method of Kaniel and Shinbrot is based on three main steps:
\begin{enumerate}
\item Construct a pair of functions
 satisfying the so-called \emph{beginning condition}.
\item Develop  sequences of functions which act as \emph{barriers} (above and below) which converge
monotonically to upper and lower envelopes of a (hypothetical) true solution.
\item Prove that the upper and lower envelopes coincide, hence defining a solution to the Boltzmann
equation itself.
\end{enumerate}
We note that there is no claim of uniqueness in the Kaniel-Shinbrot iteration, though the third step
(convergence) is typically as hard to prove as uniqueness. Usually, one views Kaniel-Shinbrot as a
proof of \emph{existence by construction}, followed by a separate proof of uniqueness in a class of
solutions containing the Kaniel-Shinbrot solution.

We start with two functions $g_1, h_1$, which are supposed to be upper 
and lower bounds (respectively) for a true solution of Boltzmann's equation. 
The first iterates $g_2,h_2$ are defined by the formulas
\begin{equation}
\begin{aligned}
\left( \partial_t + v \cdot \nabla_x + \rho_{h_1} \right) g_2 & = Q^+ (g_1,g_1) \\
\left( \partial_t + v \cdot \nabla_x + \rho_{g_1} \right) h_2 & = Q^+ ( h_1,h_1) \\
g_2 (t=0) = h_2 (t=0) & = f_0
\end{aligned}
\end{equation}
Kaniel and Shinbrot \cite{KS1984} assume that
$g_1, h_1$ are chosen to guarantee the following inequalities
(for all times on the interval of interest):
\begin{equation}
\label{eq:KSBegin1}
0 \leq h_1 \leq h_2 \leq g_2 \leq g_1
\end{equation}
and this is the so-called \emph{beginning condition} of the Kaniel-Shinbrot iteration.

The beginning condition (\ref{eq:KSBegin1}) secured, Kaniel and Shinbrot define the rest of the
iteration (here $n \geq 2$):
\begin{equation}
\begin{aligned}
\left( \partial_t + v \cdot \nabla_x + \rho_{h_n} \right) g_{n+1} & = Q^+ (g_n,g_n) \\
\left( \partial_t + v \cdot \nabla_x + \rho_{g_n} \right) h_{n+1} & = Q^+ ( h_n,h_n) \\
g_{n+1} (t=0) = h_{n+1} (t=0) & = f_0
\end{aligned}
\end{equation}
They prove by induction that, as long as the beginning condition (\ref{eq:KSBegin1}) is satisfied,
the following inequalities hold for each $n$:
\begin{equation}
0 \leq h_n \leq h_{n+1} \leq g_{n+1} \leq g_n \leq g_1
\end{equation}
In other words, there is a sequence $h_n$ increasing from below and a decreasing sequence
$g_n$, all bounded above by the fixed function $g_1$. This allows us to apply monotone convergence
pointwise and conclude the existence (under mild regularity assumption) of limits
$g,h$ with $0 \leq h \leq g \leq g_1$ satisfying the following equations:
\begin{equation}
\label{eq:KSsystem}
\begin{aligned}
\left( \partial_t + v \cdot \nabla_x + \rho_h \right) g & = Q^+ (g,g) \\
\left( \partial_t + v \cdot \nabla_x + \rho_g \right) h & = Q^+ (h,h) \\
g(t=0) = h (t=0) & = f_0
\end{aligned}
\end{equation}
This system is satisfied, of course, if $g=h=f$ is the (supposedly unique) solution of Boltzmann's
equation; hence, if the system has a unique solution $(g,h=g)$, then that solution is exactly the unique
solution of Boltzmann's equation. Thus the question of convergence of the Kaniel-Shinbrot scheme
is closely related to a uniqueness question.

\begin{remark} 
The method  of Kaniel-Shinbrot  \cite{KS1984} is applicable to the Boltzmann equation under an angular cutoff condition (Grad cut-off). 
We note that the Boltzmann equation with constant collision kernel satisfies Grad's cut-off 
(it is enough to note that $Q^+$ and $Q^- = f \rho_f$ each make sense taken separately, if $f$ is nice enough).
\end{remark}

Usually we do not prove that the system (\ref{eq:KSsystem}) has a unique solution, since this requires
more effort than is actually necessary. In fact, if we can only prove that $g \equiv h$ (for instance by
a Gronwall argument), then the function $g$ (or equivalently $h$) is itself a solution of Boltzmann's
equation, but there is no guarantee of uniqueness. In that case, uniqueness is usually proven by
an independent argument. This is indeed the strategy employed in the present work.

The Kaniel-Shinbrot iteration has been applied to ``large'' initial conditions
which are ``squeezed'' between two nearby Maxwellian distributions. This was
first achieved by Toscani \cite{To1988}, using a clever choice of (locally Maxwellian)
functions $g_1, h_1$ satisfying the beginning condition of Kaniel and Shinbrot.
The approach was later adapted to soft potentials (with Grad cut-off) by Alonso and Gamba.
\cite{Alonso2009}. Unfortunately, it is not clear to us how to adapt Toscani's proof to the scaling-critical
($L^2_{x,v}$) setting; the lower envelope $h_1$ should presumably be a Maxwellian, but
the upper envelope $g_1$ must be some $L^2$ function which tracks the singularities of
the data. There does not appear to be an obvious choice for upper envelope $g_1$ (satisfying
the beginning condition) when the
data $f_0$ is not small.

\subsection{The method of Kaniel and Shinbrot revisited}

The beginning conditions for Kaniel-Shinbrot is traditionally satisfied
by taking $g_1$ to be a Maxwellian
distribution which bounds $f_0$ from above, with $h_1 \equiv 0$; or, by ``squeezing''
$f$ between two Maxwellians $g_1,h_1$
 which need not be small (but must be close to each other).
However these ideas
do not work in our setting
 since $f_0$ does not need to be bounded above pointwise; indeed, the only 
 \emph{quantitative}
 estimate we are allowed is that $f_0 \in L^2_{x,v}$.

Instead, our strategy is to solve the \textbf{gain-term-only}
 Boltzmann equation using
a bilinear estimate, and subsequently apply the Kaniel-Shinbrot iteration
to the solution of the gain-only equation in order to develop a solution
of the full Boltzmann equation. Thus, for us, $h_1$ is identically zero
and $g_1$ satisfies
$$
\left( \partial_t + v \cdot \nabla_x \right) g_1 =
Q^+ (g_1,g_1)
$$
with initial data $g_1 (t=0) = f_0$. 
It would seem that the Kaniel-Shinbrot iteration gains us nothing, since
we are initiating the iteration with the solution to a nonlinear equation.
However, it turns out that at critical regularity, the gain-only equation
is easier to solve than the full Boltzmann equation, as was observed
by D. Arsenio, \cite{Ar2011} In particular, the gain term $Q^+$ satisfies
bilinear estimates which are not available for the loss term.

\begin{remark}
The suggestion to apply Kaniel-Shinbrot at low regularities is due to
Arsenio in \cite{Ar2011}, who discussed the possibility in the introduction.
However, Arsenio did \emph{not} implement the Kaniel-Shinbrot iteration,
instead relying on a compactness argument,
apparently due to the lack of uniqueness in his formulation.
We have overcome this limitation by propagating some auxiliary
 regularity and
moment bounds for the gain-only equation, to the point that a uniqueness 
theorem for the full Boltzmann equation is indeed available, thereby
allowing us to prove convergence of the Kaniel-Shinbrot iteration.
\end{remark}

\section{An Abstract Well-Posedness Theorem}
\label{sec:LWPthm}

%{\color{red}{
In this section we present an abstract well-posedness theorem, 
which is inspired by ``space-time" methods 
that are often used in the context of dispersive PDEs. 
%}}

Let $\mathcal{H}$ be a separable Hilbert space over $\mathbb{R}$ or $\mathbb{C}$,
and let $k \geq 2$ be an integer.
  Suppose we have a map
\begin{equation}
\mathcal{A} :  \mathcal{H}^{\times k} \rightarrow L^1 \left( \mathbb{R}, \mathcal{H} \right)
\end{equation}
such that $\mathcal{A}$ is linear with respect to each factor
of $\mathcal{H}$ (keeping the others fixed), and an estimate of the following
form holds:
\begin{equation}
\label{eq:Abd}
\left\Vert \mathcal{A} (x_1,\dots,x_k) (t) \right\Vert_{L^1_t \mathcal{H}}
\leq C_0 \prod_{j=1}^k \left\Vert x_j \right\Vert_{\mathcal{H}}
\qquad \qquad x_1,\dots,x_k \in \mathcal{H}
\end{equation}
We will say that $\mathcal{A}$ is a bounded $k$-linear
map $\mathcal{H}^{\times k} \rightarrow L^1_t \mathcal{H}$, and we will generally
write it equivalently as $\mathcal{A} (t,x_1,\dots,x_k)$.
We are interested in properly defining, and then solving, the equation 
\begin{equation}
\label{eq:dxdtA}
\frac{dx}{dt} = \mathcal{A} (t,x(t),\dots,x(t))
\end{equation}
 when
$x(0)=x_0 \in \mathcal{H}$ is a given element of $\mathcal{H}$ with small norm.
As we will see, the bound (\ref{eq:Abd}) along with the $k$-linearity is
\emph{sufficient} to solve (\ref{eq:dxdtA}) globally in time for small data; scattering will also
follow automatically, in the sense that $\lim_{t \rightarrow +\infty} x(t)$ exists in
the norm topology of $\mathcal{H}$. We will find that
$x(t) \in W^{1,1} \left( (0,T), \mathcal{H} \right)$ for any $T>0$, so equation
(\ref{eq:dxdtA}) holds in a strong sense. The theorem, along with its proof,
 is inspired by certain methods due to
Klainerman and Machedon for solving dispersive PDE. \cite{KM1993,KM2008}

\begin{remark}
In the complex case, it
 is acceptable for $\mathcal{A}$ to be conjugate linear with respect to some
or all entries; the changes to the proof are trivial so we only discuss the linear case.
\end{remark}

Note that \emph{a priori} we can only evaluate $\mathcal{A} (t,x_1,\dots,x_k)$ for
$\textnormal{a.e. } t$ given \emph{fixed} elements $x_1,\dots,x_k$
of $\mathcal{H}$; in particular, the exceptional set in $t$ may depend on
$x_1,\dots,x_k$. However, if $x(t)$ is a $C^1$ curve, then near any given time
$t_0$, $x$ is almost a constant. This observation motivates the following result:
\begin{lemma}
\label{lem:abswplem}
Let $\mathcal{H}$ be a separable Hilbert space and 
suppose $\mathcal{A}: \mathcal{H}^{\times k} \rightarrow L^1_t \mathcal{H}$ is
a   mapping which is linear or conjugate linear 
in each entry; furthermore, suppose that
the estimate (\ref{eq:Abd}) holds. Then for any $T \in (0,\infty)$
 there exists a unique  $k$-linear map
\begin{equation}
\tilde{\mathcal{A}} : \Big( W^{1,1} \left( (0,T), \mathcal{H} \right) \Big)^{\times k}
\rightarrow L^1 \left( (0,T) , \mathcal{H} \right)
\end{equation}
which satisfies
\begin{equation}
\tilde{\mathcal{A}} \left(t, f_1  x_1, \dots, f_k  x_k \right)
= \left( \prod_{j=1}^k f_j (t) \right) \mathcal{A} ( t,x_1,\dots,x_k )
\end{equation}
for any $x_1,\dots,x_k \in \mathcal{H}$ and any smooth bounded real-valued
functions $f_1,\dots,f_k$ on $[0,T]$; here, $f_j x_j$ denotes the
function $(f_j x_j) (t) = f_j (t) x_j$. The following estimate
holds as well:
\begin{equation}
\begin{aligned}
& \left\Vert \tilde{\mathcal{A}} \left( t, x_1 (\cdot), \dots, x_k (\cdot) \right)
\right\Vert_{L^1_{t \in (0,T)} \mathcal{H}} \leq \\
& \qquad \qquad \qquad \leq  (1+k) C_0 \prod_{j=1}^k \left(
\left\Vert x_j (t) \right\Vert_{L^\infty_{t \in (0,T)} \mathcal{H}}
+ \left\Vert \frac{dx_j}{dt} \right\Vert_{L^1_{t \in (0,T)} \mathcal{H}} \right)
\end{aligned}
\end{equation}
for any $x_1 (\cdot), \dots, x_j (\cdot) \in W^{1,1} \left( (0,T), \mathcal{H} \right)$.
\end{lemma}
\begin{proof} \emph{(Sketch.)}
It is possible to prove this result by expanding each $x_j$ via Duhamel's formula
and using $k$-linearity. However, it is much easier to simply differentiate $\mathcal{A}$
directly as follows, denoting $\zeta_j = \frac{dx_j}{dt}$:
\begin{equation}
\begin{aligned}
& \frac{\partial}{\partial \sigma}
\mathcal{A} \left( t, x_1 (\sigma), \dots, x_k (\sigma) \right) = 
\mathcal{A} \left( t, \zeta_1 (\sigma), x_2 (\sigma), \dots, x_k (\sigma) \right) + \\
& \qquad \qquad \qquad \qquad \qquad \qquad
 + \dots + \mathcal{A} \left( t, x_1 (\sigma), \dots, x_{k-1} (\sigma),
\zeta_k (\sigma) \right)
\end{aligned}
\end{equation}
We can integrate both sides in $\sigma$ from $0$ to $t$, in order to relate the diagonal
$\sigma = t$ in terms of quantities off the diagonal:
\begin{equation}
\begin{aligned}
& \mathcal{A} \left( t, x_1 (t), \dots, x_k (t) \right) = 
\mathcal{A} \left( t, x_1 (0), \dots, x_k (0)\right) + \\
& \qquad \qquad \qquad
 + \int_0^t \mathcal{A} \left( t, \zeta_1 (\sigma), x_2 ( \sigma ), \dots, x_k ( \sigma ) \right) d\sigma + \\
& \qquad \qquad \qquad
+ \dots + \int_0^t \mathcal{A} \left( t, x_1 (\sigma), \dots, x_{k-1} (\sigma),
\zeta_k (\sigma) \right) d\sigma
\end{aligned}
\end{equation}
The first term is obviously bounded in $L^1_t \mathcal{H}$ due to (\ref{eq:Abd}). We demonstrate
how to estimate the first integral term (the others are treated similarly):
\begin{equation*}
\begin{aligned}
& \left\Vert \int_0^t \mathcal{A} \left( t, \zeta_1 (\sigma), x_2 (\sigma), \dots,
x_k (\sigma) \right) d\sigma \right\Vert_{L^1_{t \in (0,T)} \mathcal{H}} \leq \\
& \qquad \qquad \qquad
 \leq \left\Vert \int_0^t \left\Vert \mathcal{A} \left( t, \zeta_1 (\sigma), x_2 (\sigma), \dots,
x_k ( \sigma ) \right) \right\Vert_{\mathcal{H}} d\sigma \right\Vert_{L^1_{t \in (0,T)}} \\
& \qquad \qquad \qquad
 \leq \left\Vert \int_0^T \left\Vert \mathcal{A} \left( t, \zeta_1 (\sigma), x_2 (\sigma), \dots,
x_k ( \sigma ) \right) \right\Vert_{\mathcal{H}} d\sigma \right\Vert_{L^1_{t \in (0,T)}} \\
& \qquad \qquad \qquad
 \leq \int_0^T \left\Vert \mathcal{A} \left( t, \zeta_1 (\sigma), x_2 (\sigma), \dots,
x_k ( \sigma ) \right) \right\Vert_{L^1_{t \in (0,T)} \mathcal{H}} d\sigma  \\
& \qquad \qquad \qquad \leq
C_0 \int_0^T \left\Vert \zeta_1 ( \sigma ) \right\Vert_{\mathcal{H}}
\left\Vert x_2 ( \sigma ) \right\Vert_{\mathcal{H}} \ldots
\left\Vert x_k ( \sigma ) \right\Vert_{\mathcal{H}} d\sigma \\
& \qquad \qquad \qquad \leq
C_0 \left\Vert \zeta_1 \right\Vert_{L^1_{t \in (0,T)} \mathcal{H}}
\left\Vert x_2 \right\Vert_{L^\infty_{t \in (0,T)} \mathcal{H}} \ldots
\left\Vert x_k \right\Vert_{L^\infty_{t \in (0,T)} \mathcal{H}}
\end{aligned}
\end{equation*}
Gathering terms together, we are able to conclude.
\end{proof}

\begin{remark}
The map $\tilde{\mathcal{A}}$ clearly extends $\mathcal{A}$, in the sense that we can
view any $x_0 \in \mathcal{H}$ as a function of time by calling it a constant function. 
Since there is no ambiguity, we will refer to both operators using the common
notation $\mathcal{A}$.
\end{remark}

\begin{theorem}
\label{thm:abswp}
Let $\mathcal{H}$ be a separable Hilbert space, fix an integer $k \geq 2$, and
let $\mathcal{A} : \mathcal{H}^{\times k} \rightarrow L^1 \left( \mathbb{R}, \mathcal{H} \right)$
be a mapping which is linear or conjugate linear in each entry, and satisfies the estimate
\begin{equation}
\left\Vert \mathcal{A} (t, x_1, \dots, x_k) \right\Vert_{L^1_t \mathcal{H}}
\leq C_0 \prod_{j=1}^k \left\Vert x_j \right\Vert_{\mathcal{H}} \qquad \qquad
x_1, \dots, x_k \in \mathcal{H}
\end{equation}
Then, defining
\begin{equation*}
M = \left( \frac{1}{4^k k (1+k) C_0} \right)^{1 / (k-1)}
\end{equation*}
we find that for any $x_0 \in \mathcal{H}$ with $\left\Vert x_0 \right\Vert_{\mathcal{H}} \leq M$
there exists a global solution $x(t) \in \bigcap_{T > 0}
W^{1,1} \left( (0,T), \mathcal{H} \right)$ of the integral equation
\begin{equation}
x(t) = x_0 + \int_0^t \mathcal{A} \left( \sigma, x(\sigma), \dots, x(\sigma) \right) d\sigma
\end{equation}
and this solution is unique in the regularity class 
$\bigcap_{T>0} W^{1,1}_{t \in [0,T]} \mathcal{H}$. 
Moreover, for the solutions arising in this way, the following estimate holds:
\begin{equation}
\label{eq:absest}
\left\Vert x (t) \right\Vert^k_{L^\infty_{t \geq 0} \mathcal{H}}
+ \left\Vert \mathcal{A} \left( t, x(t), \dots, x(t) \right)
\right\Vert_{L^1_{t \geq 0} \mathcal{H}} \leq
C_1 \left\Vert x_0 \right\Vert^k_{\mathcal{H}}
\end{equation}
for some constant $C_1$ depending only on $k$ and $C_0$. In particular, (\ref{eq:absest})
implies that $\lim_{t \rightarrow +\infty} x(t)$ exists strongly in $\mathcal{H}$
(i.e., the solution scatters).
\end{theorem}
\begin{proof} \emph{(Sketch)}
We will use the following norm on $W^{1,1} \left( (0,T), \mathcal{H} \right)$:
\begin{equation}
\left\Vert x ( \cdot ) \right\Vert_{\mathcal{W}^{1,1}_{t \in (0,T)} \mathcal{H}} =
\left\Vert x ( t ) \right\Vert_{L^\infty_{t \in (0,T)} \mathcal{H}}
+ \left\Vert \frac{dx}{dt} ( t ) \right\Vert_{L^1_{t \in (0,T)} \mathcal{H}}
\end{equation}
This norm is equivalent to the usual norm on $W^{1,1}$ for fixed
finite $T$ by Sobolev
embedding, but exhibits better
scaling properties in this context for large  $T$.

Define the map $\mathfrak{F} : W^{1,1} \left( (0,T), \mathcal{H} \right) 
\rightarrow W^{1,1} \left( (0,T), \mathcal{H} \right)$ by the formula
\begin{equation}
\left[ \mathfrak{F} (x(\cdot)) \right] (t) = x_0 +
\int_0^t \mathcal{A} \left( \sigma, x(\sigma), \dots, x(\sigma) \right) d\sigma
\end{equation}
This is well-defined by Lemma \ref{lem:abswplem}.

Using Lemma \ref{lem:abswplem}, we easily derive the following boundedness and
locally Lipschitz estimates:

\begin{equation*}
\left\Vert \left[ \mathfrak{F} ( x ( \cdot ) ) 
\right] (t) \right\Vert_{\mathcal{W}^{1,1}_{t \in (0,T)} \mathcal{H}}
\leq \left\Vert x_0 \right\Vert_{\mathcal{H}} + 2 (1+k) C_0 
\left\Vert x ( \cdot ) \right\Vert^k_{\mathcal{W}^{1,1}_{t \in (0,T)} \mathcal{H}}
\end{equation*}

and

\begin{equation*}
\begin{aligned}
& \left\Vert \left[ \mathfrak{F} ( x_2 ( \cdot ) ) 
- \mathfrak{F} ( x_1 ( \cdot ) )
\right] (t) \right\Vert_{\mathcal{W}^{1,1}_{t \in (0,T)} \mathcal{H}} \leq \\
& \leq 2 k (1+k) C_0 \left( \sum_{i=1,2}
\left\Vert x_i ( t ) \right\Vert_{\mathcal{W}^{1,1}_{t \in (0,T)} \mathcal{H}}
\right)^{k-1}
\left\Vert x_2 ( t ) - x_1 ( t ) \right\Vert_{\mathcal{W}^{1,1}_{t \in (0,T)} \mathcal{H}}
\end{aligned}
\end{equation*}

Therefore, defining the closed ball
\begin{equation*}
\mathfrak{B} = \left\{ x ( \cdot ) \in W^{1,1} \left( (0,T), \mathcal{H} \right)
\; \left| \; \left\Vert x ( t ) \right\Vert_{\mathcal{W}^{1,1}_{t \in (0,T)} \mathcal{H}}
\leq 2 M \right. \right\}
\end{equation*}
with $M$ as in the statement of the theorem,
we find that $\mathfrak{F} \mathfrak{B} \subset \mathfrak{B}$ and $\mathfrak{F}$ is a
strict contraction of $\mathfrak{B}$. Hence, we may apply the Banach fixed point theorem
and thereby extract a unique fixed point of $\mathfrak{F}$ within $\mathfrak{B}$.
\end{proof}

\begin{theorem}
\label{thm:absScattering}
Let $\mathcal{H}, \mathcal{A}, M, k, C_0$ be as in the statement of Theorem \ref{thm:abswp}.
Consider the integral equation
\begin{equation}
x (t) = x_0 + \int_0^t \mathcal{A} \left( \sigma, x ( \sigma ), \dots, x ( \sigma ) \right) d\sigma
\end{equation}
with unique solutions $x \in \bigcap_{T > 0} W^{1,1} \left( (-T,T),
 \mathcal{H} \right)$, $x (0) = x_0$,
as given by Theorem \ref{thm:abswp} for any $x_0 \in \mathcal{H}$ such that
$\left\Vert x_0 \right\Vert_{\mathcal{H}} \leq M$. Define the map
\begin{equation}
\mathcal{S} : \overline{B_M^{\mathcal{H}} (0)}
\rightarrow \bigcap_{T > 0} W^{1,1} \left( (-T,T),
 \mathcal{H} \right)
\end{equation} 
such that 
\begin{equation}
\left[ \mathcal{S} (x_0) \right] ( t )
 = x_0 + \int_0^t \mathcal{A} \left( \sigma, 
\left[ \mathcal{S} ( x_0 ) \right] ( \sigma ), \dots, 
\left[ \mathcal{S} ( x_0 ) \right] ( \sigma ) \right) d\sigma
\end{equation}
The map $\mathcal{S}$ is well-defined  by the statement and proof of Theorem \ref{thm:abswp}.
For any $r \in (0, M)$ define the maps $\mathfrak{S}^{+}_r$, $\mathfrak{S}^{-}_r$,
\begin{equation}
\mathfrak{S}^{\pm}_r : B_r^{\mathcal{H}} (0) \rightarrow \mathcal{H}
\end{equation}
\begin{equation}
\mathfrak{S}^{\pm}_r (x_0) = \lim_{t \rightarrow \pm \infty} \left[ \mathcal{S} ( x_0 ) \right] ( t )
\end{equation}
where the limit is taken in the norm topology of $\mathcal{H}$; this is possible by
Theorem \ref{thm:abswp}. Let
$\mathfrak{U}^{\pm}_r$ denote the image of $\mathfrak{S}^{\pm}_r$, and note that
$0 \in \mathfrak{U}^{+}_r \bigcap \mathfrak{U}^{-}_r$.

Then, there exists $r_0 = r_0 (k, C_0) > 0$ such that 
if $0 < r < r_0$ then $\mathfrak{U}^{+}_r , \mathfrak{U}^{-}_r$ are each open in the norm
topology of $\mathcal{H}$, and $\mathfrak{S}^{+}_r , \mathfrak{S}^{-}_r$ are each
bijective and bi-Lipschitz. As a consequence, the composite maps
\begin{equation}
\mathfrak{S}_r^{+} \circ \left( \mathfrak{S}_r^{-} \right)^{-1} \; : \;
\mathfrak{U}_r^{-} \rightarrow \mathfrak{U}_r^{+}
\end{equation} 
\begin{equation}
\mathfrak{S}_r^{-} \circ \left( \mathfrak{S}_r^{+} \right)^{-1} \; : \;
\mathfrak{U}_r^{+} \rightarrow \mathfrak{U}_r^{-}
\end{equation} 
are bijective and bi-Lipschitz.
\end{theorem}
\begin{proof} \emph{(Sketch.)} The key estimate states that $x_0$ expresses a Lipschitz depencence on
$x_{+\infty} = \lim_{t \rightarrow +\infty} \left[ \mathcal{S} (x_0) \right] (t)$,
at least within sufficiently small neighborhoods of $0 \in \mathcal{H}$.

Let $T>0$ and consider the solution $x(t) = \left[ \mathcal{S} (x_0) \right] (t)$ for
$t \in (0,T)$. As long as $\left\Vert x_0 \right\Vert_{\mathcal{H}}$
 is sufficiently small (depending only on $k, C_0$), we can guarantee that
$\left\Vert x(T) \right\Vert_{\mathcal{H}} < M$, so that Theorem \ref{thm:abswp} can be applied
\emph{backwards} in time with data $x(T)$. Considering two solutions 
$x(t) = \left[ \mathcal{S} (x_0) \right] (t)$, $y(t) = \left[ \mathcal{S} (y_0) \right] (t)$,
we can apply this procedure to each of them and derive the following identity:
\begin{equation*}
\begin{aligned}
& x(t) - y(t) = x(T) - y(T) - \int_t^T \mathcal{A} \left( \sigma, x(\sigma), \dots, x(\sigma) \right) d\sigma
+ \\
& \qquad \qquad \qquad \qquad \qquad \qquad \qquad
 + \int_t^T \mathcal{A} \left( \sigma, y(\sigma), \dots, y(\sigma) \right) d\sigma
\end{aligned}
\end{equation*}
Hence, using the norms defined in the proof of Theorem \ref{thm:abswp}, along with Lemma
\ref{lem:abswplem}, we have:
\begin{equation}
\begin{aligned}
& \left\Vert x(t) - y(t) \right\Vert_{\mathcal{W}^{1,1} \left( (0,T), \mathcal{H} \right)}
\leq \left\Vert x(T) - y(T) \right\Vert_{\mathcal{H}} + \\
& + 2 k (1+k) C_0
\left(\sum_{z \in \left\{ x, y \right\} }
\left\Vert z(t) \right\Vert_{\mathcal{W}^{1,1} \left( (0,T), \mathcal{H} \right)} \right)^{k-1}
\left\Vert x(t) - y(t) \right\Vert_{\mathcal{W}^{1,1} \left( (0,T), \mathcal{H} \right)}
\end{aligned}
\end{equation}
In view of the statement and proof of Theorem \ref{thm:abswp}, under the above assumptions
we can deduce the quantitative estimate,
\begin{equation}
\left\Vert x(t) - y(t) \right\Vert_{\mathcal{W}^{1,1} \left( (0,T), \mathcal{H} \right)}
\leq 2 \left\Vert x(T) - y(T) \right\Vert_{\mathcal{H}}
\end{equation}
as long as $\left\Vert x_0 \right\Vert_{\mathcal{H}}, \left\Vert y_0 \right\Vert_{\mathcal{H}}$ are
sufficiently small (depending on only $k,C_0$). This immediately implies
\begin{equation}
\left\Vert x_0 - y_0 \right\Vert_{\mathcal{H}} \leq 2
\left\Vert x(T) - y(T) \right\Vert_{\mathcal{H}}
\end{equation}
Taking strong limits in $\mathcal{H}$ as $T \rightarrow +\infty$, we obtain
\begin{equation}
\left\Vert x_0 - y_0 \right\Vert_{\mathcal{H}}
\leq 2 \left\Vert x_{+\infty} - y_{+\infty} \right\Vert_{\mathcal{H}}
\end{equation}
which is the desired Lipschitz estimate.

The last claim is the following: for all $q \in ( 0 , M) $,
$\mathfrak{S}^{\pm}_q \left[ B_q^{\mathcal{H}} (0) \right]$ contains a neighborhood of
$0 \in \mathcal{H}$. This is routine to check by adapting the proof of Theorem \ref{thm:abswp}.
\end{proof}

\begin{remark}
\label{rem:locwp}
If, instead of the ``critical'' estimate (\ref{eq:Abd}), $\mathcal{A}$ satisfies a
``subcritical'' estimate of the form
\begin{equation}
\label{eq:spacetimeLp}
\left\Vert \mathcal{A} (t,x_1, \dots, x_k)\right\Vert_{L^p_t \mathcal{H}} \leq
\tilde{C} \prod_{j=1}^k \left\Vert x_j \right\Vert_{\mathcal{H}} \qquad \qquad
x_1,\dots,x_k \in \mathcal{H} 
\end{equation}
for some $p > 1$,
then we can always convert $\mathcal{A}$ into a form suitable for the application of
Theorem \ref{thm:abswp} by multiplying $\mathcal{A}$ by a bump function in time which
is equal to one on an interval $[0,T]$. In that case, the constant $C_0$ in the
theorem would be $C_0 \approx \tilde{C} T^{\frac{1}{p^\prime}}$, so that the allowable size of the
data tends to infinity as $T$ tends to zero. Hence, Theorem \ref{thm:abswp} can be used to
prove a wide range of local well-posedness results in the large for the
strictly scaling-subcritical case.
\end{remark}

\begin{remark}
There is a version of Theorem \ref{thm:abswp} when $k=1$, i.e. linear equations, but only
if $C_0 < \frac{1}{4}$.
\end{remark}

\begin{remark}
Local well-posedness for arbitrary $x_0 \in \mathcal{H}$ is \emph{\textbf{not}} recovered under the
sole assumption (\ref{eq:Abd}); this is because the equation for $\tilde{x} (t) = x(t) - x_0$
contains linear terms, and we can only solve the linear case when $C_0 < \frac{1}{4}$ per the
previous remark. (The forcing term $\mathcal{A} (t,x_0,\dots,x_0)$
 can always be made negligible, for fixed $x_0$, by
localizing to a small time interval depending on $x_0$.) 
If, for any $T > 0$ and any $x_0 \in \mathcal{H}$,
 estimates of the following form are satisfied for open intervals 
$I \subset \left( -T, T \right)$,
\begin{equation*}
\limsup_{\delta \rightarrow 0^+}
\sup_{I \subset \left( -T, T \right) \; : \; |I| \leq \delta}
\sup_{y_0 \in \mathcal{H} \backslash \left\{ 0 \right\}}
\frac{1}{\left\Vert y_0 \right\Vert_{\mathcal{H}}}
\left\Vert \mathcal{A} \left( t, y_0, x_0, \dots, x_0 \right)
\right\Vert_{L^1_{t \in I} \mathcal{H}} = 0
\end{equation*}
(and similarly for the other entries of $\mathcal{A}$), then large data LWP can be recovered in the
limited sense that the time of existence depends on $x_0 \in \mathcal{H}$ instead
of $\left\Vert x_0 \right\Vert_{\mathcal{H}}$.
\end{remark}

\section{Example: Cubic NLS in $d=2$}
\label{sec:cubNLSd2}

%{\color{red}{
In this section we illustrate how Theorem  \ref{thm:abswp}
can be used to recover small data global well-posedness and scattering for 
the $L^2$ critical nonlinear Schr\"{o}dinger equation in spatial dimension $d=2$. 
Furthermore, we illustrate an approach to study propagation of regularity for the same equation.
Although these results themselves are well known, we illustrate how they can be recovered using the 
tools of Section \ref{sec:LWPthm}. This will form a footprint for our study of the Boltzmann equation in subsequent sections. 
%}}

Consider the  nonlinear Schr\"{o}dinger equation (NLS)
\begin{equation}
\label{eq:cbnlsd2}
\left( i \partial_t + \Delta \right) \varphi = |\varphi|^2 \varphi \qquad \quad
\varphi (t,x) : \mathbb{R} \times \mathbb{R}^2 \rightarrow \mathbb{C}
\end{equation}
where $\Delta \equiv \Delta_x$ and $\varphi (0, x) = \varphi_0 (x) \in L^2 (\mathbb{R}^2 )$. The
nonlinearity can be written $\varphi \overline{\varphi} \varphi$, so it is either linear
or conjugate linear in each entry.

\subsection{Small data global existence and scattering} 
We wish to solve this equation for \emph{small data} $\varphi_0 (x)$
 in the scaling-critical space $L^2 (\mathbb{R}^2)$. We pont out that the method
as formulated in the statement of Theorem \ref{thm:abswp}
 yields \emph{no conclusion} for (\ref{eq:cbnlsd2})
given initial data outside a small ball of the origin in $L^2$;
 this is expected due to the fact that
(\ref{eq:cbnlsd2}) is $L^2$-critical with respect to scaling.

We impose the unitary change of variables
\begin{equation}
\psi (t) = e^{- i t \Delta} \varphi (t) 
\end{equation}
which implies $\psi (0) = \varphi (0) = \varphi_0$ and
\begin{equation}
\partial_t \psi (t) = -i e^{-i t \Delta} g \left( e^{i t \Delta} \psi \right)
\end{equation}
where $g (u) = u \overline{u} u$. Let us define the more general nonlinearity
$$g (u,v,w) = u \overline{v} w$$ 
and estimate for given $u_0, v_0, w_0 \in L^2 \left( \mathbb{R}^2 \right)$:
\begin{equation}
\begin{aligned}
& \left\Vert e^{-i t \Delta} g \left( e^{i t \Delta} u_0,
e^{i t \Delta} v_0, e^{i t \Delta} w_0 \right) \right\Vert_{L^1_t L^2_x} = 
\left\Vert  g \left( e^{i t \Delta} u_0,
e^{i t \Delta} v_0, e^{i t \Delta} w_0 \right) \right\Vert_{L^1_t L^2_x} \\
& \qquad \qquad =  \left\Vert  \left( e^{i t \Delta} u_0 \right)
\left( \overline{e^{i t \Delta} v_0} \right)
\left(  e^{i t \Delta} w_0 \right) \right\Vert_{L^1_t L^2_x} \\
& \qquad \qquad \leq \left\Vert e^{i t \Delta} u_0 \right\Vert_{L^3_t L^6_x}
\left\Vert e^{i t \Delta} v_0 \right\Vert_{L^3_t L^6_x}
\left\Vert e^{i t \Delta} w_0 \right\Vert_{L^3_t L^6_x} \\
& \qquad \qquad \leq C \left\Vert u_0 \right\Vert_{L^2_x} \left\Vert v_0 \right\Vert_{L^2_x}
\left\Vert w_0 \right\Vert_{L^2_x}
\end{aligned}
\end{equation} 
We have used unitarity, the H{\" o}lder, and the Strichartz estimates, in that order. In other words, we have shown:
\begin{equation}
\label{eq:cubNLSd2L2}
\left\Vert  e^{-i t \Delta} g \left( e^{i t \Delta} u_0,
e^{i t \Delta} v_0, e^{i t \Delta} w_0 \right) \right\Vert_{L^1_t L^2_x} \leq
C \left\Vert u_0 \right\Vert_{L^2_x} \left\Vert v_0 \right\Vert_{L^2_x}
\left\Vert w_0 \right\Vert_{L^2_x}
\end{equation}
Applying Theorem \ref{thm:abswp} with
\begin{equation}
\label{eq:cubNLSd2A}
\mathcal{A} (t, u_0, v_0, w_0) = 
-i e^{-i t \Delta} g \left( e^{i t \Delta} u_0,
e^{i t \Delta} v_0, e^{i t \Delta} w_0 \right)
\end{equation}
 we find that solutions of (\ref{eq:cbnlsd2}) are globally well-posed and scatter, as long as the
data $\varphi_0 \in L^2 \left( \mathbb{R}^2 \right)$ has sufficiently small norm.
 Theorem \ref{thm:abswp} guarantees that, at the very least,
 uniqueness of \emph{small} solutions holds within the class of \emph{all} mild solutions 
 satisfying the bound $g(\varphi) \in L^1_{t \in [0,T]} L^2_x$;
 this uniqueness criterion can be equivalently written
$\varphi \in L^3_{t \in [0,T]} L^6_x$ by definition of $g$.

\begin{theorem}
There exists a number $\eta > 0$ such that, for any
$\varphi_0 \in L^2 \left( \mathbb{R}^2 \right)$ satisfying
$$
\left\Vert \varphi_0 \right\Vert_{L^2 \left( \mathbb{R}^2 \right)} < \eta
$$
it follows that equation (\ref{eq:cbnlsd2}) has a global solution which scatters in
$L^2 \left( \mathbb{R}^2 \right)$. The solution is unique in the class of
all $L^2$ mild solutions for which $\varphi \in L^3_{t,\textnormal{loc}} L^6_x$.
\end{theorem}

\begin{remark}
It is crucial to remember that the space $W^{1,1}$ (in time) appearing in Theorem \ref{thm:abswp} is
\emph{not} the usual Sobolev norm of the solution.
This is because we only have $W^{1,1}$ \emph{after} intertwining with the
free evolution. For this reason, to avoid confusion,
in practice
it is often better to use unitarity in order to state the uniqueness criterion in terms of an equivalent
estimate on the nonlinearity, cf. (\ref{eq:absest}).
\end{remark}

\subsection{Regularity.}
\label{ssec:cubNLSd2reg}

Regularity is a subtle question because it hides two \emph{separate} questions. 

\begin{itemize} 
\item The first, which is
easy to answer, is whether any $\varphi_0 \in H^1 \left( \mathbb{R}^2 \right)$, say, yields a global
solution when the $H^1$ norm is small enough. The answer is \emph{yes} because, by Leibniz' rule
and standard commutation formulae,
and $\mathcal{A}$ as in (\ref{eq:cubNLSd2A}), we have
\begin{equation}
\label{eq:cubNLSd2H1}
\left\Vert \mathcal{A} \left( t, u_0, v_0, w_0 \right) \right\Vert_{L^1_t H^1_x}
\leq \tilde{C} \left\Vert u_0 \right\Vert_{H^1_x}
\left\Vert v_0 \right\Vert_{H^1_x} \left\Vert w_0 \right\Vert_{H^1_x}
\end{equation}
Now as long as $\left\Vert
\varphi_0 \right\Vert_{H^1}$ is smaller than some number which depends explicitly on
$\tilde{C}$, the cubic NLS will have a global solution which scatters in 
$H^1 \left( \mathbb{R}^2 \right)$, as a
direct consequence of Theorem \ref{thm:abswp} and Theorem \ref{thm:absScattering}.

\item The second, more difficult, question is whether $H^1$ regularity is propagated for
smooth solutions which are \emph{only small in $L^2$}. This can be seen as a persistence of regularity
question, since we know that any small $L^2$ data will lead to a global $L^2$ solution.
The answer, perhaps surprisingly, is \emph{yes}, as we now show.

\end{itemize}

The key is to introduce a new norm, $H_\varepsilon^1$, parameterized by $\varepsilon \in (0,1]$,
which is equivalent to $H^1$ up to an $\varepsilon$-dependent factor, but tends to the $L^2$ norm
as $\varepsilon \rightarrow 0^+$. The goal is to prove a bound
of the form
\begin{equation}
\label{eq:cubNLSd2Heps}
\left\Vert \mathcal{A} \left( t, u_0, v_0, w_0 \right) \right\Vert_{L^1_t H^1_\varepsilon}
\leq \tilde{C} \left\Vert u_0 \right\Vert_{H^1_\varepsilon}
\left\Vert v_0 \right\Vert_{H^1_\varepsilon} \left\Vert w_0 \right\Vert_{H^1_\varepsilon}
\end{equation}
where the constant $\tilde{C}$ is independent of $\varepsilon$. Now as long as $\varphi_0 \in H^1$
has $L^2$ norm smaller than some constant depending explicitly on $\tilde{C}$ (\emph{not} the original $C$
from (\ref{eq:cubNLSd2L2})), we can pick a value of $\varepsilon$ depending on $\varphi_0$ so that
the $H_\varepsilon^1$ norm is small enough. The key here is that the constants appearing in
Theorems \ref{thm:abswp} and \ref{thm:absScattering} are quantitative.

The simplest norm which makes the above argument work seems to be the following one:
\begin{equation}
\label{eq:cubNLSd2Hepsno}
\left\Vert \varphi_0 \right\Vert_{H^1_\varepsilon}^2 =
\left\Vert \varphi_0 \right\Vert_{L^2}^2 +
\varepsilon^2 \left\Vert \varphi_0 \right\Vert_{\dot{H}^1}^2
\end{equation}
Now if $\left\Vert \varphi_0 \right\Vert_{L^2} < \eta$ and $\varphi_0 \in H^1$, then there exists
a value of $\varepsilon$ (depending explicitly on $\left\Vert \varphi_0 \right\Vert_{L^2}$
and $\left\Vert \varphi_0 \right\Vert_{\dot{H}^1}$) such that
$\left\Vert \varphi_0 \right\Vert_{H^1_\varepsilon} < \eta$. We have only to choose $\eta$ according
to the constant $\tilde{C}$ instead of the constant $C$; unfortunately, the ``gap'' between
$C$ and $\tilde{C}$ seems to be unrecoverable by this approach.

In order to establish (\ref{eq:cubNLSd2Heps}) for the norm (\ref{eq:cubNLSd2Hepsno}), we estimate
the $L^2$ and $\dot{H}^1$ norms separately, tracking the location of $\varepsilon$ throughout. 
The important observation is a power of $\varepsilon$ is always accompanied by a single derivative
on one of the factors ($u_0$, $v_0$ or $w_0$), while the remaining factors remain in $L^2$.
Thus we may estimate as follows, where $\lesssim$ allows an arbitrary constant which is
\emph{independent of $\varepsilon$}:
\begin{equation}
\begin{aligned}
& \left\Vert \mathcal{A} \left( t, u_0, v_0, w_0 \right) \right\Vert_{L^1_t H^1_\varepsilon} \\
& \qquad \lesssim
\left\Vert \mathcal{A} \left( t, u_0, v_0, w_0 \right) \right\Vert_{L^1_t L^2} +
\varepsilon \left\Vert \mathcal{A} \left( t, u_0, v_0, w_0 \right) \right\Vert_{L^1_t \dot{H}^1} \\
& \qquad \lesssim
\left\Vert u_0 \right\Vert_{L^2} \left\Vert v_0 \right\Vert_{L^2} \left\Vert w_0 \right\Vert_{L^2}
+ \varepsilon
\left\Vert u_0 \right\Vert_{\dot{H}^1} \left\Vert v_0 \right\Vert_{L^2} \left\Vert w_0 \right\Vert_{L^2}  \\
& \qquad \quad + \varepsilon
\left\Vert u_0 \right\Vert_{L^2} \left\Vert v_0 \right\Vert_{\dot{H}^1} \left\Vert w_0 \right\Vert_{L^2} +
\varepsilon
\left\Vert u_0 \right\Vert_{L^2} \left\Vert v_0 \right\Vert_{L^2} \left\Vert w_0 \right\Vert_{\dot{H}^1} \\
& \qquad \lesssim
\left\Vert u_0 \right\Vert_{H^1_\varepsilon} \left\Vert v_0 
\right\Vert_{H^1_\varepsilon} \left\Vert w_0 \right\Vert_{H^1_\varepsilon}
\end{aligned}
\end{equation}

As a result of this calculation, we can conclude the following:
\begin{theorem}
\label{thm:cubNLSH1}
There exists a number $\tilde{\eta} > 0$ such that all the following is true:

Let $\varphi_0 \in H^1 \left( \mathbb{R}^2 \right)$ be such that
$$
\left\Vert \varphi_0 \right\Vert_{L^2 \left( \mathbb{R}^2 \right)} < \tilde{\eta} \,.
$$
Then equation (\ref{eq:cbnlsd2}) has a global solution which scatters in
$H^1 \left( \mathbb{R}^2 \right)$. The solution is unique in the class of
all $L^2$ mild solutions for which $\varphi \in L^3_{t,\textnormal{loc}} L^6_x$.
\end{theorem}

\section{The gain-only Boltzmann equation}
\label{sec:dispest}

%{\color{red}{

In this section, we focus on the gain-only Boltzmann equation.\footnote{The gain-only Boltzmann equation refers to the Boltzmann equation having the $Q^{+}$ term only.}  We employ the inverse Wigner transform which converts this kinetic equation into a hyperbolic Schr\"{o}dinger equation, a technique we explored in \cite{CDP2017,CDP-DCDS2019}. Subsequently,
we can prove a certain bilinear Strichartz estimate (stated in Proposition \ref{prop:gain-eq-bound}), 
based on which we can use Theorem \ref{thm:abswp} to establish small data global well-posedness for this hyperbolic Schr\"{o}dinger equation.
The bilinear Strichartz estimate is obtained from a certain bilinear estimate based on Lorentz spaces, and the validity of the endpoint Strichartz estimate 
for the hyperbolic Schr\"{o}dinger equation (which is crucial for our argument, since the endpoint Strchartz estimate fails on the kinetic side). 
However, once we obtain the bilinear Strichartz estimate on the dispersive side, we can convert it to a {\it bilinear} Strichartz estimate on the kinetic side, see
Proposition \ref{prop:kin-gain-bound}. Consequently, this proposition combined with Theorem \ref{thm:abswp} provide us with small data global well-posedness for the gain-only Boltzmann equation, which is the main result of this section.

Everything below only applies to the gain-only Boltzmann equation
with constant collision kernel in dimension $d=2$. 

%}}

\subsection{Hyperbolic Schr\"{o}dinger equation associated with the gain-only Boltzmann equation}
\label{ssec:gainL2}

We will require the Wigner transform, which we shall now define.
Given a function $f \in L^2_{x,v}$, the Wigner (or Wigner-Weyl) transformation is
defined by the following formula:
\begin{equation}
\gamma \left( x,x^\prime \right) = \int_{\mathbb{R}^d}
f \left( \frac{x+x^\prime}{2},v \right) e^{i v \cdot (x-x^\prime)} dv
\end{equation}
Up to a linear change of variables, this is equivalent to a partial Fourier transform
accounting for only the velocity variable. The inverse transformation is defined by:
\begin{equation}
f \left( x,v \right) = \frac{1}{(2\pi)^d}
\int_{\mathbb{R}^d} \gamma \left( x + \frac{y}{2}, x-\frac{y}{2} \right)
e^{-i v \cdot y} dy
\end{equation}
One of the main interests driving the use of the Wigner transform is that it converts
the free transport generator $-v \cdot \nabla_x$ into the hyperbolic Schr{\" o}dinger
generator $i \Delta_x - i \Delta_{x^\prime}$. Aside from being the starting point for
semiclassical limits (up to scaling), the Wigner transform allows for the transfor of
ideas from the literature of nonlinear Schr{\" o}dinger equations (NLS) into the
kinetic realm. For the present study, the big ideas which we wish to adapt are largely
related to $X^{s,b}$ spaces (also known as Bourgain spaces),
 which are well-studied for NLS and hyperbolic-NLS,
but have not been fully utilized in the kinetic theory literature. We note that the
spaces used in this paper are not actually  Bourgain spaces, but rather, they are
scale-invariant spaces inspired by Bourgain spaces. (See Section \ref{sec:LWPthm}.)

In our situation, namely the Boltzmann equation with constant collision kernel in $d=2$,
$L^2_{x,x^\prime}$ is a scaling critical space for $\gamma$, and corresponds
to $L^2_{x,v}$ for $f$.

\begin{remark}
The use of the Wigner transform is \emph{necessary} for the type of proof used
here. Indeed, if one were to execute the corresponding steps on the kinetic side
(and thereby produce the needed bilinear bound for $Q^+$ acting on the freely transported
solution),
the proof would \emph{fail} because the endpoint kinetic Strichartz estimates are
false in all dimensions. \cite{Bennettetal2014} By contrast, we will be using
the usual endpoint Strichartz estimates for the free hyperbolic Schr{\" o}dinger equation
 in $d=4$ (note the dimension doubling!), which are indeed \emph{true} by Keel-Tao, \cite{KT1998}.
\end{remark}

We use the notation $\eta_\Vert = P_\omega \eta$ and 
$\eta_\bot = \eta - P_\omega \eta$.

\begin{equation}
Q^+ (f,g) (v) = \int_{\mathbb{S}^1} d\omega \int_{\mathbb{R}^2} du
f ( v^*) g (u^*)
\end{equation}

\begin{equation}
Q^- (f,g) (v) = \int_{\mathbb{S}^1} d\omega \int_{\mathbb{R}^2} du
f ( v) g (u)
\end{equation}

\begin{equation}
\left( Q^+ (f,g) \right)^\wedge (\eta)
= \int_{\mathbb{S}^1} d\omega
\hat{f} \left( \eta_\bot \right)
\hat{g} \left( \eta_\Vert \right)
\end{equation}

The Wigner transform of the Boltzmann gain operator $Q^+$ is
\begin{equation}
\begin{aligned}
& B^+ (\gamma_1, \gamma_2) (x,x^\prime) = i \int_{\mathbb{S}^1} d\omega \times \\
& \times \gamma_1 \left( x - \frac{1}{2} P_\omega \left( x-x^\prime \right),
x^\prime + \frac{1}{2} P_\omega \left( x-x^\prime \right) \right) \times \\
& \times \gamma_2 \left( \frac{x+x^\prime}{2} + \frac{1}{2}
P_\omega \left( x-x^\prime \right), \frac{x+x^\prime}{2} -
\frac{1}{2} P_\omega \left( x-x^\prime \right) \right)
\end{aligned}
\end{equation}

\begin{theorem}
\label{thm:gain-eq-gwp}
For any $\gamma_0 \in L^2_{x,x^\prime} \left( \mathbb{R}^2 \times \mathbb{R}^2 \right)$
 with sufficiently small
$L^2_{x,x^\prime}$ norm, there exists a unique global mild solution to the equation
\begin{equation}
\left( i \partial_t + \Delta_x - \Delta_{x^\prime} \right) \gamma (t)
= B^+ \left( \gamma(t), \gamma(t) \right)
\end{equation}
with $\gamma (0) = \gamma_0$
such that $\gamma \in C_t L^2_{x,x^\prime}$ and  $B^+ \left( \gamma , \gamma \right) \in
L^1_{t,\textnormal{loc}} L^2_{x,x^\prime}$. For this solution, it holds that
$\gamma \in L^\infty_{t \in \mathbb{R}} L^2_{x,x^\prime}$ and
$B^+ (\gamma,\gamma) \in L^1_{t \in \mathbb{R}}
L^2_{x,x^\prime}$, and the solution scatters in $L^2_{x,x^\prime}$
as $t \rightarrow \pm \infty$.
\end{theorem}

Theorem \ref{thm:gain-eq-gwp} follows from
Theorem \ref{thm:abswp} along with
 the following estimate for the
gain term $B^+$:
\begin{proposition}
\label{prop:gain-eq-bound}
There is a constant $C > 0$ such that for any
$\gamma_{0,1}, \gamma_{0,2} \in L^2_{x,x^\prime} \left( \mathbb{R}^2 \times \mathbb{R}^2 \right)$,
\begin{equation}
\left\Vert B^+ \left( 
e^{i t \Delta_{\pm}} \gamma_{0,1},
e^{i t \Delta_{\pm}} \gamma_{0,2} \right)
\right\Vert_{L^1_t L^2_{x,x^\prime} \left( \mathbb{R} \times \mathbb{R}^2
\times \mathbb{R}^2 \right) }
\leq C \prod_{i=1,2}
\left\Vert \gamma_{0,i} \right\Vert_{L^2_{x,x^\prime}
\left( \mathbb{R}^2 \times \mathbb{R}^2 \right)}
\end{equation}
where $\Delta_{\pm} = \Delta_x - \Delta_{x^\prime}$.
\end{proposition}

We will need the Lorentz spaces $L^{p,q}$ defined by the following quasi-norm, for any
function $h(\xi) : \mathbb{R}^n \rightarrow \mathbb{C}$,
\begin{equation}
\left\Vert h \left( \xi \right)
\right\Vert_{L^{p,q}_* (\mathbb{R}^n)} =
p^{\frac{1}{q}}
\left\Vert \lambda \left|
\left\{ \xi \in \mathbb{R}^n \; : \;
\left| h (\xi) \right| \geq \lambda \right\} \right|^{\frac{1}{p}}
\right\Vert_{L^q \left( \mathbb{R}^+, \frac{d\lambda}{\lambda}\right)}
\end{equation}
Note that $L^{p,p} = L^p$ for $1 < p < \infty$. In all cases of interest here, the Lorentz quasi-norm
above can be shown to be equivalent to a Banach space norm.

\begin{lemma}
\label{lem:L42bd}
For any Schwartz functions $f,g : \mathbb{R}^2 \rightarrow \mathbb{C}$, there holds
\begin{equation}
\label{eq:QplusL42}
\left\Vert \left( Q^+ (f,g) \right)^\wedge (\eta)
\right\Vert_{L^2_\eta (\mathbb{R}^2)} \leq
C \left\Vert \hat{f} (\eta) \right\Vert_{L^{4,2}_{\eta} (\mathbb{R}^2)}
\left\Vert \hat{g} (\eta) \right\Vert_{L^{4,2}_{\eta} (\mathbb{R}^2)}
\end{equation}
Also, if $\gamma_{0,1}, \gamma_{0,2} \in L^{4,2}_{x,x^\prime} \left( \mathbb{R}^2 \times
\mathbb{R}^2 \right)$, there holds
\begin{equation}
\label{eq:BplusL42}
\left\Vert B^+ \left( \gamma_{0,1}, \gamma_{0,2} \right)
\right\Vert_{L^2_{x,x^\prime} \left( \mathbb{R}^2 \times \mathbb{R}^2 \right)}
\leq C \prod_{i=1,2}
\left\Vert \gamma_{0,i}
\right\Vert_{L^{4,2}_{x,x^\prime} \left( \mathbb{R}^2 \times \mathbb{R}^2 \right)}
\end{equation}
\end{lemma}

\begin{proof}
(Lemma \ref{lem:L42bd})

We apply Minkowski, H{\" o}lder, and Fubini (twice), as follows: 
\begin{equation*}
\begin{aligned}
 \left\Vert \left( Q^+ (f,g) \right)^\wedge (\eta) \right\Vert_{L^2_\eta}
& =
\left\Vert \int_{\mathbb{S}^1} d\omega
\hat{f} \left( \eta_\bot \right)
 \hat{g} \left( \eta_\Vert \right) \right\Vert_{L^2_\eta} \\
& \leq \int_{\mathbb{S}^1} d\omega
\left\Vert \hat{f} \left( \eta_\bot \right)
 \hat{g} \left( \eta_\Vert \right)
\right\Vert_{L^2_\eta} \\
& = \int_{\mathbb{S}^1} d\omega
\left\Vert \hat{f} \left( \eta_\bot \right) \right\Vert_{L^2_{\eta_\bot}}
\left\Vert \hat{g} \left( \eta_\Vert \right)
\right\Vert_{L^2_{\eta_\Vert}} \\
& \leq 
\left\Vert \hat{f} \left( \eta_\bot \right) \right\Vert_{L^2_\omega L^2_{\eta_\bot}}
\left\Vert \hat{g} \left( \eta_\Vert \right)
\right\Vert_{L^2_\omega L^2_{\eta_\Vert}} \\
& =  C \left\Vert \frac{1}{|\eta|^{\frac{1}{2}}}
\hat{f} \left( \eta \right) \right\Vert_{L^2_\eta}
\left\Vert \frac{1}{|\eta|^{\frac{1}{2}}} \hat{g} \left( \eta \right)
\right\Vert_{L^2_\eta} \\
\end{aligned}
\end{equation*}
Then again, because $|\eta|^{-1} \in L^{2,\infty} \left( \mathbb{R}^2 \right)$, 
we may apply the duality $\left( L^{2,1} \right)^{\prime} = L^{2,\infty}$
(\cite{Gra2008} Theorem 1.4.17 (v)),
combined
with the ``power property,'' to deduce
\begin{equation*}
\left\Vert \frac{1}{|\eta|^{\frac{1}{2}}}
\hat{f} \left( \eta \right) \right\Vert_{L^2_\eta} 
= \left\Vert \frac{1}{|\eta|}
\left| \hat{f} \left( \eta \right) \right|^2 \right\Vert_{L^1_\eta}^{\frac{1}{2}}
\lesssim \left\Vert \left| \hat{f} (\eta) \right|^2 \right\Vert^{\frac{1}{2}}_{L^{2,1}_\eta}
\lesssim
\left\Vert \hat{f} (\eta) \right\Vert_{L^{4,2}_\eta \left( \mathbb{R}^2 \right)}
\end{equation*}
hence we obtain
\begin{equation}
\left\Vert \left( Q^+ (f,g) \right)^\wedge (\eta) \right\Vert_{L^2_\eta
\left( \mathbb{R}^2 \right)}
\lesssim
\left\Vert \hat{f} (\eta) \right\Vert_{L^{4,2}_\eta \left( \mathbb{R}^2 \right)}
\left\Vert \hat{g} (\eta) \right\Vert_{L^{4,2}_\eta \left( \mathbb{R}^2 \right)}
\end{equation}
which is (\ref{eq:QplusL42}). \emph{Remark:} The full duality of Lorentz spaces is not
actually necessary at this stage; in fact, a simple application of the Hardy-Littlewood rearrangement
inequality is sufficient.

Using the change of variables
\begin{equation*}
w = \frac{x+x^\prime}{2} \qquad \qquad z=\frac{x-x^\prime}{2}
\end{equation*}
we find that (\ref{eq:BplusL42}) follows immediately from (\ref{eq:QplusL42}) and
H{\" o}lder's inequality, as long
as we can show
\begin{equation} 
\label{eq:L42mixed}
L^{4,2}_{w,z} \left( \mathbb{R}^2 \times \mathbb{R}^2 \right)
\subset
L^4_w \left( \mathbb{R}^2 , L^{4,2}_z \left( \mathbb{R}^2 \right) \right)
\end{equation}
The $L^4_w L^{4,2}_z$ norm of a function
$F (w,z)$ can be controlled directly from the definition of $L^{p,q}$
as follows:
\begin{equation*}
\begin{aligned}
& \left\{ \int_{\mathbb{R}^2} dw
\int_0^\infty \frac{d\lambda}{\lambda} \lambda^2
\left| \left\{ z \in \mathbb{R}^2 \; : \;
\left| F (w,z) \right| \geq \lambda \right\} \right|^{\frac{1}{2}} \times \right. \\
& \qquad \qquad \qquad \qquad \times \left.
\int_0^\infty \frac{d\lambda^\prime}{\lambda^\prime} \left( \lambda^\prime \right)^2
\left| \left\{ z \in \mathbb{R}^2 \; : \;
\left| F (w,z) \right| \geq \lambda^\prime \right\} \right|^{\frac{1}{2}}
\right\}^{\frac{1}{4}}
\end{aligned}
\end{equation*}
Now the idea is to move the $dw$ integral to the \emph{inside} and apply Cauchy-Schwarz
in $w$, followed by Fubini; this leads us to the following quantity:
\begin{equation*}
\begin{aligned}
& \left\{ 
\int_0^\infty \frac{d\lambda}{\lambda} \lambda^2
\left| \left\{ (w,z) \in \mathbb{R}^4 \; : \;
\left| F (w,z) \right| \geq \lambda \right\} \right|^{\frac{1}{2}} \times \right. \\
& \qquad \qquad \qquad  \times \left.
\int_0^\infty \frac{d\lambda^\prime}{\lambda^\prime} \left( \lambda^\prime \right)^2
\left| \left\{ (w,z) \in \mathbb{R}^4 \; : \;
\left| F (w,z) \right| \geq \lambda^\prime \right\} \right|^{\frac{1}{2}}
\right\}^{\frac{1}{4}}
\end{aligned}
\end{equation*}
But this is comparable to the $L^{4,2}_{w,z}$ norm of $F$, so we are done.
\end{proof}

Finally we are ready to prove our main result for this section.

\begin{proof}
(Proposition \ref{prop:gain-eq-bound})

We estimate by Lemma \ref{lem:L42bd}, combined with H{\" o}lder's inequality in time:
\begin{equation}
\label{eq:st11}
\begin{aligned}
& \left\Vert B^+ \left( 
e^{i t \Delta_{\pm}} \gamma_{0,1},
e^{i t \Delta_{\pm}} \gamma_{0,2} \right)
\right\Vert_{L^1_t L^2_{x,x^\prime} \left( \mathbb{R} \times \mathbb{R}^2
\times \mathbb{R}^2 \right) } \\
& \qquad \qquad \qquad \qquad \qquad \qquad
 \leq C \prod_{i=1,2}
\left\Vert e^{i t \Delta_{\pm}} \gamma_{0,i} \right\Vert_{L^2_t L^{4,2}_{x,x^\prime}
\left( \mathbb{R}^2 \times \mathbb{R}^2 \right)}
\end{aligned}
\end{equation}
We apply Theorem 10.1 of Keel-Tao \cite{KT1998}, with $H = L^2_{x,x^\prime} \left( \mathbb{R}^2 \times
\mathbb{R}^2 \right)$,
$B_0 = L^2_{x,x^\prime} \left( \mathbb{R}^2 \times \mathbb{R}^2 \right)$,
$B_1 = L^1_{x,x^\prime} \left( \mathbb{R}^2 \times \mathbb{R}^2 \right)$,
and $\left( q, \sigma, \theta \right) = \left( 2, 2, \frac{1}{2} \right)$ to deduce the
Strichartz estimate (see Appendix \ref{app:strich})
\begin{equation}
\label{eq:st22}
\left\Vert e^{i t \Delta_{\pm}} \gamma_{0} \right\Vert_{L^2_t L^{4,2}_{x,x^\prime}
\left( \mathbb{R}^2 \times \mathbb{R}^2 \right)}
\lesssim \left\Vert \gamma_0 \right\Vert_{L^2_{x,x^\prime}
\left( \mathbb{R}^2 \times \mathbb{R}^2 \right)}
\end{equation}
Here we have used the real interpolation space
\begin{equation}
\left( \left( L^2_{x,x^\prime} , L^1_{x,x^\prime} \right)_{\frac{1}{2}, 2}
\right)^{\prime} = \left( L^{\frac{4}{3},2}_{x,x^\prime} \right)^{\prime} = L^{4,2}_{x,x^\prime}
\end{equation}
e.g. see Chapter 5 of the book \cite{BeLo1976}.

Combining (\ref{eq:st11}) and (\ref{eq:st22}), we are able to conclude.
\end{proof}

\subsection{Back to the gain-only Boltzmann equation}

Combining Proposition \ref{prop:gain-eq-bound} and Plancherel's
theorem, and defining 
$T(t) = e^{-t v \cdot \nabla_x}$,
 we easily deduce the following bound stated in the spatial domain:

\begin{proposition}
\label{prop:kin-gain-bound}
There is a constant $C>0$ such that for any $f_0 , g_0 \in
L^2_{x,v} \left( \mathbb{R}^2 \times \mathbb{R}^2 \right)$,
\begin{equation}
\left\Vert Q^+ \left( T(t) f_0, T(t) g_0 \right) 
\right\Vert_{L^1_t L^2_{x,v} \left( \mathbb{R} \times
\mathbb{R}^2 \times \mathbb{R}^2 \right)} \leq C
\left\Vert f_0 \right\Vert_{L^2_{x,v} \left( \mathbb{R}^2 
\times \mathbb{R}^2 \right)}
\left\Vert g_0 \right\Vert_{L^2_{x,v} \left( \mathbb{R}^2 
\times \mathbb{R}^2 \right)}
\end{equation}
\end{proposition}

%{\color{red}{

The following theorem is an immediate consequence
of Proposition \ref{prop:kin-gain-bound} and Theorem \ref{thm:abswp}: 
\begin{theorem}
\label{thm:gain-eq-gwp-kin}
For any $f_0 \in L^2_{x,v} \left( \mathbb{R}^2 \times \mathbb{R}^2 \right)$
 with sufficiently small
$L^2_{x,v}$ norm, there exists a unique global ($t \in \mathbb{R}$) mild solution to the equation
\begin{equation}
\label{eq:Qpe000}
\left( \partial_t + v \cdot \nabla_x \right) f (t)
= Q^+ \left( f(t), f(t) \right)
\end{equation}
with $f (0) = f_0$
such that $f \in C_t L^2_{x,v}$ and  $Q^+ \left( f , f \right) \in
L^1_{t,\textnormal{loc}} L^2_{x,v}$. For this solution, it holds that
$f \in L^\infty_{t \in \mathbb{R}} L^2_{x,v}$ and
$Q^+ (f,f) \in L^1_{t \in \mathbb{R}}
L^2_{x,v}$, and the solution scatters in $L^2_{x,v}$
as $t \rightarrow \pm \infty$.
\end{theorem}

\begin{remark}
It is not necessary in Theorem \ref{thm:gain-eq-gwp-kin} for $f_0$ to be non-negative.
However, assuming $f_0$ is non-negative, we can show that the solution $f(t)$ of
the $Q^+$ equation
(\ref{eq:Qpe000}) is non-negative for $\textnormal{ a.e. } (t,x,v) \in \left( 0,\infty \right)
\times \mathbb{R}^2 \times \mathbb{R}^2$. Indeed, there is a globally convergent
expansion of $f(t)$ in terms of $f_0$, which comes from iterating Duhamel's formula:
\begin{equation}
\begin{aligned}
& f(t) = T(t) f_0 +
\int_0^t T(t-t_1) Q^+ \left( T(t_1) f_0, T(t_1) f_0 \right) dt_1 + \\
& + \int_0^t \int_0^{t_1} T(t-t_1)
Q^+ \left(  T(t_1-t_2) Q^+ \left( T(t_2) f_0, T(t_2) f_0 \right),
 T(t_1) f_0 \right) dt_2 dt_1 + \dots
\end{aligned}
\end{equation}
If $f_0 \geq 0$ then all the terms in the series are non-negative for $t \geq 0$; hence,
 the solution $f(t)$ is non-negative at positive times.
\end{remark}

%}}

\subsection{Short-time estimates.}

The bilinear estimates above will not be suitable for every result we
wish to prove, e.g. uniqueness, where we must rely upon integrability 
properties instead of regularity. For this reason we will require the following
``short-time''  estimates which follow essentially from the dominated
convergence theorem.

\begin{proposition}
\label{prop:shorttime}
Let $f_0 \in L^2_{x,v} \left( \mathbb{R}^2 \times \mathbb{R}^2 \right)$.
Then there holds
\begin{equation}
\limsup_{T \rightarrow 0^+} \sup_{g_0 \in L^2_{x,v},\;
\left\Vert g_0 \right\Vert_{L^2_{x,v}} = 1}
\left\Vert Q^+ \left( T(t) f_0, T(t) g_0 \right) \right\Vert_{L^1_{t \in [-T,T]}
L^2_{x,v}} = 0
\end{equation}
\begin{equation}
\limsup_{T \rightarrow 0^+} \sup_{g_0 \in L^2_{x,v},\;
\left\Vert g_0 \right\Vert_{L^2_{x,v}} = 1}
\left\Vert Q^+ \left( T(t) g_0, T(t) f_0 \right) \right\Vert_{L^1_{t \in [-T,T]}
L^2_{x,v}} = 0
\end{equation}
\end{proposition}
\begin{proof}
We only prove the first bound; the second proceeds similarly.
By the proof of Proposition \ref{prop:gain-eq-bound}, 
for any two density
matrices $\gamma_{0,1}, \gamma_{0,2} \in L^2_{x,x^\prime}$, $B^+$
(the Wigner transform of $Q^+$) satisfies the bilinear estimates
\begin{equation}
\begin{aligned}
& \left\Vert B^+ \left( e^{i t \Delta_{\pm}} \gamma_{0,1},
e^{i t \Delta_{\pm}} \gamma_{0,2} \right) 
\right\Vert_{L^1_{t \in [-T,T]} L^2_{x,x^\prime}} \\
& \qquad \qquad \qquad \qquad \leq C \prod_{i \in \left\{ 1,2 \right\}}
\left\Vert e^{i t \Delta_{\pm}} \gamma_{0,i} 
\right\Vert_{L^2_{t \in [-T,T]} L^{4,2}_{x,x^\prime}}
\end{aligned}
\end{equation}
Apply Strichartz in the \emph{second entry only} to yield
\begin{equation}
\begin{aligned}
& \left\Vert B^+ \left( e^{i t \Delta_{\pm}} \gamma_{0,1},
e^{i t \Delta_{\pm}} \gamma_{0,2} \right) 
\right\Vert_{L^1_{t \in [-T,T]} L^2_{x,x^\prime}} \\
& \qquad \qquad \qquad \qquad \leq C 
\left\Vert e^{i t \Delta_{\pm}} \gamma_{0,1} 
\right\Vert_{L^2_{t \in [-T,T]} L^{4,2}_{x,x^\prime}}
\left\Vert \gamma_{0,2} \right\Vert_{L^2_{x,x^\prime}}
\end{aligned}
\end{equation}
Now observe that since $\gamma_{0,1} \in L^2_{x,x^\prime}$ by assumption,
it follows that $e^{i t \Delta_{\pm}} \gamma_{0,1} \in L^2_t L^{4,2}_{x,x^\prime}$
by Strichartz; therefore, by the dominated convergence theorem,
\begin{equation}
\limsup_{T \rightarrow 0^+}
\left\Vert e^{i t \Delta_{\pm}} \gamma_{0,1} \right\Vert_{L^2_{t \in [-T,T]}
L^{4,2}_{x,x^\prime}} = 0
\end{equation}
We take the sup in $\gamma_{0,2}$, followed by the limsup in $T$, and then
conclude by Plancherel.
\end{proof}

\section{Tools for the analysis of the full Boltzmann equation}

%{\color{red}{

In this section, we present key tools that will allow us to treat the full Boltzmann equation in subsequent sections. 

We start this section by presenting Strichartz estimates for the spatial density 
\begin{equation}
\rho_f (x) = \int_{\mathbb{R}^2} f(x,v) dv
\end{equation}
in Section \ref{ssec:Strich}.

The main challenge for solving Boltzmann's equation (with a constant collision kernel)
 in $L^2_{x,v} \left( \mathbb{R}^2 \times \mathbb{R}^2 \right)$ is that the
spatial density $\rho_f$ is not necessarily well-defined when $f \in L^2_{x,v}$; therefore, since the loss term 
has the form $Q^- (f,f) = f \rho_f$, we find that $Q^-$ might not make sense. 
The ideal way to deal with this situation would be to realize that $Q^-$ \emph{subtracts}
from $f$, and therefore view the loss term as an \emph{unbounded operator} at least
when $t \rightarrow 0^+$. However, it is not clear to us how to implement this strategy,
nor whether it would produce enough integrability to prove uniqueness
(and we are not aware of any full treatment of this problem in the literature).
The simplest way to avoid the issue of unbounded operators is to introduce an auxiliary
norm; one natural possibility would be the $L^1_{x,v}$ norm of $f$ (since it is conserved if
$f_0$ has enough smoothness and decay),
but we have instead elected to impose moment and regularity bounds on $f_0$ so that
we can employ Strichartz estimates \emph{in the auxiliary space}, which we introduce in Section 
\ref{sec:wr}.

%}}

\subsection{Strichartz Estimates for the Spatial Density}
\label{ssec:Strich}

The following lemma follows from a velocity averaging argument. We present the details 
following the dispersive context \cite{KM2008} for the reader's convenience.

\begin{lemma}
\label{lem:A1lem}
Fix a sufficiently small number $\delta > 0$.
Let $I \subseteq \mathbb{R}$ be an open interval and let $f (t,x,v)
 : I \times \mathbb{R}^2 \times \mathbb{R}^2
\rightarrow \mathbb{R}$ be a measurable and locally integrable
 function. Then  the following estimates hold whenever 
the respective norms are finite:
\begin{equation}
\begin{aligned}
\label{eq:rhoLinf}
& \left\Vert \rho_f \right\Vert_{L^2_{t \in I} L^\infty_x}
\leq \tilde{C}_\delta
\left( \left\Vert \left< v \right>^{\frac{1}{2}+\delta}
\left< \nabla_x \right>^{\frac{1}{2}+\delta} f
\right\Vert_{L^\infty_{t \in I} L^2_{x,v}}  + \right. \\
& \qquad \qquad \qquad \qquad \qquad
\left. + \left\Vert \left< v \right>^{\frac{1}{2}+\delta}
\left< \nabla_x \right>^{\frac{1}{2}+\delta}
\left( \partial_t + v \cdot \nabla_x \right) f
\right\Vert_{L^1_{t \in I} L^2_{x,v}}\right)
\end{aligned}
\end{equation}
\begin{equation}
\label{eq:rhoL4}
\left\Vert \rho_f \right\Vert_{L^2_{t \in I} L^4_x}
\leq C_\delta
\left( \left\Vert \left< v \right>^{\frac{1}{2}+\delta} f
\right\Vert_{L^\infty_{t \in I} L^2_{x,v}} +
\left\Vert \left< v \right>^{\frac{1}{2}+\delta}
\left( \partial_t + v \cdot \nabla_x \right) f
\right\Vert_{L^1_{t \in I} L^2_{x,v}}\right)
\end{equation}
The constants $C_\delta, \tilde{C}_\delta$ 
do not depend on the interval $I$.
\end{lemma}
\begin{proof}
Observe that
(\ref{eq:rhoLinf}) follows immediately from
(\ref{eq:rhoL4}) due to Morrey inequalities \cite{St1970} and
the fact that $\left< \nabla_x \right>$ commutes with
the operators $\left( \partial_t + v\cdot \nabla_x \right)$ and
$f \mapsto \rho_f$. Therefore, we will prove only the estimate (\ref{eq:rhoL4}); moreover,
up to possibly increasing the constant $C_\delta$
 by a fixed factor, we are
free to assume that $I = \mathbb{R}$ by standard approximation arguments.
If the right hand side of (\ref{eq:rhoL4}) is finite, then it immediately
follows that
 $\left< v \right>^{\frac{1}{2}+\delta} f \in C \left( I, L^2_{x,v} \right)$, 
so we can assume $f$ is as regular as necessary by standard approximation arguments.
Finally, by Duhamel's formula we have
\begin{equation*}
f(t) = e^{-t v \cdot \nabla_x} f_0 +
\int_0^t e^{-(t-\sigma) v \cdot \nabla_x} 
\left\{ \left( \partial_t + v\cdot \nabla_x \right) f\right\}
 (\sigma) d\sigma
\end{equation*}
Using Duhamel, along with the linearity of the map
$f \mapsto \rho_f$ and Minkowski's inequality (first in $x$, then in $t$), we obtain
\begin{equation*}
\begin{aligned}
& \left\Vert \rho_f \right\Vert_{L^2_t L^4_x} \leq
\left\Vert \rho \left[ e^{-t v \cdot \nabla_x} f_0 \right] \right\Vert_{L^2_t L^4_x} + \\
& \qquad \qquad \qquad \qquad + 
\int_{\mathbb{R}} \left\Vert
\rho \left[ e^{-(t-\sigma) v \cdot \nabla_x} 
\left\{ \left( \partial_t + v\cdot \nabla_x \right) f\right\}
 (\sigma) \right]
\right\Vert_{L^2_t L^4_x} d\sigma
\end{aligned}
\end{equation*}
and therefore we immediately deduce (\ref{eq:rhoL4}) once the same inequality
holds with
\begin{equation*}
\left( \partial_t + v \cdot \nabla_x \right) f = 0
\end{equation*}
In  words, we can assume $f$ is a solution of the free transport equation.

Altogether, we only need to show that if $f_0 (x,v)$ is smooth and compactly supported
in $\mathbb{R}^2 \times \mathbb{R}^2$ then
\begin{equation}
\left\Vert \rho_{T(t) f_0} \right\Vert_{L^2_t L^4_x
\left( \mathbb{R} \times \mathbb{R}^2 \right)}
\leq C_\delta \left\Vert \left< v \right>^{\frac{1}{2}+\delta}
f_0  \right\Vert_{L^2_{x,v} \left( \mathbb{R}^2 \times \mathbb{R}^2 \right)}
\end{equation}
where $T(t) f_0 = e^{-t v \cdot \nabla_x} f_0$. By the fractional Gagliardo-Nirenberg-Sobolev
inequality \cite{St1970}, it suffices to show
\begin{equation}
\label{eq:H1over2bd}
\left\Vert \left( -\Delta_x \right)^{\frac{1}{4}}
\rho_{T(t) f_0} \right\Vert_{L^2_t L^2_x \left( \mathbb{R} \times \mathbb{R}^2 \right)}
\leq C_\delta \left\Vert 
\left< v \right>^{\frac{1}{2}+\delta}
f_0  \right\Vert_{L^2_{x,v} \left( \mathbb{R}^2 \times \mathbb{R}^2 \right)}
\end{equation}
whenever $f_0$ is smooth and compactly supported in
$\mathbb{R}^2 \times \mathbb{R}^2$.\footnote{Note that if $f_0$ is smooth and compactly supported,
then for any fixed $t \in \mathbb{R}$, $T(t) f_0$ is also smooth and compactly supported.}
 We will establish (\ref{eq:H1over2bd}) using
the spacetime Fourier transform to conclude the lemma.

To prove (\ref{eq:H1over2bd}), we apply Plancherel in $(t,x)$ on the left-hand side,
and in $x$ on the right-hand side; hence, an equivalent bound is:
\begin{equation}
\label{eq:H1over2bdFourier}
\left\Vert \mathcal{F}_{t,x} \left\{
\left( -\Delta_x \right)^{\frac{1}{4}}
\rho_{T(t) f_0} 
\right\} (\tau,\xi)
\right\Vert_{L^2_\tau L^2_\xi }
\leq C_\delta  \left\Vert \left< v \right>^{\frac{1}{2}+\delta}
\mathcal{F}_x \left\{ f_0 \right\} (\xi,v)
  \right\Vert_{L^2_{\xi,v} }
\end{equation}
Let us define
\begin{equation}
H (\xi,v) = \mathcal{F}_x \left\{ f_0 \right\} (\xi,v)
\end{equation}
Then (\ref{eq:H1over2bdFourier}) may be re-cast as the following inequality:
\begin{equation}
\label{eq:H1over2bdFourierH}
\left\Vert 
\left| \xi \right|^{\frac{1}{2}}
\int_{\mathbb{R}^2} dv \delta \left( \tau + v \cdot \xi \right)
H \left( \xi, v\right)
\right\Vert_{L^2_\tau L^2_\xi }^2
\leq C_\delta^2 \left\Vert \left< v \right>^{\frac{1}{2}+\delta} 
H (\xi,v)
  \right\Vert_{L^2_{\xi,v} }^2
\end{equation}
The quantity on the left can be equivalently written
\begin{equation*}
\int_{\mathbb{R}} d\tau \int_{\mathbb{R}^2} d\xi 
\int_{\mathbb{R}^2} dv
\int_{\mathbb{R}^2} du \delta \left( \tau + v \cdot \xi \right)
\delta \left( \tau + u \cdot \xi \right) 
\left| \xi \right|
H \left( \xi,v \right) \overline{H \left( \xi,u \right)}
\end{equation*}
which is the same as:
\begin{equation*}
\begin{aligned}
& \int_{\mathbb{R}} d\tau \int_{\mathbb{R}^2} d\xi 
\int_{\mathbb{R}^2} dv
\int_{\mathbb{R}^2} du \delta \left( \tau + v \cdot \xi \right)
\delta \left( \tau + u \cdot \xi \right) 
\left| \xi \right| \times \\
& \qquad \qquad \qquad \times
\left( \frac{1}{ \left< u \right>^{\frac{1}{2}+\delta}}
\left< v \right>^{\frac{1}{2}+\delta}
H \left( \xi,v \right) \right)
\left( \frac{1}{\left< v \right>^{\frac{1}{2}+\delta}}
\left< u \right>^{\frac{1}{2}+\delta} 
\overline{H \left( \xi,u \right)} \right)
\end{aligned}
\end{equation*}
The idea of \cite{KM2008} is to apply the Cauchy-Schwarz inequality, $AB \leq 
\frac{A^2}{2} + \frac{B^2}{2}$, but \emph{pointwise}
in $(\tau,\xi,v,u)$ (not in the integral sense!) to the two
terms in the large parentheses. Thus we will end up with the sum
of two terms, one involving only $H(\xi,v)$ and the other only
involving $H(\xi,u)$; under the obvious symmetry
$u \leftrightarrow v$, we can discard one of them up to a factor of $2$.

Thus we now only need to prove
\begin{equation*}
\begin{aligned}
& \int_{\mathbb{R}} d\tau \int_{\mathbb{R}^2} d\xi 
\int_{\mathbb{R}^2} dv
\int_{\mathbb{R}^2} du \delta \left( \tau + v \cdot \xi \right)
\delta \left( \tau + u \cdot \xi \right) 
\left| \xi \right| \times \\
& \qquad \qquad \qquad \qquad \qquad \qquad \qquad \qquad \times
\left( \frac{1}{ \left< u \right>^{1 + 2 \delta}}
\left< v \right>^{1 + 2 \delta}
\left| H \left( \xi,v \right) \right|^2 \right) \\
& \qquad \qquad \leq C_\delta^2 \left\Vert
\left< v \right>^{\frac{1}{2}+\delta}
H \left( \xi,v \right) \right\Vert_{L^2_{\xi,v}}^2
\end{aligned}
\end{equation*}
(we can assume $H$ vanishes for $\xi$ close to the origin, so that the integral
on the left certainly makes sense). Hence if we can show that
\begin{equation}
\sup_{\left( \tau,\xi \right) \in \mathbb{R} \times \mathbb{R}^2_{\neq 0}}
\int_{\mathbb{R}^2} du \delta \left( \tau + u \cdot \xi \right)
\frac{ \left| \xi \right| }{\left< u \right>^{1+2\delta}}
< \infty
\end{equation}
then we will be done (note that the other $\delta$-function,
$\delta \left( \tau + v \cdot \xi \right)$, is absorbed
by the integral in $\tau$, but only \emph{after} using the
supremum bound).

Let us define
\begin{equation*}
I  \left( \tau,\xi \right) 
= \int_{\mathbb{R}^2} du \delta \left( \tau + u \cdot \xi \right)
\frac{ \left| \xi \right| }{\left< u \right>^{1+2\delta}}
\end{equation*}
If we denote the line 
\begin{equation*}
P \left( \tau, \xi \right)
= \left\{ u \in \mathbb{R}^2 \; \left|
\; \tau + u \cdot \xi = 0 \right. \right\}
\end{equation*}
then it follows that
\begin{equation*}
I \left( \tau,\xi \right) 
= \int_{u \in P \left( \tau,\xi \right) } d\ell (u)
\frac{ 1 }{\left< u \right>^{1+2\delta}}
\end{equation*}
where $d\ell (u)$ is the induced linear measure.
We can only increase the value of the integral 
of $\left< u \right>^{-1-2\delta}$ by translating the
line $P \left( \tau,\xi \right)$ toward the origin of $\mathbb{R}^2$.
Therefore,
\begin{equation*}
\sup_{(\tau,\xi) \in \mathbb{R} \times \mathbb{R}^2_{\neq 0}}
I \left( \tau,\xi \right)
\leq
\int_{q \in \mathbb{R}} \frac{dq}{\left( 1+q^2
\right)^{\frac{1}{2}+\delta} }
< \infty
\end{equation*}
so we are able to conclude.
\end{proof}

\subsection{Weights and regularity}
\label{sec:wr}

Let $\varepsilon \in \left( 0,1 \right]$ and define the norm
\begin{equation}
\left\Vert f_0 \right\Vert_{H^{1,1}_\varepsilon}^2
= \left\Vert f_0 \right\Vert_{L^2_{x,v}}^2
+ \varepsilon^2 \left\Vert v f_0 \right\Vert_{L^2_{x,v}}^2
+ \varepsilon^2 \left\Vert \nabla_x f_0 \right\Vert_{L^2_{x,v}}^2
+ \varepsilon^4 \left\Vert v \otimes \nabla_x
f_0 \right\Vert_{L^2_{x,v}}^2 
\end{equation}
Note that the space $H^{1,1}_\varepsilon$ is independent of 
$\varepsilon > 0$, but the
norm of a fixed element $f_0 \in H^{1,1}_\varepsilon$
 does depend on $\varepsilon$
in general.
The norm on $H^{1,1}_\varepsilon$ is equivalently written:
\begin{equation}
\left\Vert f_0 \right\Vert_{H^{1,1}_\varepsilon}
= \left\Vert \left( 1 + \varepsilon^2 |v|^2 \right)^{\frac{1}{2}}
\left( 1 + \varepsilon^2 |\xi|^2 \right)^{\frac{1}{2}}
\mathcal{F}_x f_0 (\xi,v) \right\Vert_{L^2_{\xi,v}}
\end{equation}
where $\mathcal{F}_x f_0$ is the Fourier transform of $f_0$ in the spatial variable only.
This may also be written 
\begin{equation}
\left\Vert f_0 \right\Vert_{H^{1,1}_\varepsilon}
= \left\Vert \left< \varepsilon v \right>   \left< \varepsilon \nabla_x \right> 
f_0 \right\Vert_{L^2_{x,v}}
\end{equation}
where $\left< v \right> = \left( 1 + |v|^2 \right)^{\frac{1}{2}}$.
We will use the notation $H^{1,1}
 \equiv H^{1,1}_{1}$ when the dependence on $\varepsilon$
is unimportant. 

More generally, we also define the norms
\begin{equation}
\left\Vert f_0 \right\Vert_{H^{\alpha,\beta}_\varepsilon}
= \left\Vert \left< \varepsilon v \right>^\beta
\left< \varepsilon \nabla_x \right>^\alpha
f_0 \right\Vert_{L^2_{x,v}}
\end{equation}
where the exponents $\alpha,\beta \geq 0$ are chosen
independently.

The following commutation relations are standard:
\begin{equation*}
\nabla_x Q^+ \left( f,g \right) =
Q^+ \left( \nabla_x f, g \right) + Q^+ \left( f, \nabla_x g\right)
\end{equation*}
\begin{equation*}
\nabla_x T(t) f_0 = T(t) \nabla_x f_0 
\end{equation*}
\begin{equation*}
v T(t) f_0 = T(t) \left( v f_0 \right)
\end{equation*}
Additionally, from conservation of energy, we have:
\begin{equation*}
\left| v Q^+ \left( f,g \right) \right|
\lesssim Q^+ \left( \left| v f \right| , \left| g\right| \right)
+ Q^+ \left( \left| f \right|, \left| v g \right|  \right)
\end{equation*}
Using the commutation relations and Proposition \ref{prop:kin-gain-bound},
we have:
\begin{equation}
\label{eq:bdA}
\left\Vert \nabla_x Q^+ \left( T(t) f_0, T(t) g_0 \right)
\right\Vert_{L^1_t L^2_{x,v}} \lesssim
\left\Vert \nabla_x f_0 \right\Vert_{L^2_{x,v}}
\left\Vert g_0 \right\Vert_{L^2_{x,v}}
+ \left\Vert f_0 \right\Vert_{L^2_{x,v}}
\left\Vert \nabla_x g_0 \right\Vert_{L^2_{x,v}}
\end{equation}
\begin{equation}
\label{eq:bdB}
\left\Vert v Q^+ \left( T(t) f_0, T(t) g_0 \right)
\right\Vert_{L^1_t L^2_{x,v}} \lesssim
\left\Vert v f_0 \right\Vert_{L^2_{x,v}}
\left\Vert g_0 \right\Vert_{L^2_{x,v}}
+ \left\Vert f_0 \right\Vert_{L^2_{x,v}}
\left\Vert v g_0 \right\Vert_{L^2_{x,v}}
\end{equation}
and
\begin{equation}
\label{eq:bdC}
\begin{aligned}
& \left\Vert v \otimes \nabla_x
Q^+ \left( T(t) f_0, T(t) g_0 \right) \right\Vert_{L^1_t L^2_{x,v}}
\lesssim
\left\Vert v \otimes \nabla_x f_0 \right\Vert_{L^2_{x,v}}
\left\Vert g_0 \right\Vert_{L^2_{x,v}} + \\
& \;\; + \left\Vert \nabla_x f_0 \right\Vert_{L^2_{x,v}} 
\left\Vert v g_0 \right\Vert_{L^2_{x,v}} +
\left\Vert v f_0 \right\Vert_{L^2_{x,v}}
\left\Vert \nabla_x g_0 \right\Vert_{L^2_{x,v}} +
\left\Vert f_0 \right\Vert_{L^2_{x,v}} 
\left\Vert v \otimes \nabla_x g_0 \right\Vert_{L^2_{x,v}}
\end{aligned}
\end{equation}
Using (\ref{eq:bdA}), (\ref{eq:bdB}) and (\ref{eq:bdC}), and the
definition of $H^{1,1}_\varepsilon$, we obtain the following
estimate:
\begin{equation}
\label{eq:bdCm1}
\left\Vert Q^+ \left( T(t) f_0, T(t) g_0 \right)
\right\Vert_{L^1_t H^{1,1}_\varepsilon} \leq C
\left\Vert f_0 \right\Vert_{H^{1,1}_\varepsilon} 
\left\Vert g_0 \right\Vert_{H^{1,1}_\varepsilon}
\end{equation}
where the constant $C$ does not depend on $\varepsilon \in \left( 0,1 \right]$.

\begin{proposition}
For any $f_0, g_0 \in H^{1,1}$, there holds
\begin{equation}
\left\Vert Q^+ \left( T(t) f_0, T(t) g_0 ) \right)
\right\Vert_{L^1_t H^{1,1}_\varepsilon} \leq
C \left\Vert f_0 \right\Vert_{H^{1,1}_\varepsilon}
\left\Vert g_0 \right\Vert_{H^{1,1}_\varepsilon}
\end{equation}
where $T(t) = e^{-t v \cdot \nabla_x}$.
The constant $C$ is independent of $\varepsilon \in \left( 0,1 \right]$.
\end{proposition}

Similarly, we also have:

\begin{proposition}
For any $f_0, g_0 \in H^{0,1}$, there holds
\begin{equation}
\left\Vert \left< \varepsilon v \right>
 Q^+ \left( T(t) f_0, T(t) g_0 ) \right)
\right\Vert_{L^1_t L^2_{x,v}} \leq
C \left\Vert
\left< \varepsilon v \right> f_0 \right\Vert_{L^2_{x,v}}
\left\Vert \left< \varepsilon v \right>
g_0 \right\Vert_{L^2_{x,v}}
\end{equation}
where $T(t) = e^{-t v \cdot \nabla_x}$.
The constant $C$ is independent of $\varepsilon \in \left( 0,1 \right]$.
\end{proposition}

The bounds in the preceding two propositions can be interpolated
against Proposition \ref{prop:kin-gain-bound}, using
Theorem 5.1.2 of the book \cite{BeLo1976}, to obtain:

\begin{proposition}
\label{prop:gain-eq-bound-kin-alpha}
Let $\alpha \in (0,1)$.
For any $f_0, g_0 \in H^{\alpha,\alpha}$, there holds
\begin{equation}
\left\Vert Q^+ \left( T(t) f_0, T(t) g_0 ) \right)
\right\Vert_{L^1_t H^{\alpha,\alpha}_\varepsilon} \leq
C \left\Vert f_0 \right\Vert_{H^{\alpha,\alpha}_\varepsilon}
\left\Vert g_0 \right\Vert_{H^{\alpha,\alpha}_\varepsilon}
\end{equation}
where $T(t) = e^{-t v \cdot \nabla_x}$.
The constant $C$ is independent of $\varepsilon,\alpha$.
\end{proposition}

\begin{proposition}
\label{prop:gain-eq-bound-kin-alpha-2}
Let $\alpha \in (0,1)$. For any $f_0, g_0 \in H^{0,\alpha}$, there holds
\begin{equation}
\left\Vert \left< \varepsilon v \right>^\alpha
 Q^+ \left( T(t) f_0, T(t) g_0 ) \right)
\right\Vert_{L^1_t L^2_{x,v}} \leq
C \left\Vert
\left< \varepsilon v \right>^\alpha f_0 \right\Vert_{L^2_{x,v}}
\left\Vert \left< \varepsilon v \right>^\alpha
g_0 \right\Vert_{L^2_{x,v}}
\end{equation}
where $T(t) = e^{-t v \cdot \nabla_x}$.
The constant $C$ is independent of $\varepsilon,\alpha$.
\end{proposition}

\subsection{A useful lemma}
\label{ssec:usefullemma}

The following lemma is a consequence of Section \ref{sec:LWPthm}; we record it here
to help clarify the main ideas underlying the present work. Note that the theory of
Section \ref{sec:LWPthm} cannot be applied ``out of box'' to the Boltzmann equation
accounting for the loss term. For this reason, it is crucial to observe that the theory
of Section \ref{sec:LWPthm} rests upon a single bound which can be applied to the
$Q^+$ term in any estimate.

\begin{lemma}
\label{lem:usefullemma}
Let $I = (a,b) \subset \mathbb{R}$ be a nonempty open interval with
$-\infty \leq a < b \leq +\infty$. Furthermore, for $i=1,2$, suppose
$f_i (t,x,v) : I \times \mathbb{R}^2 \times \mathbb{R}^2 \rightarrow \mathbb{R}$ 
is a function such
that $f_i \in L^\infty_{t \in I} L^2_{x,v}$ and
$\left( \partial_t + v \cdot \nabla_x \right) f_i \in L^1_{t \in I} L^2_{x,v}$.
Then the following estimate holds:
\begin{equation}
\begin{aligned}
& \left\Vert Q^+ \left( f_1 (t), f_2 (t) \right) \right\Vert_{L^1_{t \in I} L^2_{x,v}} \\
& \qquad \quad \leq C \prod_{i=1,2} \left(
\left\Vert f_i (t) \right\Vert_{L^\infty_{t \in I} L^2_{x,v}}
+ \left\Vert \left( \partial_t + v \cdot \nabla_x \right)
f_i (t) \right\Vert_{L^1_{t \in I} L^2_{x,v}} \right)
\end{aligned}
\end{equation}
for some constant $C$ which does not depend on $f_1,f_2$ or the interval $I$.
\end{lemma}
\begin{proof}
We may assume without loss that $I = (0,T)$ for some $T > 0$.
The lemma then follows from Proposition \ref{prop:gain-eq-bound-kin-alpha} and 
Lemma \ref{lem:abswplem}, under the following assignments:
$\mathcal{H} = L^2_{x,v}$, $x_j (t) = e^{t v \cdot \nabla_x} f_j (t)$, and
$$
\mathcal{A} (t,x_1,x_2) =
e^{t v \cdot \nabla_x} Q^+
\left( e^{-t v \cdot \nabla_x} x_1, e^{-t v \cdot \nabla_x} x_2 \right)
$$
Here we have used that $e^{-t v \cdot \nabla_x}$ is an isometry on $L^2_{x,v}$ for
any $t \in \mathbb{R}$.
\end{proof}

Similarly we deduce the following result as a consequence of
Proposition \ref{prop:gain-eq-bound-kin-alpha} and
Lemma \ref{lem:abswplem}:

\begin{lemma}
\label{lem:usefullemma2}
Let $\varepsilon \in (0,1]$ and let $\alpha \in (0,1)$. 
Let $I = (a,b) \subset \mathbb{R}$ be a nonempty open interval with
$-\infty \leq a < b \leq +\infty$. Furthermore, for $i=1,2$, suppose
$f_i (t,x,v) : I \times \mathbb{R}^2 \times \mathbb{R}^2 \rightarrow \mathbb{R}$ 
is a function such
that $f_i \in L^\infty_{t \in I} H^{\alpha,\alpha}$ and
$\left( \partial_t + v \cdot \nabla_x \right) f_i \in L^1_{t \in I} 
H^{\alpha,\alpha}$.
Then the following estimate holds:
\begin{equation}
\begin{aligned}
& \left\Vert Q^+ \left( f_1 (t), f_2 (t) \right) \right\Vert_{L^1_{t \in I} 
H^{\alpha,\alpha}_\varepsilon} \\
& \qquad \quad \leq C \prod_{i=1,2} \left(
\left\Vert f_i (t) \right\Vert_{L^\infty_{t \in I} 
H^{\alpha,\alpha}_\varepsilon}
+ \left\Vert \left( \partial_t + v \cdot \nabla_x \right)
f_i (t) \right\Vert_{L^1_{t \in I} H^{\alpha,\alpha}_\varepsilon} \right)
\end{aligned}
\end{equation}
for some constant $C$ which does not depend on $f_1,f_2,\alpha,\varepsilon$ or the interval $I$.
\end{lemma}

The following result is similarly straightforward to prove by omitting spatial
derivatives throughout the argument.

\begin{lemma}
\label{lem:usefullemma3}
Let $\varepsilon \in (0,1]$ and let $\alpha \in (0,1)$.
Let $I = (a,b) \subset \mathbb{R}$ be a nonempty open interval with
$-\infty \leq a < b \leq +\infty$. Furthermore, for $i=1,2$, suppose
$f_i (t,x,v) : I \times \mathbb{R}^2 \times \mathbb{R}^2 \rightarrow \mathbb{R}$ 
is a function such
that $\left< v \right>^\alpha f_i \in L^\infty_{t \in I} L^2_{x,v}$ and
$\left< v \right>^\alpha
\left( \partial_t + v \cdot \nabla_x \right) f_i \in L^1_{t \in I} L^2_{x,v}$.
Then the following estimate holds:
\begin{equation}
\begin{aligned}
& \left\Vert \left< \varepsilon v \right>^\alpha
Q^+ \left( f_1 (t), f_2 (t) \right) \right\Vert_{L^1_{t \in I} L^2_{x,v}} \\
&  \quad \leq C \prod_{i=1,2} \left(
\left\Vert \left< \varepsilon v \right>^\alpha
 f_i (t) \right\Vert_{L^\infty_{t \in I} L^2_{x,v}}
+ \left\Vert 
\left< \varepsilon v \right>^\alpha
\left( \partial_t + v \cdot \nabla_x \right)
f_i (t) \right\Vert_{L^1_{t \in I} L^2_{x,v}} \right)
\end{aligned}
\end{equation}
for some constant $C$ which does not depend on $f_1,f_2,\alpha,\varepsilon$ or the interval $I$.
\end{lemma}

\section{Uniqueness}
\label{sec:unique}

%{\color{red}{
In this section, we present our main uniqueness result. 
%}}

\begin{theorem} 
\label{thm:unique}
There is at most one 
mild solution of the full Boltzmann equation
on an interval $[0,T]$, with given initial data $f_0 $,
 such that the estimates
\begin{equation}
\label{eq:lsA}
\left< v \right>^{\frac{1}{2}+}
Q^+ (f,f) \in L^1_{t \in [0,T]} L^2_{x,v} 
\end{equation}
\begin{equation}
\label{eq:lsB}
\rho_f \in L^2_{t \in [0,T]} L^\infty_x \bigcap L^2_{t \in [0,T]} L^4_x
\end{equation}
\begin{equation}
\label{eq:lsC}
\left< v \right>^{\frac{1}{2}+}  f \in
L^\infty_{t \in [0,T]} L^2_{x,v} \bigcap
L^\infty_{t \in [0,T]} L^4_x L^2_v
\end{equation}
are all verified.
\end{theorem}

\begin{remark}
Theorem \ref{thm:unique} makes no assumptions about the non-negativity of either
$f (t)$ or $f(0) = f_0$;
in particular, neither $f$ nor $\rho_f$ needs to be non-negative anywhere
on their respective domains of definition. 
\end{remark}

\subsection{Proof of Theorem \ref{thm:unique}}
\label{ssec:unique-pf}

Let $f,g$ be two mild solutions of Boltzmann's equation on the given
interval $[0,T]$ (each satisfying the bounds stated in the theorem),
 and consider the difference
\begin{equation}
w = f - g
\end{equation}
The function $w$ satisfies the
difference equation
\begin{equation}
\left( \partial_t + v \cdot \nabla_x \right) w
= Q^+ (f,w) + Q^+ (w,g) - w \rho_f - g \rho_w
\end{equation}
with $w(0) = 0$. Now we apply the lemma to follow (it is not hard
to check that all necessary bounds follow from the hypotheses of
the uniqueness theorem and the fact that $f,g$ solve Boltzmann's
equation with $w$ being their difference).

\begin{lemma}
\label{lem:unlem}
Assume that $f_i$, $i = 1,2,3,4$, satisfy the bounds
\begin{equation}
\left< v \right>^{\frac{1}{2}+} f_i \in L^\infty_{t \in [0,T]} L^2_{x,v}
\bigcap L^\infty_{t \in [0,T]} L^4_x L^2_v
\end{equation}
\begin{equation}
\left< v \right>^{\frac{1}{2}+}
 \left( \partial_t + v \cdot \nabla_x \right) f_i \in
L^1_{t \in [0,T]} L^2_{x,v}
\end{equation}
\begin{equation}
\rho_{f_i} \in L^2_{t \in [0,T]} L^\infty_x \bigcap
L^2_{t \in [0,T]} L^4_x
\end{equation}
Also assume that $w$ is a mild solution of the equation
\begin{equation}
\left( \partial_t + v \cdot \nabla_x \right) w
= Q^+ (f_1, w) + Q^+ ( w, f_2 ) +
w \rho_{f_3} + f_4 \rho_w
\end{equation}
for $t \in [0,T]$, and satisfies the bounds
\begin{equation}
\left< v \right>^{\frac{1}{2}+} w \in L^\infty_{t \in [0,T]} L^2_{x,v}
\end{equation}
\begin{equation}
\left< v \right>^{\frac{1}{2}+}
 \left( \partial_t + v \cdot \nabla_x \right) w \in
L^1_{t \in [0,T]} L^2_{x,v}
\end{equation}
Then if $w(t=0)=0$ then $w \equiv 0$ for $ 0 \leq t \leq T$.
\end{lemma}
\begin{proof}
The bounds imposed on $w$ immediately imply that
$\left< v \right>^{\frac{1}{2}+} w \in C \left(
[0,T], L^2_{x,v}\right)$. 
Let us suppose the conclusion fails and define
\begin{equation}
t_0 = \inf \left\{ t \in [0,T] \left| \left\Vert
w ( t ) \right\Vert_{L^2_{x,v}} > 0 \right. \right\}
\end{equation}
Then $0 \leq t_0 < T$, and $w \equiv 0$ for all
$0 \leq t \leq t_0$ by continuity.

Let us define the error, for $0 \leq s \leq T - t_0$,
\begin{equation}
e_{t_0} (s) = \left\Vert 
\left< v \right>^{\frac{1}{2}+} w
\right\Vert_{L^\infty_{t \in [t_0,t_0+s]} L^2_{x,v}}
+ \left\Vert \left< v
 \right>^{\frac{1}{2}+} \left( \partial_t
+ v \cdot \nabla_x \right) w
\right\Vert_{L^1_{t \in [t_0,t_0+s]} L^2_{x,v}}
\end{equation}
and note that $e_{t_0} (s) < +\infty$ by hypothesis.
We re-write the equation for $w$ as follows:
\begin{equation}
\begin{aligned}
& \left( \partial_t + v \cdot \nabla_x \right) w\\
& \qquad = Q^+ \left( f_1 - T(t-t_0) f_1 ( t_0),w \right) 
+ Q^+ \left( T(t-t_0) f_1 (t_0), w \right) \\
&\qquad+ Q^+ \left( w,f_2 - T(t-t_0) f_2 (t_0) \right)
+ Q^+ \left( w, T(t-t_0) f_2 (t_0) \right) \\
& \qquad + w \rho_{f_3} + f_4 \rho_w
\end{aligned}
\end{equation}
The most dangerous terms are
\begin{equation}
\label{eq:dan1}
 Q^+ \left( T(t-t_0) f_1 (t_0), w \right) 
\end{equation}
and
\begin{equation}
\label{eq:dan2}
Q^+ \left( w, T(t-t_0) f_2 (t_0) \right)
\end{equation}
because a quantitative estimate will always be proportional
to 
$$\left\Vert f_i ( t_0 ) \right\Vert_{L^2_{x,v}} \times e_{t_0} (s)$$
which is not necessarily a small multiple of $e_{t_0} (s)$ (unless
$\left\Vert f_i ( t_0 ) \right\Vert_{L^2_{x,v}}$ is small). We
will address this problem using the short-time estimates from
Proposition \ref{prop:shorttime}.

We will show how to estimate (\ref{eq:dan1}); the alternative
term (\ref{eq:dan2}) is dealt with similarly. To begin, let 
us define
$$ \zeta = \left( \partial_t + v \cdot \nabla_x \right) w $$
then use Duhamel's formula to write
\begin{equation}
w (t) = \int_{t_0}^t T\left(t-\sigma \right)
 \zeta ( \sigma ) d\sigma
\end{equation}
since $w ( t_0 ) = 0$. Due to the bilinearity of $Q^+$, we can
now write
\begin{equation}
\begin{aligned}
& Q^+ \left( T ( t-t_0 ) f_1 ( t_0 ), w \right) =
\int_{t_0}^t Q^+ \left( T ( t-t_0 ) f_1 (t_0), 
T (t-\sigma) \zeta (\sigma) \right) d\sigma \\
& = \int_{t_0}^t Q^+ \left( T ( t-t_0 ) f_1 (t_0), 
T (t-t_0) T ( t_0 - \sigma ) \zeta (\sigma) \right) d\sigma
\end{aligned}
\end{equation}
Now by Minkowski's inequality we have
\begin{equation}
\begin{aligned}
& \left\Vert 
Q^+ \left( T ( t-t_0 ) f_1 ( t_0 ), w \right)
\right\Vert_{L^1_{t \in [t_0, t_0 + s]} L^2_{x,v}}  \leq \\
&  \int_{t_0}^{t_0+s} 
\left\Vert Q^+ \left( T ( t-t_0 ) f_1 (t_0), 
T (t-t_0) T ( t_0 - \sigma ) \zeta ( \sigma ) \right)
\right\Vert_{L^1_{t \in [t_0,t_0 + s]} L^2_{x,v}} d\sigma
\end{aligned}
\end{equation}
Apply Proposition \ref{prop:shorttime} to obtain
\begin{equation}
\begin{aligned}
& \left\Vert 
Q^+ \left( T ( t-t_0 ) f_1 ( t_0 ), w \right)
\right\Vert_{L^1_{t \in [t_0, t_0 + s]} L^2_{x,v}}  \\
& \leq \int_{t_0}^{t_0+s} 
\delta_{f_1 ( t_0 )} ( s ) \left\Vert
 T ( t_0 - \sigma ) \zeta ( \sigma) 
\right\Vert_{L^2_{x,v}} d\sigma \\
& = \int_{t_0}^{t_0+s} 
\delta_{f_1 ( t_0 )} ( s ) \left\Vert
  \zeta ( \sigma) 
\right\Vert_{L^2_{x,v}} d\sigma
=  
\delta_{f_1 ( t_0 )} ( s ) \left\Vert
  \zeta ( t) 
\right\Vert_{L^1_{t \in [t_0,t_0 + s]} L^2_{x,v}}  \\
& \leq \delta_{f_1 ( t_0 )} (s) e_{t_0} (s)
\end{aligned}
\end{equation}
where for each $f_1 (t_0) \in L^2_{x,v}$,
$$ \limsup_{s \rightarrow 0^+} \delta_{f_1 (t_0)} (s) = 0 $$
The same argument can be applied with a weight
$\left< v \right>^{\frac{1}{2}+}$, to yield
\begin{equation}
\left\Vert \left< v \right>^{\frac{1}{2}+} Q^+ \left(
T ( t-t_0) f_1 ( t_0 ), w \right)
\right\Vert_{L^1_{t \in [t_0, t_0 + s ]} L^2_{x,v}} \leq
\tilde{\delta}_{f_1 ( t_0 )} (s) 
e_{t_0} (s)
\end{equation}
where for each $f_1 (t_0)$ with
$\left< v \right>^{\frac{1}{2}+} f_1 (t_0) \in L^2_{x,v}$ there holds
$$ \limsup_{s \rightarrow 0^+}
 \tilde{\delta}_{f_1 (t_0)} (s) = 0 $$

Next we consider the term
\begin{equation}
Q^+ \left( f_1 - T(t-t_0) f_1 (t_0), w \right)
\end{equation}
(the corresponding term involving $Q^+$ and $f_2$ is dealt with
similarly). Here we use Lemma \ref{lem:usefullemma3} to write
\begin{equation*}
\begin{aligned}
& \left\Vert \left< v \right>^{\frac{1}{2}+}
Q^+ \left( f_1 - T(t-t_0) f_1 (t_0), w \right)
\right\Vert_{L^1_{t \in [t_0,t_0+s]} L^2_{x,v}} \\
& \leq C \left( \left\Vert 
\left< v \right>^{\frac{1}{2}+} \left( f_1 - T(t-t_0) f_1 (t_0)
\right)
\right\Vert_{L^\infty_{t \in [t_0,t_0+s]} L^2_{x,v}}\right. \\
& \qquad \qquad \qquad \qquad \qquad \left.
+ \left\Vert \left< v \right>^{\frac{1}{2}+}
\left( \partial_t + v \cdot \nabla_x \right) f_1
\right\Vert_{L^1_{t \in [t_0,t_0+s]} L^2_{x,v}} \right)  \\
& \times\left( \left\Vert \left< v \right>^{\frac{1}{2}+} w
\right\Vert_{L^\infty_{t \in [t_0,t_0+s]} L^2_{x,v}} +
\left\Vert \left< v \right>^{\frac{1}{2}+} \left( \partial_t
+ v \cdot \nabla_x \right) w \right\Vert_{L^1_{t \in [t_0,t_0+s]}
L^2_{x,v}} \right) \\
& \leq C
\left\Vert \left< v \right>^{\frac{1}{2}+}
\left( \partial_t + v \cdot \nabla_x \right) f_1
\right\Vert_{L^1_{t \in [t_0,t_0+s]} L^2_{x,v}}\times
e_{t_0} (s)
\end{aligned}
\end{equation*}

Now let us consider the term
$$ w \rho_{f_3} $$
We have by H{\" o}lder's inequality
\begin{equation*}
\begin{aligned}
 \left\Vert \left< v \right>^{\frac{1}{2}+}
 w \rho_{f_3} \right\Vert_{L^1_{t \in [t_0,t_0+s]}
L^2_{x,v}} & \leq
\left\Vert \left< v \right>^{\frac{1}{2}+} w
\right\Vert_{L^\infty_{t \in [t_0,t_0+s]} L^2_{x,v}}
\left\Vert \rho_{f_3}
\right\Vert_{L^1_{t \in [t_0,t_0+s]} L^\infty_x} \\
& \leq s^{\frac{1}{2}}
\left\Vert \rho_{f_3} \right\Vert_{L^2_{t \in [0,T]}
L^\infty_x} e_{t_0} (s)
\end{aligned}
\end{equation*}

Finally consider the term
$$ f_4 \rho_w. $$
We have by H{\" o}lder's inequality and Lemma
\ref{lem:A1lem},
\begin{equation}
\begin{aligned}
& \left\Vert \left< v \right>^{\frac{1}{2}+} f_4 \rho_w 
\right\Vert_{L^1_{t \in [t_0,t_0+s]} L^2_{x,v}} \leq
\left\Vert \left< v \right>^{\frac{1}{2}+} f_4
\right\Vert_{L^\infty_{t \in [t_0,t_0+s]} L^4_x L^2_v}
\left\Vert \rho_w \right\Vert_{L^1_{t \in [t_0,t_0+s]} L^4_x} \\
& \qquad \qquad \leq
s^{\frac{1}{2}}
\left\Vert \left< v \right>^{\frac{1}{2}+} f_4
\right\Vert_{L^\infty_{t \in [0,T]} L^4_x L^2_v}
\left\Vert \rho_w \right\Vert_{L^2_{t \in [t_0,t_0+s]} L^4_x} \\
& \qquad \qquad \leq
C s^{\frac{1}{2}}
\left\Vert \left< v \right>^{\frac{1}{2}+} f_4
\right\Vert_{L^\infty_{t \in [0,T]} L^4_x L^2_v}
\times e_{t_0} (s)
\end{aligned}
\end{equation}

Altogether we can conclude the following bound:
$$ e_{t_0} (s) \leq c (s) e_{t_0} (s) $$
where
\begin{equation}
\begin{aligned}
c (s) = & C  
\sum_{i=1,2}\left( \tilde{\delta}_{f_i (t_0)} (s) +
\left\Vert \left< v \right>^{\frac{1}{2}+}
\left( \partial_t + v \cdot \nabla_x \right) f_i
\right\Vert_{L^1_{t \in [t_0,t_0+s]} L^2_{x,v}} \right)
\\ 
& + C s^{\frac{1}{2}} \left\Vert \rho_{f_3}
\right\Vert_{L^2_{t \in [0,T]} L^\infty_x}
+ C s^{\frac{1}{2}} \left\Vert
\left< v \right>^{\frac{1}{2}+}
f_4 \right\Vert_{L^\infty_{t \in [0,T]} L^4_x L^2_v}
\end{aligned}
\end{equation}
Clearly $c(s) \rightarrow 0$ as $s \rightarrow 0^+$; hence,
taking $s$ small enough, we shall have $c(s) < 1$. This implies
that $e_{t_0} (s) < e_{t_0} (s)$; since $e_{t_0} (s)$ is finite,
we can conclude that $e_{t_0} (s) = 0$ for some $s>0$ sufficiently
small. This contradicts the definition of $t_0$, so we are done.
\end{proof}

\section{The Kaniel-Shinbrot Iteration}

%{\color{red}{

The problem we encounter in trying to solve Boltzmann's equation is that
we are unable to prove Proposition \ref{prop:kin-gain-bound} with
$Q^-$ in place of $Q^+$. Indeed, it is not even clear whether $Q^- (f,f)$ is
meaningful, in general, when $f$ is a mild solution of the gain-only equation obtained from
Theorem \ref{thm:gain-eq-gwp-kin}. On the other hand, it is definitely possible to
solve uniquely the full Boltzmann equation (with constant collision kernel in $d=2$)
\emph{locally in time} if we assume:
\begin{equation*}
\left< v \right>^{\frac{1}{2}+}  \left<
\nabla_x \right>^{\frac{1}{2}+}  f_0 \in
L^2_{x,v} \left( \mathbb{R}^2 \times \mathbb{R}^2 \right)
\end{equation*}
The challenge, therefore, is to propagate sufficient regularity 
for the gain-only equation, \emph{assuming a smallness condition only for
the $L^2_{x,v}$ norm}. To this end, we will need to 
employ a small parameter $\varepsilon \in \left( 0,1 \right]$ to
encode the fact that higher derivatives may be much larger than
the $L^2_{x,v}$ norm of $f_0$.

We proceed by first establishing regularity of the gain-only equation in Section \ref{sec:gainReg}.
Then, in Section \ref{sec:KanielShinbrot}, we present a novel application of the iterative method of Kaniel-Shinbrot 
to establish existence of global solution to the Boltzmann equation.

%}}

\subsection{Regularity for the Gain-Only Equation}
\label{sec:gainReg}

\begin{theorem}
\label{thm:gain-eq-gwp-kin-alpha} 
There exists a number $\eta \in \left( 0, 1 \right)$ such that all the following is true: \\
\begin{enumerate}[label=(\roman*)]
\item For any $f_0 \in H^{\alpha,\alpha}$, $\alpha \in (0,1)$,
  with $\left\Vert f_0 \right\Vert_{L^2_{x,v}} < \eta$,
 there exists a unique global ($t \in \mathbb{R}$) mild solution to the
\emph{gain-only} Boltzmann equation
\begin{equation}
\left( \partial_t + v \cdot \nabla_x \right)
f(t) = Q^+ \left( f(t), f(t) \right)
\end{equation}
with $f(0) = f_0$ such that $f \in C_{t,\textnormal{loc}}
 H^{\alpha,\alpha}$ and $Q^+ (f,f) \in L^1_{t,\textnormal{loc}}
H^{\alpha,\alpha}$. For this solution, it holds that
$f \in L^\infty_t H^{\alpha,\alpha}$
 and $Q^+ (f,f) \in L^1_t H^{\alpha,\alpha}$, and the solution
scatters in $H^{\alpha,\alpha}$ as $t \rightarrow \pm \infty$. 

\item For any $f_0 \in H^{\alpha,\alpha}$ with 
$\left\Vert f_0 \right\Vert_{L^2_{x,v}} < \eta$,
 we have the following estimate:
\begin{equation}
\label{eq:L2globest}
\left\Vert f \right\Vert_{L^\infty_{t \in \mathbb{R}} L^2_{x,v}}^2
+ \left\Vert Q^+ \left( f,f \right)
\right\Vert_{L^1_{t \in \mathbb{R}} L^2_{x,v}}
\leq C \left\Vert f_0 \right\Vert_{L^2_{x,v}}^2
\end{equation}
for the solution $f$ of the gain-only Boltzmann equation (note, this bound only depends
on the $L^2_{x,v}$ norm of $f_0$). Also, if $f_0 (x,v) \geq 0$
$\textnormal{a.e.}-(x,v)$ then
$f (t,x,v) \geq 0$ for $\textnormal{a.e.}-(t,x,v)$ such that $t \geq 0$.

\item If $\alpha > \frac{1}{2}$, then we have
$\left< v \right>^{\frac{1}{2}+} f
 \in L^\infty_t L^4_x L^2_v$ and $\rho_f \in L^2_t L^\infty_x
\bigcap L^2_t L^4_x$. Combining these estimates, the loss term
$Q^- (f,f) = \rho_f f$ (although \emph{not} appearing in the equation for $f$)
satisfies 
$$
\left< v \right>^{\frac{1}{2}+}
 Q^- (f,f) \in L^2_{t \in \mathbb{R}} L^2_{x,v}.
 $$
\end{enumerate}
\end{theorem}
\begin{proof}
Parts \emph{(i)} and \emph{(ii)}
 are direct consequences of Proposition \ref{prop:gain-eq-bound-kin-alpha},
combined with Theorem \ref{thm:abswp} taking
$\mathcal{H} = H^{\alpha,\alpha}_\varepsilon$ where
$\varepsilon = \varepsilon \left( f_0 \right)$ is sufficiently small; 
 here we have used
the fact that the constant
$C$ in Proposition \ref{prop:gain-eq-bound-kin-alpha} does not
depend on $\varepsilon$.
Note that $L^2_{x,v} \subset H^{\alpha,\alpha}$, so
the uniqueness in $L^2_{x,v}$ implies that $L^2_{x,v}$ and
$H^{\alpha,\alpha}$ 
solutions coincide globally in time (as long as the
$L^2_{x,v}$ norm of $f_0$ is small enough).

For part \emph{(iii)}, to see
that $\left< v \right>^{\frac{1}{2}+}
  f \in L^\infty_t L^4_x L^2_v$, we may
 observe that $\left< v \right>^{\frac{1}{2}+}
\left< \nabla_x \right>^{\frac{1}{2}+}
 f \in L^\infty_t L^2_{x,v}$ and apply
the Sobolev embedding theorem in the $x$ variable. 
On the other hand, the estimate $\rho_f \in L^2_t L^\infty_x 
\bigcap L^2_t L^4_x$
follows directly from Lemma \ref{lem:A1lem} and the estimates
from part \emph{(i)}.  The
estimate on $Q^- (f,f)$ then follows from H{\" o}lder's inequality.
Note that, contrary to part \emph{(ii)},
 all the bounds from part \emph{(iii)} depend explicitly on 
the $H^{\alpha,\alpha}$ norm of $f_0$.
\end{proof}

\subsection{The full equation via Kaniel-Shinbrot iteration}
\label{sec:KanielShinbrot}

The iteration of Kaniel and Shinbrot constructs a decreasing sequence
$g_n (t,x,v)$ and an increasing sequence $h_n (t,x,v)$ with
$0 \leq h_n \leq g_n$. The goal is to show that $\lim_n g_n = \lim_n h_n = f$,
with $f$ being a solution of the full Boltzmann equation. One can view
the functions $g_n, h_n$ as being ``barriers'' which progressively limit
the possible oscillation of $f$, until eventually there is no room
left in which to wiggle.

Recall the convenient notation
\begin{equation}
\rho_f (x) = \int_{\mathbb{R}^2} f (x,v) dv
\end{equation}
The iteration is as follows:
\begin{equation*}
\begin{aligned}
& \left( \partial_t + v \cdot \nabla_x + \rho_{g_n}\right) h_{n+1}
= Q^+ \left( h_n, h_n \right) \\
&  \; \left( \partial_t + v \cdot \nabla_x + \rho_{h_n} \right) g_{n+1}
= Q^+ \left( g_n, g_n \right) \\
& \;\;\; g_{n+1} (0) = h_{n+1} (0) = f_0
\end{aligned}
\end{equation*}
For each $n$, observe that we are simply solving \emph{linear} differential
equations (with the initial data always fixed at $f_0$),
so the existence of the iteration is typically not a big problem.
It is possible to show, using monotonicity, that if
\begin{equation}
0 \leq h_{n-1} \leq h_n \leq g_n \leq g_{n-1}
\end{equation}
holds globally, then
\begin{equation} \label{eq:mon}
0 \leq h_n \leq h_{n+1} \leq g_{n+1} \leq g_n
\end{equation}
Hence, in order to exploit monotonicity, we must at least have
\begin{equation}
0 \leq h_1 \leq h_2 \leq g_2 \leq g_1
\end{equation}
where
\begin{equation} \label{eq:g2h2}
\begin{aligned}
& \left( \partial_t + v \cdot \nabla_x + \rho_{g_1}\right) h_{2}
= Q^+ \left( h_1, h_1 \right) \\
&  \; \left( \partial_t + v \cdot \nabla_x + \rho_{h_1} \right) g_{2}
= Q^+ \left( g_1, g_1 \right) \\
& \;\;\; g_{2} (0) = h_{2} (0) = f_0
\end{aligned}
\end{equation}
and this is the so-called \emph{beginning condition} (note that no initial conditions
are imposed for $(h_1,g_1)$). Note that the beginning
condition has to be verified \emph{for all time} (or at least on the full
time interval for which the iteration is to be employed). For this reason, establishing
the beginning condition is considered the most difficult part of the Kaniel-Shinbrot iteration.

%{\color{red}{

We choose  $h_1$ as follows 
$$h_1 \equiv 0,$$ 
and we choose $g_1$ to solve the gain only equation  
\begin{equation} \label{eq:g1}
\left( \partial_t + v \cdot \nabla_x \right) g_1 = Q^+ \left( g_1, g_1 \right),\qquad
\qquad g_1 (0) = f_0.
\end{equation}
Then we compute $h_2$ and $g_2$ according to \eqref{eq:g2h2} to obtain 
\begin{equation} \label{eq:h2}
h_2(t) = f_0 e^{-\int_0^t \rho_{g_1}(\tau) d\tau} 
\end{equation}
and 
\begin{equation} \label{eq:g2}
g_2(t) = f_0  + \int_0^t Q^{+}(g_1,g_1)(\tau) \; d\tau. 
\end{equation}
Therefore the condition 
\begin{equation} \label{eq:beg1}
0 \leq h_1(t)  \leq h_2(t) \leq g_2(t),
\end{equation} 
is satisfied for all $t \geq 0$. 
On the other hand, since $h_1 \equiv 0$ we see from \eqref{eq:g2h2} and \eqref{eq:g1} that $g_2$ and $g_1$ solve the same initial value problem. 
Therefore 
\begin{equation} \label{eq:beg2}
g_2(t) =  g_1(t)
\end{equation} 
for all $t \geq 0$ for which we can make sense of the gain only equation. 
We conclude that for our choice of $h_1$ and $g_1$, the beginning condition follows \eqref{eq:beg1} and \eqref{eq:beg2}. 
%}}

Since all the $g_n,h_n$ are bounded
by $g_1$, under the conditions of 
Theorem \ref{thm:gain-eq-gwp-kin-alpha} 
with $f_0 \in H^{\frac{1}{2}+,\frac{1}{2}+}$ 
we automatically have
\begin{equation*}
\sup_n \left\Vert h_n \right\Vert_{L^\infty_t L^2_{x,v}} \leq
\sup_n \left\Vert g_n \right\Vert_{L^\infty_t L^2_{x,v}} < \infty 
\end{equation*}
\begin{equation*}
\sup_n \left\Vert Q^+ \left( h_n, h_n \right) \right\Vert_{L^1_{t \geq 0} L^2_{x,v}}
\leq \sup_n \left\Vert Q^+ \left( g_n, g_n \right) \right\Vert_{L^1_{t \geq 0} L^2_{x,v}}
< \infty
\end{equation*}
\begin{equation*}
\sup_n \left\Vert Q^- \left( h_n, h_n \right)
\right\Vert_{L^1_{t \in [0,T]} L^2_{x,v}}
\leq \sup_n \left\Vert Q^- \left( g_n, g_n \right) \right\Vert_{L^1_{t \in [0,T]} L^2_{x,v}}
< \infty
\end{equation*}
assuming the iteration makes sense. Moreover, since the functions $h_n$ are
increasing and the $g_n$ are decreasing, we can define their pointwise limits
\begin{equation*}
 g = \lim_n g_n \qquad \qquad h = \lim_n h_n
\end{equation*}
Since $0 \leq h_n \leq g_n \leq g_1$,
 and $Q^\pm (g_1,g_1) \in \left( L^1_{t,x,v} \right)_{\textnormal{loc}}$,
an easy application of the dominated convergence theorem shows that
\begin{equation*}
Q^\pm (h_n,h_n) \rightarrow Q^\pm (h,h) \qquad \qquad Q^\pm (g_n,g_n) 
\rightarrow Q^\pm (g,g)
\end{equation*}
in the sense of distributions. Mixed terms such as $Q^- (h_n,g_n)$ are handled similarly.
Altogether we conclude that the limits $g,h$ satisfy
\begin{equation*}
\begin{aligned}
& \left( \partial_t + v \cdot \nabla_x + \rho_g \right) h
= Q^+ \left( h,h \right) \\
& \left( \partial_t + v \cdot \nabla_x + \rho_h \right) g
= Q^+ \left( g,g \right) \\
& g(0) = h(0) = f_0
\end{aligned}
\end{equation*}
in the sense of distributions.

We have yet to show that $h = g$ in order to conclude the convergence of the
Kaniel-Shinbrot iteration. Let us define
\begin{equation}
w (t,x,v) = g (t,x,v) - h(t,x,v) \geq 0
\end{equation}
and note that $w \leq g_1$. The function $w$ satisfies the following equation
in the sense of distributions:
\begin{equation*}
\left( \partial_t + v \cdot \nabla_x \right) w
= Q^+ \left( g,w \right) + Q^+ \left( w,h \right)
+ \rho_w h - \rho_h w
\end{equation*}
\begin{equation*}
w (0) = 0
\end{equation*}
The goal is to show that $w=0$ globally in $t\geq 0$. This follows
from Lemma \ref{lem:unlem} as long as we can show
\begin{equation}
\left< v \right>^{\frac{1}{2}+} Q^+ (g_1,g_1) \in
L^1_{t \in [0,T]} L^2_{x,v}
\end{equation}
\begin{equation}
\rho_{g_1} \in L^2_{t \in [0,T]} L^\infty_x \bigcap
L^2_{t \in [0,T]} L^4_x
\end{equation}
\begin{equation}
\left< v \right>^{\frac{1}{2}+} g_1 \in
L^\infty_{t \in [0,T]} L^2_{x,v}
\bigcap L^\infty_{t \in [0,T]} L^4_x L^2_v
\end{equation}
but these bounds follow from Theorem \ref{thm:gain-eq-gwp-kin-alpha}
since we assume $f_0 \in H^{\frac{1}{2}+,\frac{1}{2}+}$.
We can conclude that the Kaniel-Shinbrot iteration converges to
a solution of Boltzmann's equation.

As a final crucial remark, let us note that since $0 \leq f \leq g_1$ (by
construction), and $f_0 \in H^{\frac{1}{2}+,\frac{1}{2}+}$,
by Theorem \ref{thm:gain-eq-gwp-kin-alpha} we have the following
estimates for the \emph{full Boltzmann equation} with small $L^2_{x,v}$ norm:
\begin{equation*}
\left< v \right>^{\frac{1}{2}+}  Q^+ (f,f) \in L^1_t L^2_{x,v}
\end{equation*}
\begin{equation*}
\left< v \right>^{\frac{1}{2}+}  Q^- (f,f) \in L^2_t L^2_{x,v}
\end{equation*}
\begin{equation*}
\left< v \right>^{\frac{1}{2}+}  f \in 
L^\infty_t L^2_{x,v} \bigcap L^\infty_t L^4_x L^2_v
\end{equation*}
\begin{equation*}
\rho_f \in L^2_t L^\infty_x \bigcap L^2_t L^4_x
\end{equation*}
\begin{equation*}
\left\Vert f \right\Vert_{L^\infty_t L^2_{x,v}}^2 +
\left\Vert Q^+ (f,f) \right\Vert_{L^1_t L^2_{x,v}} \leq C
\left\Vert f_0 \right\Vert_{L^2_{x,v}}^2
\end{equation*}
Let us emphasize that we have \emph{not} established that
$f \in L^\infty_{t \in [0,T]} H^{\frac{1}{2}+,\frac{1}{2}+}$
 so it is not valid to
apply Lemma \ref{lem:A1lem} directly to the solution $f$ of Boltzmann's equation
in order to deduce that $\rho_f \in L^2_{t \in [0,T]} L^\infty_x$. Rather,
we are using the fact that $ 0 \leq \rho_f \leq \rho_{g_1}$ combined with
the propagation of regularity for the gain-only equation,
$g_1 \in L^\infty_t H^{\frac{1}{2}+,\frac{1}{2}+}$, and applying
Lemma \ref{lem:A1lem} to $g_1$. Indeed, to obtain the best possible bounds,
we are required to convert all regularity information on $g_1$ into
\emph{integrability} information via the Sobolev embedding, at which point
it becomes useful information for the solution $f$ of the full Boltzmann equation.
This is a strange situation because we are using the regularity condition
$f_0 \in H^{\frac{1}{2}+,\frac{1}{2}+}$
 to construct global solutions $f(t)$ for which
\emph{a priori} the $H^{\frac{1}{2}+,\frac{1}{2}+}$
 norm could blow up to $+\infty$ in finite time
(we will show later by an independent   argument
that this blow-up scenario cannot happen).

\section{Scattering in $L^2_{x,v}$}
\label{sec:scatteringL2}

The idea is to use the non-negativity of $f$ in a rather strong way. We can write the
solution of Boltzmann's equation as follows:
\begin{equation}
T(-t) f(t) + \int_0^t T(-\sigma) Q^- \left( f(\sigma), f(\sigma)\right) d\sigma
= f_0 + \int_0^t T(-\sigma) Q^+ \left( f(\sigma),f(\sigma) \right) d\sigma
\end{equation}
Everything on either side is non-negative (we are assuming $t \geq 0$), so we can write
\begin{equation}
 \int_0^t T(-\sigma) Q^- \left( f(\sigma), f(\sigma)\right) d\sigma
\leq f_0 + \int_0^t T(-\sigma) Q^+ \left( f(\sigma),f(\sigma) \right) d\sigma
\end{equation}
which implies
\begin{equation}
 \int_0^t T(-\sigma) Q^- \left( f(\sigma), f(\sigma)\right) d\sigma
\leq f_0 + \int_0^\infty T(-\sigma) Q^+ \left( f(\sigma),f(\sigma) \right) d\sigma
\end{equation}
Then by monotone convergence in $t$, for almost every $(x,v)$ we have
\begin{equation}
 \int_0^\infty T(-\sigma) Q^- \left( f(\sigma), f(\sigma)\right) d\sigma
\leq f_0 + \int_0^\infty T(-\sigma) Q^+ \left( f(\sigma),f(\sigma) \right) d\sigma
\end{equation}
Taking the $L^2_{x,v}$ norm of both sides and applying Minkowski 
on the \emph{right hand side} only, and using the fact that $T(t)$ preserves
$L^2_{x,v}$, we obtain:
\begin{equation}
\label{eq:QminusSbd}
\left\Vert T(-t) Q^- \left( f(t),f(t) \right) \right\Vert_{L^2_{x,v} L^1_{t\geq 0}}
\leq \left\Vert f_0 \right\Vert_{L^2_{x,v}}
+ \left\Vert Q^+ \left( f(t), f(t) \right) \right\Vert_{L^1_{t\geq 0} L^2_{x,v}}
\end{equation}
We have
\begin{equation}
\label{eq:QplusSbd}
Q^+ \left( f(t),f(t) \right) \in L^1_{t\geq 0} L^2_{x,v}
\end{equation}
because (\ref{eq:QplusSbd}) holds for the solution of the
 gain-only Boltzmann equation (with small
data $f_0 \in L^2_{x,v}$), and the solution of the full Boltzmann equation
is bounded above by the solution of the gain-only Boltzmann equation as a result
of the Kaniel-Shinbrot construction.

We can combine (\ref{eq:QminusSbd}) and (\ref{eq:QplusSbd}) to conclude
\begin{equation}
T(-t) Q^{\pm} \left( f(t), f(t) \right) \in L^2_{x,v} L^1_{t\geq 0}
\end{equation}
and this implies that the limit in norm
\begin{equation}
\lim_{t \rightarrow + \infty} \int_0^t
T(-\sigma) Q \left( f(\sigma),f(\sigma) \right) d\sigma
\end{equation}
exists in $L^2_{x,v}$, by the dominated convergence theorem. Indeed, the
$L^2_{x,v}$ remainder is bounded by
\begin{equation}
\int dx dv \int_t^\infty d\sigma \int_t^\infty
d\sigma^\prime
\left\{ T ( -\sigma) \left| Q \left( f(\sigma),f(\sigma)\right) \right|\right\}
\left\{ T (-\sigma^\prime) \left| Q \left( f(\sigma^\prime),f(\sigma^\prime)\right)
\right| \right\}
\end{equation}
and this clearly tends to zero as $t \rightarrow + \infty$.

As a result of the convergence argument detailed above, if we define
\begin{equation}
f_{+ \infty} = f_0 + 
\lim_{t \rightarrow + \infty} \int_0^t
T(-\sigma) Q \left( f(\sigma),f(\sigma) \right) d\sigma
\end{equation}
then it follows that $f_{+ \infty} \in L^2_{x,v}$ and
\begin{equation}
\lim_{t \rightarrow + \infty}
\left\Vert T(-t) f(t) - f_{+ \infty} \right\Vert_{L^2_{x,v}} = 0
\end{equation}
The same argument implies the following slightly more general result
(which does not require uniqueness, nor that $f_0$ necessarily have
small $L^2_{x,v}$ norm):
\begin{theorem}
\label{thm:scatteringL2}
Suppose $f \in \bigcap_{T > 0} L^\infty_{t \in [0,T]} L^2_{x,v}$
is a \emph{non-negative} mild solution of the full Boltzmann equation,
\begin{equation}
\left( \partial_t + v \cdot \nabla_x \right) f
= Q^+ \left( f,f \right) - f \rho_f
\end{equation}
such that, along the solution $f (t)$, the \emph{gain operator} $Q^+$ satisfies
\begin{equation}
T(-t) Q^+ \left( f(t),f(t) \right) \in L^2_{x,v} L^1_{t \geq 0}
\end{equation}
Then $f(t)$ scatters in $L^2_{x,v}$ as $t \rightarrow +\infty$;
that is, there exists a function
$f_{+ \infty} \in L^2_{x,v}$ such that the following limit
\begin{equation}
\lim_{t \rightarrow +\infty}
\left\Vert  f (t) - T(t) f_{+ \infty} \right\Vert_{L^2_{x,v}} = 0
\end{equation}
holds.
\end{theorem}

\begin{remark}
The gain-only Boltzmann equation scatters in 
$H^{\frac{1}{2}+,\frac{1}{2}+}$, assuming
only that $f_0 \in H^{\frac{1}{2}+,\frac{1}{2}+} \cap B_\eta^{L^2}$;
 of course, this implies that solutions of the gain-only
equation remain uniformly bounded in $H^{\frac{1}{2}+,\frac{1}{2}+}$
 as $t \rightarrow +\infty$.
However, we do not know whether the full Boltzmann
equation scatters in $H^{\frac{1}{2}+,\frac{1}{2}+}$; indeed, whereas we show in Section \ref{sec:reg}
that the solution of the full Boltzmann equation propagates 
$H^{\frac{1}{2}+,\frac{1}{2}+}$ for
small $L^2_{x,v}$ solutions, we do not even know whether the 
$H^{\frac{1}{2}+,\frac{1}{2}+}$ norm
(for the full Boltzmann equation)
remains \emph{bounded} in time as $t \rightarrow +\infty$.
\end{remark}

\begin{remark}
Due to the lack of $L^1_t L^2_{x,v}$ \emph{bilinear} spacetime estimates for
$Q^- (f,f)$, we cannot use Theorem \ref{thm:absScattering} (or its proof) to describe qualitatively the
correspondence between $f_0$ and $f_{+ \infty}$ for the full Boltzmann equation (though Theorem
\ref{thm:absScattering} clearly applies to the \emph{gain-only} equation).
\end{remark}

\section{Propagation of Regularity for the full equation}
\label{sec:reg}

Recall that some extra regularity for the \emph{gain-only} equation 
was required to produce enough \emph{integrability}
to close the Kaniel-Shinbrot iteration and prove uniqueness. However, so far we have said nothing
about the regularity of the full Boltzmann equation. The point of this section is to prove that, for
all the regularity which we required to construct a solution, such regularity is indeed propagated
by the solution itself.

\begin{remark}
It is important to observe that it is \emph{not necessary} to propagate regularity for the full
Boltzmann equation in order to close the Kaniel-Shinbrot iteration. Thus, the regularity for the
full equation is propagated \emph{a posteriori}.
\end{remark}

\subsection{Loss operator bounds.}
\label{sec:reg-loss}

Recall the loss operator
\begin{equation}
Q^- \left( f, g \right) = f \rho_g
\end{equation}

\begin{lemma}
\label{lem:LossBilin}
Let $\alpha \in \left( \frac{1}{2}, 1 \right]$.
For any two measurable and locally integrable functions $f_0 (x,v),\; g_0 (x,v)$ such
that $\left< v \right>^\alpha \left< \nabla_x \right>^\alpha
f_0,\; \left< v \right>^\alpha \left< \nabla_x \right>^\alpha
g_0 \in L^2_{x,v}$, the function $Q^- \left( T(t) f_0, T(t) g_0 \right)$ is
in $L^2_{t,x,v} \left( \mathbb{R} \times \mathbb{R}^2
\times \mathbb{R}^2 \right)$ and the following estimate
holds:
\begin{equation}
\begin{aligned}
& \left\Vert \left< v \right>^\alpha
\left< \nabla_x \right>^\alpha 
Q^- \left( T(t) f_0, T(t) g_0 \right)
\right\Vert_{L^2_{t,x,v}} \leq \\
& \qquad \qquad \qquad \qquad \qquad
\leq  C \left\Vert \left< v \right>^\alpha
\left< \nabla_x \right>^\alpha f_0 \right\Vert_{L^2_{x,v}}
\left\Vert \left< v \right>^\alpha
\left< \nabla_x \right>^\alpha g_0 \right\Vert_{L^2_{x,v}}
\end{aligned}
\end{equation}
\end{lemma}
\begin{proof}
We will assume $\alpha \in \left( \frac{1}{2}, 1 \right)$; the case $\alpha = 1$
follows in a similar manner by using the Leibniz differentiation rule
(note that $H^1_x = L^2_x \cap \dot{H}^1_x$, 
and that $\left| \nabla_x \right|$ can be replaced by $\nabla_x$ in defining
the $\dot{H}^1_x$ semi-norm).

We begin with the $L^2_x$ estimate. We have
\begin{equation*}
\begin{aligned}
 \left\Vert \left< v \right>^\alpha Q^-
\left( T(t) f_0, T(t) g_0 \right) \right\Vert_{L^2_{t,x,v}}
& = \left\Vert \left< v \right>^\alpha
\left\{ T(t) f_0 \right\} \rho_{T(t) g_0} \right\Vert_{L^2_{t,x,v}} 
 \\
& \leq \left\Vert \left< v \right>^\alpha T(t) f_0 \right\Vert_{L^\infty_t L^2_{x,v}}
\left\Vert \rho_{T(t) g_0} \right\Vert_{L^2_t L^\infty_x} \\
& \leq \left\Vert \left< v \right>^\alpha f_0 \right\Vert_{L^2_{x,v}}
\left\Vert \left< v \right>^\alpha \left< \nabla_x \right>^\alpha
g_0 \right\Vert_{L^2_{x,v}} \\
& \leq \left\Vert \left< v \right>^\alpha
\left< \nabla_x \right>^\alpha f_0 \right\Vert_{L^2_{x,v}}
\left\Vert \left< v \right>^\alpha \left< \nabla_x \right>^\alpha
g_0 \right\Vert_{L^2_{x,v}}
\end{aligned}
\end{equation*}
where we have used that $\alpha > \frac{1}{2}$ in order
to apply Lemma \ref{lem:A1lem}.

Let us now turn to the $\dot{H}^\alpha_x$ estimate; by Theorem \ref{thm:FracLeibKPV}
(due to Kenig-Ponce-Vega \cite{KPV1993})
we have
\begin{equation*}
\begin{aligned}
& \left\Vert \left< v \right>^\alpha 
\left| \nabla_x \right|^\alpha Q^-
\left( T(t) f_0, T(t) g_0 \right) \right\Vert_{L^2_x}
 = \left\Vert 
\left| \nabla_x \right|^\alpha \Big( 
\left\{ \left< v \right>^\alpha
T(t) f_0 \right\} \rho_{T(t) g_0} \Big)\right\Vert_{L^2_x} 
 \\
& \qquad \qquad \qquad \leq 
\left\Vert \Big( \left< v \right>^\alpha T(t) f_0 \Big)
\left| \nabla_x \right|^\alpha
\rho_{T(t) g_0} \right\Vert_{L^2_x} + \\
& \qquad \qquad \qquad \qquad
+ C \left\Vert \rho_{T(t) g_0} \right\Vert_{L^\infty_x}
\left\Vert \left< v \right>^\alpha
\left| \nabla_x \right|^\alpha T(t) f_0
\right\Vert_{L^2_x}.
\end{aligned}
\end{equation*}
Now we take the $L^2_{t,v}$ norm of both sides, and then apply H{\" o}lder's inequality
and Lemma \ref{lem:A1lem} (which is justified because $\alpha > \frac{1}{2}$).
\begin{equation*}
\begin{aligned}
& \left\Vert \left< v \right>^\alpha 
\left| \nabla_x \right|^\alpha Q^-
\left( T(t) f_0, T(t) g_0 \right) \right\Vert_{L^2_{t,x,v}} \leq \\
& \qquad \qquad \qquad \leq 
\left\Vert \Big( \left< v \right>^\alpha T(t) f_0 \Big)
\left| \nabla_x \right|^\alpha
\rho_{T(t) g_0} \right\Vert_{L^2_{t,x,v}} + \\
& \qquad \qquad \qquad \qquad \qquad
+ C \Bigg\Vert \left\Vert \rho_{T(t) g_0} \right\Vert_{L^\infty_x}
\left\Vert \left< v \right>^\alpha
\left| \nabla_x \right|^\alpha T(t) f_0
\right\Vert_{L^2_x} \Bigg\Vert_{L^2_{t,v}} \\
& \qquad \qquad \qquad \leq 
\left\Vert \left< v \right>^\alpha T(t) f_0 \right\Vert_{L^\infty_t L^4_x L^2_v}
\left\Vert \left| \nabla_x \right|^\alpha
\rho_{T(t) g_0} \right\Vert_{L^2_t L^4_x} + \\
& \qquad \qquad \qquad \qquad \qquad
+ C \left\Vert \rho_{T(t) g_0} \right\Vert_{L^2_t L^\infty_x}
\left\Vert \left< v \right>^\alpha
\left| \nabla_x \right|^\alpha T(t) f_0
\right\Vert_{L^\infty_t L^2_{x,v}}  \\
& \qquad \qquad \qquad \leq 
\left\Vert \left< v \right>^\alpha T(t) f_0 \right\Vert_{L^\infty_t L^4_x L^2_v}
\left\Vert \left< v \right>^\alpha \left| \nabla_x \right|^\alpha
g_0 \right\Vert_{L^2_{x,v}} + \\
& \qquad \qquad \qquad \qquad \qquad
+ C \left\Vert \left< v \right>^\alpha \left< \nabla_x \right>^\alpha
g_0 \right\Vert_{L^2_{x,v}}
\left\Vert \left< v \right>^\alpha
\left| \nabla_x \right|^\alpha T(t) f_0
\right\Vert_{L^\infty_t L^2_{x,v}}  \\
& \qquad \qquad \qquad \leq 
\left\Vert \left< v \right>^\alpha
\left< \nabla_x \right>^\alpha T(t) f_0 \right\Vert_{L^\infty_t L^2_{x,v}}
\left\Vert \left< v \right>^\alpha \left< \nabla_x \right>^\alpha
g_0 \right\Vert_{L^2_{x,v}} + \\
& \qquad \qquad \qquad \qquad \qquad
+ C \left\Vert \left< v \right>^\alpha \left< \nabla_x \right>^\alpha
g_0 \right\Vert_{L^2_{x,v}}
\left\Vert \left< v \right>^\alpha
\left< \nabla_x \right>^\alpha T(t) f_0
\right\Vert_{L^\infty_t L^2_{x,v}}  \\
\end{aligned}
\end{equation*}
Note that $\left| \nabla_x \right|$ commutes with $\rho_{(\cdot)}$, and we
have used the Sobolev embedding
$H^{\frac{1}{2}}_x \left( \mathbb{R}^2 \right)
\subset L^4_x \left( \mathbb{R}^2 \right)$ in the last step. We finally use
the fact that $T(t)$ preserves $H^{\alpha,\beta}$ to obtain:
\begin{equation*}
\begin{aligned}
& \left\Vert \left< v \right>^\alpha 
\left| \nabla_x \right|^\alpha Q^-
\left( T(t) f_0, T(t) g_0 \right) \right\Vert_{L^2_{t,x,v}}
\leq \\
& \qquad \qquad \qquad \qquad \qquad
\leq C \left\Vert \left< v \right>^\alpha \left< \nabla_x \right>^\alpha
f_0 \right\Vert_{L^2_{x,v}}
\left\Vert \left< v \right>^\alpha \left< \nabla_x \right>^\alpha g_0
\right\Vert_{L^2_{x,v}}
\end{aligned}
\end{equation*}
Combining the $L^2_x$ and $\dot{H}^\alpha_x$ estimates allows us to conclude.
\end{proof}

The next lemma is a refinement of Lemma \ref{lem:LossBilin} which only
places a spatial gradient on one argument at a time.

\begin{lemma}
\label{lem:LossBilinUnsym}
Let $\alpha \in \left( \frac{1}{2}, 1 \right]$, and let
$I \subseteq \mathbb{R}$ be an open interval (either bounded or unbounded).
Let $f (t,x,v) : I \times \mathbb{R}^2 \times \mathbb{R}^2 \rightarrow \mathbb{C}$ be
a measurable and locally integrable function such that
\begin{equation}
\left< v \right>^\alpha
\left< \nabla_x \right>^\alpha f \in L^\infty \left( I, L^2_{x,v} \right)
\end{equation}
and
\begin{equation}
\left< v \right>^\alpha \left< \nabla_x \right>^\alpha
\left( \partial_t + v \cdot \nabla_x \right) f
\in L^1 \left( I, L^2_{x,v} \right)
\end{equation}
Then the following estimate holds:
\begin{equation}
\begin{aligned}
& \left\Vert \left< v \right>^\alpha
\left| \nabla_x \right|^\alpha Q^- (f,f)
\right\Vert_{L^2 \left( I, L^2_{x,v} \right)} \leq \\
& \qquad \qquad \leq C \times \left\{
\left\Vert \rho_f \right\Vert_{L^2 \left( I, L^\infty_x \right)}
\left\Vert \left< v \right>^\alpha
\left| \nabla_x \right|^\alpha f \right\Vert_{L^\infty
\left( I, L^2_{x,v}\right)} + \right. \\
& \qquad \qquad \quad \left. +
\left\Vert \left< v \right>^\alpha f
\right\Vert_{L^\infty \left( I, L^4_x L^2_v \right)}
\left\Vert \left< v \right>^\alpha \left| \nabla_x \right|^\alpha
f \right\Vert_{L^\infty \left( I, L^2_{x,v} \right)} + \right. \\ 
& \qquad \qquad \quad \left. +
\left\Vert \left< v \right>^\alpha f
\right\Vert_{L^\infty \left( I, L^4_x L^2_v \right)}
\left\Vert \left< v \right>^\alpha 
\left| \nabla_x \right|^\alpha
\left( \partial_t + v \cdot \nabla_x \right) f
\right\Vert_{L^1 \left( I, L^2_{x,v} \right)} \right\}
\end{aligned}
\end{equation}
The constant $C$ does not depend on the interval $I$, but it may depend
on $\alpha$.
\end{lemma}
\begin{proof}
As in the proof of Lemma \ref{lem:LossBilin}, we will assume
$\alpha \in \left( \frac{1}{2}, 1 \right)$. The case $\alpha = 1$ may be
checked directly in a similar fashion.

We begin by applying Theorem \ref{thm:FracLeibKPV}, which is due
to Kenig-Ponce-Vega \cite{KPV1993}:
\begin{equation*}
\begin{aligned}
& \left\Vert \left< v \right>^\alpha
\left| \nabla_x \right|^\alpha Q^- (f,f) \right\Vert_{L^2_x} 
= \left\Vert \left| \nabla_x \right|^\alpha
\bigg( \left< v \right>^\alpha f \rho_f \bigg)
\right\Vert_{L^2_x} \\
& \qquad \qquad \leq
\left\Vert \left< v \right>^\alpha f \left| \nabla_x \right|^\alpha \rho_f
\right\Vert_{L^2_x} + C
\left\Vert \rho_f \right\Vert_{L^\infty_x}
\left\Vert \left< v \right>^\alpha
\left| \nabla_x \right|^\alpha f \right\Vert_{L^2_x}
\end{aligned}
\end{equation*}
We take the $L^2_{t \in I} L^2_v$ norm of both sides,
followed by H{\" o}lder's inequality:
\begin{equation*}
\begin{aligned}
& \left\Vert \left< v \right>^\alpha
\left| \nabla_x \right|^\alpha Q^- (f,f) \right\Vert_{L^2_{t \in I} L^2_{x,v}} 
\leq \\
& \qquad \quad \leq
\left\Vert \left< v \right>^\alpha f \left| \nabla_x \right|^\alpha \rho_f
\right\Vert_{L^2_{t \in I} L^2_{x,v}} + C
\bigg\Vert \left\Vert \rho_f \right\Vert_{L^\infty_x}
\left\Vert \left< v \right>^\alpha
\left| \nabla_x \right|^\alpha f \right\Vert_{L^2_x}
\bigg\Vert_{L^2_{t \in I} L^2_v} \\
& \qquad \quad = C \bigg\Vert \left\Vert \rho_f \right\Vert_{L^\infty_x}
\left\Vert \left< v \right>^\alpha
\left| \nabla_x \right|^\alpha f \right\Vert_{L^2_x}
\bigg\Vert_{L^2_{t \in I} L^2_v} +
\left\Vert \left< v \right>^\alpha f \left| \nabla_x \right|^\alpha \rho_f
\right\Vert_{L^2_{t \in I} L^2_{x,v}} 
\end{aligned}
\end{equation*}
Finally we apply H{\" o}lder's inequality, followed by Lemma
\ref{lem:A1lem} since $\alpha > \frac{1}{2}$; we are using the
fact that $\left| \nabla_x \right|^\alpha$ commutes with
$\rho_{(\cdot)}$. This yields:
\begin{equation*}
\begin{aligned}
& \left\Vert \left< v \right>^\alpha \left| \nabla_x \right|^\alpha
Q^- (f,f) \right\Vert_{L^2_{t \in I} L^2_{x,v}} \leq \\
& \qquad \quad \leq C \left\Vert \rho_f \right\Vert_{L^2_{t \in I} L^\infty_x}
\left\Vert \left< v \right>^\alpha \left| \nabla_x \right|^\alpha
f \right\Vert_{L^\infty_{t \in I} L^2_{x,v}} + \\
& \qquad \qquad \qquad \quad  + 
\left\Vert \left< v \right>^\alpha f \right\Vert_{L^\infty_{t \in I} L^4_x L^2_v}
\left\Vert \left| \nabla_x \right|^\alpha \rho_f 
\right\Vert_{L^2_{t \in I} L^4_x} \\
& \qquad \quad \leq C \left\Vert \rho_f \right\Vert_{L^2_{t \in I} L^\infty_x}
\left\Vert \left< v \right>^\alpha \left| \nabla_x \right|^\alpha
f \right\Vert_{L^\infty_{t \in I} L^2_{x,v}} + \\
& \qquad \qquad \qquad \quad + 
C \left\Vert \left< v \right>^\alpha f \right\Vert_{L^\infty_{t \in I} L^4_x L^2_v}
\left\Vert \left< v \right>^\alpha \left| \nabla_x \right|^\alpha f 
\right\Vert_{L^\infty_{t \in I} L^2_{x,v}} \\
& \qquad \qquad \qquad \quad  + 
C \left\Vert \left< v \right>^\alpha f \right\Vert_{L^\infty_{t \in I} L^4_x L^2_v}
\left\Vert \left< v \right>^\alpha \left| \nabla_x \right|^\alpha 
\left( \partial_t + v \cdot \nabla_x \right) f 
\right\Vert_{L^1_{t \in I} L^2_{x,v}} \\
\end{aligned}
\end{equation*}
hence the conclusion.
\end{proof}

\subsection{Gain operator bounds.}

The proof of Lemma \ref{lem:LossBilinUnsym}, which allows to apply spatial gradients to
one entry at a time in $Q^- (f,f)$, does not work for the gain operator $Q^+ (f,f)$
 in our formulation.
The difficulty is that we do not have an exact
 commutation rule for $\left| \nabla_x \right|^\alpha$
and $Q^+ (f,f)$, and the multilinear Riesz-Thorin theorem does not apply.

Nevertheless, it \emph{is} possible to recover a useful inequality in ``Peter-Paul'' form (before
optimizing) which estimates  fractional spatial derivatives
of the gain operator, which will be essential for the global propagation of regularity
to be proven in Subsection \ref{sec:reg-global}. The strategy is to apply the 
multilinear Riesz-Thorin theorem to well-chosen \emph{inhomogeneous} norms with
a suitable $\varepsilon$-dependent weight; then, we divide out powers of $\varepsilon$ from both sides,
and optimize over $\varepsilon$.
In this way, we are able to avoid any problem-specific commutator estimates, which would not
be in keeping with the spirit of our approach.

\begin{lemma}
\label{lem:GainPP}
Let $\alpha \in \left[0, 1 \right]$, and let
$I \subseteq \mathbb{R}$ be an open interval (either bounded or unbounded).
Let $f (t,x,v) : I \times \mathbb{R}^2 \times \mathbb{R}^2 \rightarrow \mathbb{C}$ be
a measurable and locally integrable function such that
\begin{equation}
\left< v \right>^\alpha
\left< \nabla_x \right>^\alpha f \in L^\infty \left( I, L^2_{x,v} \right)
\end{equation}
and
\begin{equation}
\left< v \right>^\alpha \left< \nabla_x \right>^\alpha
\left( \partial_t + v \cdot \nabla_x \right) f
\in L^1 \left( I, L^2_{x,v} \right)
\end{equation}
Then for any $q \in \left( 0, \infty \right)$ the following estimate holds:
\begin{equation}
\begin{aligned}
& \left\Vert \left< q v \right>^\alpha
\left| \nabla_x \right|^\alpha Q^+ (f,f) \right\Vert_{L^1 \left( I, L^2_{x,v} \right)} \leq \\
& \leq
C \left\Vert \left< q v \right>^\alpha f 
\right\Vert_{L^\infty \left( I, L^2_{x,v} \right)}
\left\Vert \left< q v \right>^\alpha \left| \nabla_x \right|^\alpha  f 
\right\Vert_{L^\infty \left( I, L^2_{x,v} \right)} + \\
& +
C \left\Vert \left<  q v \right>^\alpha ( \partial_t + v \cdot \nabla_x ) f 
\right\Vert_{L^1 \left( I, L^2_{x,v} \right)}
\left\Vert \left< q v \right>^\alpha \left| \nabla_x \right|^\alpha  f 
\right\Vert_{L^\infty \left( I, L^2_{x,v} \right)} + \\
& +
C \left\Vert \left< q v \right>^\alpha f 
\right\Vert_{L^\infty \left( I, L^2_{x,v} \right)}
\left\Vert \left< q v \right>^\alpha \left| \nabla_x \right|^\alpha \left( \partial_t
+ v \cdot \nabla_x \right)  f 
\right\Vert_{L^1 \left( I, L^2_{x,v} \right)} + \\& +
C \left\Vert \left< q v \right>^\alpha  \left( \partial_t + v \cdot \nabla_x \right) f 
\right\Vert_{L^1 \left( I, L^2_{x,v} \right)}
\left\Vert \left< q v \right>^\alpha \left| \nabla_x \right|^\alpha \left( \partial_t
+ v \cdot \nabla_x \right)  f 
\right\Vert_{L^1 \left( I, L^2_{x,v} \right)}
\end{aligned}
\end{equation}
The constant $C$ is independent of $I,q,\alpha$.
\end{lemma}
\begin{proof}
Adapting the proof of Proposition \ref{prop:gain-eq-bound-kin-alpha}
as necessary, by using the multilinear Riesz-Thorin
theorem we are able to show that for any $f_0 , g_0 \in H^{\alpha,\alpha}$,
$\alpha \in [0,1]$, and 
$q,\varepsilon \in \left( 0,\infty \right)$,
\begin{equation}
\label{eq:Qpe1}
\begin{aligned}
& \left\Vert \left< q v \right>^\alpha \left< \varepsilon \nabla_x \right>^\alpha
Q^+ \left( T(t) f_0, T(t) g_0 \right) \right\Vert_{L^1_t L^2_{x,v}} \leq \\
& \qquad \qquad \qquad 
\leq C \left\Vert \left< q v \right>^\alpha \left< \varepsilon \nabla_x \right>^\alpha
f_0 \right\Vert_{L^2_{x,v}}
\left\Vert \left< q v \right>^\alpha \left< \varepsilon \nabla_x \right>^\alpha
g_0 \right\Vert_{L^2_{x,v}}
\end{aligned}
\end{equation}
where the constant $C$ does not depend on $\alpha,q,\varepsilon$. It suffices to
check the endpoints $\alpha = 0$ and $\alpha = 1$, viewing $\varepsilon,q \in \left( 0,\infty \right)$
as arbitrary constants.

Having verified (\ref{eq:Qpe1}), 
let $f,g$ be time-dependent functions as in the statement of the lemma.
Combining (\ref{eq:Qpe1}) and Lemma \ref{lem:abswplem}, and using the
fact that $T(t)$ is an isometry on $L^2_{x,v}$ for each $t \in \mathbb{R}$, 
 we deduce the following estimate, up to increasing the constant by an absolute factor:
\begin{equation}
\begin{aligned}
& \left\Vert \left< q v \right>^\alpha
\left< \varepsilon \nabla_x \right>^\alpha 
Q^+ (f,g) \right\Vert_{L^1 \left( I, L^2_{x,v} \right)} \leq \\
& \qquad \leq C \prod_{h \in \left\{ f, \; g \right\} } \Bigg(
\left\Vert \left< q v \right>^\alpha \left< \varepsilon \nabla_x \right>^\alpha h 
\right\Vert_{L^\infty \left( I, L^2_{x,v} \right)} +  \\
& \qquad \qquad \qquad \qquad \qquad 
+ \left\Vert \left< q v \right>^\alpha \left< \varepsilon \nabla_x \right>^\alpha
\left( \partial_t + v \cdot \nabla_x \right) h
\right\Vert_{L^1 \left( I, L^2_{x,v}\right)} \Bigg) \\
\end{aligned}
\end{equation}
Now  may we specialize to the case $g = f$.
\begin{equation}
\begin{aligned}
& \left\Vert \left< q v \right>^\alpha
\left< \varepsilon \nabla_x \right>^\alpha 
Q^+ (f,f) \right\Vert_{L^1 \left( I, L^2_{x,v} \right)} \leq \\
& \qquad \leq C  \Bigg(
\left\Vert \left< q v \right>^\alpha \left< \varepsilon \nabla_x \right>^\alpha f
\right\Vert^2_{L^\infty \left( I, L^2_{x,v} \right)} +  \\
& \qquad \qquad \qquad \qquad \qquad 
+ \left\Vert \left< q v \right>^\alpha \left< \varepsilon \nabla_x \right>^\alpha
\left( \partial_t + v \cdot \nabla_x \right) f
\right\Vert^2_{L^1 \left( I, L^2_{x,v}\right)} \Bigg) \\
\end{aligned}
\end{equation}

At this point we need to estimate $\varepsilon^\alpha \left| \nabla_x \right|^\alpha \lesssim
\left< \varepsilon \nabla_x \right>^\alpha$ on the \emph{left}, and
$\left< \varepsilon \nabla_x \right>^\alpha \lesssim
1 + \varepsilon^\alpha \left| \nabla_x \right|^\alpha$ on the \emph{right} (and note the
squares!), and
finally, divide throughout by $\varepsilon^\alpha$. Hence we obtain the following
``Peter-Paul'' inequality:
\begin{equation}
\label{eq:GainPeterPaul}
\begin{aligned}
& \left\Vert \left< q v \right>^\alpha
\left| \nabla_x \right|^\alpha Q^+ (f,f) \right\Vert_{L^1 \left( I, L^2_{x,v} \right)} \leq \\
& \leq \frac{C}{\varepsilon^\alpha} \Bigg(
\left\Vert \left< q v \right>^\alpha f 
\right\Vert^2_{L^\infty \left( I, L^2_{x,v} \right)} +
\left\Vert \left< q v \right>^\alpha \left( \partial_t + v \cdot \nabla_x \right) f
\right\Vert^2_{L^1 \left( I, L^2_{x,v}\right)} \Bigg) + \\
& + C \varepsilon^\alpha \Bigg(
\left\Vert \left< q v \right>^\alpha \left| \nabla_x \right|^\alpha  f 
\right\Vert^2_{L^\infty \left( I, L^2_{x,v} \right)} +
\left\Vert \left< q v \right>^\alpha 
\left| \nabla_x \right|^\alpha  \left( \partial_t + v \cdot \nabla_x \right) f
\right\Vert^2_{L^1 \left( I, L^2_{x,v}\right)} \Bigg) \\
\end{aligned}
\end{equation}
The conclusion then follows by optimal choice of $\varepsilon$ and trivial manipulations.
\end{proof}

\subsection{Local propagation of regularity.}
\label{sec:reg-local}

The idea for proving local propagation of regularity is to construct
a local solution in the more regular space
$H^{\alpha,\alpha}$ with $\alpha > \frac{1}{2}$, and then appeal to
uniqueness via Theorem \ref{thm:unique} to conclude that the $H^{\alpha,\alpha}$
solution coincides with the small $L^2_{x,v}$ solution obtained from Kaniel-Shinbrot. The various estimates
required to apply Theorem \ref{thm:unique} to $H^{\alpha,\alpha}$ solutions
 follow immediately from the
local well-posedness theory in $H^{\alpha,\alpha}$ for
$\alpha > \frac{1}{2}$, combined with the Sobolev embedding theorem and
Lemma \ref{lem:A1lem}.\footnote{Interestingly, it was the local $H^{\alpha,\alpha}$ theory
with $\alpha > \frac{1}{2}$ which served as the inspiration for
Theorem \ref{thm:unique} in the first place (and, by extension, the proof of convergence
of the Kaniel-Shinbrot scheme).} The local theory presented here relies on the
$L^2_{x,v}$ norm remaining small, which is parallel to the assumption
for the uniqueness theorem, Theorem \ref{thm:unique}; however, the
$H^{\alpha,\alpha}$ norm may be very large and the local theory will
still be valid. The time of existence for local solutions given
$f_0 \in H^{\alpha,\alpha} \cap B_\eta^{L^2}$ is  determined solely by the magnitude of the
$H^{\alpha,\alpha}$ norm. A separate argument (discussed in
subsection \ref{sec:reg-global}) is required to obtain the propagation of
regularity on arbitrarily large time intervals.

Recall the $H^{\alpha,\alpha}$ norm with $\varepsilon$ dependence,
\begin{equation}
\left\Vert f \right\Vert_{H^{\alpha,\alpha}_\varepsilon} =
\left\Vert \left< \varepsilon v \right>^\alpha
\left< \varepsilon \nabla_x \right>^\alpha f 
\right\Vert_{L^2_{x,v}} 
\end{equation}
We know that the gain term $Q^+$ obeys the following estimate, by Proposition 
\ref{prop:gain-eq-bound-kin-alpha}
\begin{equation}
\label{eq:Qpbe}
\left\Vert Q^+ \left( T(t) f_0, T(t) g_0 \right) \right\Vert_{L^1_t H^{\alpha,\alpha}_\varepsilon}
\leq C \left\Vert f_0 \right\Vert_{H^{\alpha,\alpha}_\varepsilon}
\left\Vert g_0 \right\Vert_{H^{\alpha,\alpha}_\varepsilon}
\end{equation}
and the constant does not depend on $\alpha, \varepsilon \in \left( 0,1 \right]$. With respect
to the loss term, we cannot expect bounds independent of $\varepsilon$, but we can use
Lemma \ref{lem:LossBilin} to prove the following:
\begin{equation}
\left\Vert Q^- \left( T(t) f_0, T(t) g_0 \right) \right\Vert_{L^2_t H^{\alpha,\alpha}_\varepsilon}
\leq C_\varepsilon \left\Vert f_0 \right\Vert_{H^{\alpha,\alpha}_\varepsilon}
\left\Vert g_0 \right\Vert_{H^{\alpha,\alpha}_\varepsilon}
\end{equation}
 Hence by H{\" o}lder's inequality,
\begin{equation}
\label{eq:lossL1T}
\left\Vert Q^- \left( T(t) f_0, T(t) g_0 \right) \right\Vert_{L^1_{t \in [0,T]}
 H^{\alpha,\alpha}_\varepsilon}
\leq C_\varepsilon T^{\frac{1}{2}} \left\Vert f_0 \right\Vert_{H^{\alpha,\alpha}_\varepsilon}
\left\Vert g_0 \right\Vert_{H^{\alpha,\alpha}_\varepsilon}
\end{equation}
Note that the size of the constant $C_\varepsilon$ in (\ref{eq:lossL1T})
is irrelevant for our analysis, but it can
be estimated as $C_\varepsilon \lesssim \varepsilon^{- 4 \alpha}$ when $\varepsilon \rightarrow
0^+$. The point is that the large factor of $C_\varepsilon$ 
 can always be balanced in (\ref{eq:lossL1T}) by letting
$T$ be small. Since the parameter $\varepsilon$ reflects (in this instance) the size of the
$H^{\alpha,\alpha}$ norm at a given time $t_0$, we can apply Theorem \ref{thm:abswp} using
(\ref{eq:Qpbe}) and (\ref{eq:lossL1T}) to deduce local well-posedness for the full
Boltzmann equation 
in $H^{\alpha,\alpha} \cap B_\eta^{L^2_{x,v}}$ (for some constant $\eta > 0$), 
with existence time depending only
on the $H^{\alpha,\alpha}$ norm.

\begin{remark}
We can say nothing for $f_0$ outside the $\eta$-ball of $L^2$ by the above logic,
due to the fact that the constant $C$ in (\ref{eq:Qpbe}) remains fixed regardless of
any localization in time.
\end{remark}

As a result of the preceding discussion, we may conclude:
\begin{theorem}
\label{thm:BELWPalpha}
There exists a number $\eta > 0$ such that all the following is true:

Let $\alpha \in \left( \frac{1}{2}, 1 \right]$ and $f_0 \in H^{\alpha,\alpha}$, and further
suppose that
\begin{equation}
\label{eq:f0eta91}
\left\Vert f_0 (x,v) \right\Vert_{L^2_{x,v}} < \eta
\end{equation}
Then there exists a time $T > 0$ such that, for $t \in [0,T]$, the full Boltzmann equation
\begin{equation}
\left( \partial_t + v \cdot \nabla_x \right) f = Q^+ (f,f) - f \rho_f
\end{equation}
has a unique mild solution $f(t)$ such that $f \in L^\infty_{t \in [0,T]} H^{\alpha,\alpha}$,
 $Q^{\pm} (f,f) \in L^1_{t \in [0,T]} H^{\alpha,\alpha}$ and $f(0) = f_0$ all hold.
The solution is continuous, in the sense that $f \in C \left( [0,T],
H^{\alpha,\alpha} \right)$.
 Additionally,
the time $T$ may be chosen to depend only on the $H^{\alpha,\alpha}$ norm of $f_0$;
that is, the lower bound
\begin{equation*}
T \geq T_0 \left( \left\Vert f_0 \right\Vert_{H^{\alpha,\alpha}} \right) > 0
\qquad \textnormal{ assuming } \qquad
\left\Vert f_0 (x,v) \right\Vert_{L^2_{x,v}} < \eta
\end{equation*}
may be assumed.
\end{theorem}

\subsection{Global propagation of regularity.}
\label{sec:reg-global}

The key observation to round out our discussion of regularity is that we do
not have to propagate the \emph{entire} $H^{\alpha,\alpha}$ norm, because part of it
is given to us \emph{for free} by the Kaniel-Shinbrot iteration. Indeed we
already know that $\left< v \right>^\alpha f \in L^\infty_{t \geq 0} L^2_{x,v}$,
and similarly $\left< v \right>^\alpha Q^+ (f,f) \in L^1_{t \geq 0} L^2_{x,v}$
and $\left< v \right>^\alpha Q^- (f,f) \in L^2_{t \geq 0} L^2_{x,v}$.
(See Theorem \ref{thm:gain-eq-gwp-kin-alpha} and Section \ref{sec:KanielShinbrot}.)
Hence, we have only to show that
\begin{equation}
\forall T \in ( 0,\infty), \; \qquad
\left\Vert \left< v \right>^\alpha
\left| \nabla_x \right|^\alpha f \right\Vert_{L^\infty_{t \in [0,T]} L^2_{x,v}} < \infty
\end{equation}
and
\begin{equation}
\forall T \in ( 0,\infty), \; \qquad
\left\Vert \left< v \right>^\alpha
\left| \nabla_x \right|^\alpha Q^{\pm} (f,f) \right\Vert_{L^1_{t \in [0,T]} L^2_{x,v}} < \infty
\end{equation}
Note that Theorem \ref{thm:unique}, combined with Sobolev embedding, implies 
 that the local $H^{\alpha,\alpha}$ solution from Theorem \ref{thm:BELWPalpha}
 coincides with the solution obtained via Kaniel-Shinbrot. (This is due to
the fact that Theorem \ref{thm:unique} refers only to
\emph{integrability} properties, not \emph{regularity} properties,
in the $(x,v)$ domain.) Therefore, we can assume
that the $H^{\alpha,\alpha}$ norms are finite on small time intervals. We can then
use continuity arguments, combined with
 Lemma \ref{lem:LossBilinUnsym} and Lemma \ref{lem:GainPP},
 to extend the $H^{\alpha,\alpha}$ time up to a \emph{larger}
small time interval which now only depends on controlled quantities which do not contain
$\left| \nabla_x \right|^\alpha$. Finally, a standard iteration in time provides the
desired result.

\begin{theorem}
\label{thm:fullreg}
There exists an absolute constant $\eta > 0$ such that the following is true:

Let $T \in  \left( 0,\infty \right)$
and $\alpha \in \left( \frac{1}{2}, 1 \right]$, and suppose $f(t) \in C \left( [0,T], L^2_{x,v} \right)$
 is a mild solution of the full Boltzmann equation satisfying all of the following
estimates:
\begin{equation}
\left\Vert f \right\Vert_{L^\infty_{t \in [0,T]} L^2_{x,v}}
+ \left\Vert Q^+ \left( f,f \right) \right\Vert_{L^1_{t \in [0,T]} L^2_{x,v}}
< \eta
\end{equation}
\begin{equation}
\left< v \right>^\alpha Q^+ (f,f) \in L^1_{t \in [0,T]} L^2_{x,v}
\end{equation} 
\begin{equation}
\rho_f \in L^2_{t \in [0,T]} L^\infty_x \cap L^2_{t \in [0,T]} L^4_x
\end{equation}
\begin{equation}
\left< v \right>^\alpha f \in L^\infty_{t \in [0,T]} L^2_{x,v} \cap L^\infty_{t \in [0,T]} L^4_x L^2_v
\end{equation}
and $f (0) = f_0$. If in addition $f_0 \in H^{\alpha,\alpha}$, then
$f \in L^\infty_{t \in [0,T]} H^{\alpha,\alpha}$ and
$Q^{\pm} (f,f) \in L^1_{t \in [0,T]} H^{\alpha,\alpha}$.
\end{theorem}
\begin{remark}
We emphasize the ordering of quantifiers: A single $\eta > 0$ works simultaneously
for all $T > 0$. Also, the supplied estimates automatically
imply $\left< v \right>^\alpha Q^- (f,f) \in L^2_{t \in [0,T]} L^2_{x,v}$, by
H{\" o}lder's inequality.
\end{remark}
\begin{proof}
In view of Theorem \ref{thm:BELWPalpha}, Theorem \ref{thm:unique}, and the Sobolev embedding
theorem, we only need to formally estimate
$\left< v \right>^\alpha \left| \nabla_x \right|^\alpha f \in L^\infty_{t \in [0,T]} L^2_{x,v}$
and $\left< v \right>^\alpha \left| \nabla_x \right|^\alpha Q^{\pm} (f,f)
\in L^2_{x,v}$. Additionally, due to
 Lemma \ref{lem:LossBilin}, Proposition \ref{prop:gain-eq-bound-kin-alpha},
Lemma \ref{lem:abswplem}, and Duhamel's formula with $f_0 \in H^{\alpha,\alpha}$, it will be
enough to show:
\begin{equation}
\left< v \right>^\alpha \left| \nabla_x \right|^\alpha
\left( \partial_t + v \cdot \nabla_x \right) f \in L^1_{t \in [0,T]} L^2_{x,v}
\end{equation}

Suppose $0 \leq t_0 < T$ and $0 < s \leq T-t_0$, and let $e_{t_0} (s)$ denote the quantity
\begin{equation}
e_{t_0} (s) = \left\Vert \left< q v \right>^\alpha
\left| \nabla_x \right|^\alpha \left( \partial_t + v \cdot \nabla_x \right) f
\right\Vert_{L^1_{t \in [t_0, t_0 + s]} L^2_{x,v}}
\end{equation}
whenever it is well-defined, or $+\infty$ otherwise. Note that $e_0 (s) < +\infty$ for
some $s>0$ by Theorem \ref{thm:BELWPalpha} and Theorem \ref{thm:unique}. Additionally,
if $e_{t_0} (s) < +\infty$, then
$\lim_{s \rightarrow 0^+} e_{t_0} (s) = 0$ by the dominated convergence theorem.
We want to show that $e_0 (T) < +\infty$.

We define for convenience
\begin{equation}
M = \left\Vert \left< v \right>^\alpha f \right\Vert_{L^\infty_{t \in [0,T]} L^2_{x,v}}
+ \left\Vert \left< v \right>^\alpha f \right\Vert_{L^\infty_{t \in [0,T]} L^4_x L^2_v}
+ \left\Vert \rho_f \right\Vert_{L^2_{t \in [0,T]} L^\infty_x} +
\left\Vert \rho_f \right\Vert_{L^2_{t \in [0,T]} L^4_x}
\end{equation}

Pick a number $q \in (0,1)$ such that
\begin{equation}
\left\Vert \left< q v \right>^\alpha f \right\Vert_{L^\infty_{t \in [0,T]} L^2_{x,v}}
+ \left\Vert \left< q v \right>^\alpha
 Q^+ \left( f,f \right) \right\Vert_{L^1_{t \in [0,T]} L^2_{x,v}}
< \eta
\end{equation}
where $\eta$ is as in the statement of the theorem (the size of $\eta$ may be determined by
tracking constants through the proof).
 
Suppose $t_0, s$ are such that $e_0 (s + t_0) = e_0 (t_0) + e_{t_0} (s) < + \infty$ 
(here the allowable values of
 $t_0,s$ are determined by the
solution $f$ itself, not necessarily by the statement of Theorem \ref{thm:BELWPalpha}).
Since $f$ solves Boltzmann's equation, we clearly have
\begin{equation}
\begin{aligned}
& e_{t_0} (s) \leq  \\
&  \left\Vert \left< q v \right>^\alpha
\left| \nabla_x \right|^\alpha Q^+ (f,f)
\right\Vert_{L^1_{t \in [t_0, t_0 + s]} L^2_{x,v}}
+ \left\Vert \left< q v \right>^\alpha
\left| \nabla_x \right|^\alpha Q^- (f,f)
\right\Vert_{L^1_{t \in [t_0, t_0 + s]} L^2_{x,v}}
\end{aligned}
\end{equation}

We have, as an immediate consequence of Lemma \ref{lem:LossBilinUnsym}, the estimate
\begin{equation}
\begin{aligned}
& \left\Vert \left< q v \right>^\alpha
\left| \nabla_x \right|^\alpha Q^- (f,f) \right\Vert_{L^1_{t \in [t_0, t_0 + s]} L^2_{x,v}} \leq \\
& \qquad \qquad \qquad
 \leq C M s^{\frac{1}{2}} \left( \left\Vert f_0 \right\Vert_{H^{\alpha,\alpha}}
+ \frac{1}{q^\alpha} e_0 (t_0) + \frac{1}{q^{\alpha}} e_{t_0} (s) \right)
\end{aligned}
\end{equation}
Note that $s$ can be chosen, depending only on the parameters
$M,q$ fixed as above, to make the prefactor on $e_{t_0} (s)$ as small as
we like.

By Lemma \ref{lem:GainPP}, we have
\begin{equation}
\begin{aligned}
&  \left\Vert \left< q v \right>^\alpha
\left| \nabla_x \right|^\alpha Q^+ (f,f) \right\Vert_{L^1_{t \in [t_0, t_0 + s]} L^2_{x,v}}
\leq \\
& \qquad \qquad
\leq \textnormal{const.} \times  \left( \eta + M s^{\frac{1}{2}} \right) \Bigg(
\left\Vert f_0 \right\Vert_{H^{\alpha,\alpha}} + e_0 (t_0) + e_{t_0} (s) \Bigg)
\end{aligned}
\end{equation}
Combining estimates (and picking $\eta$ small enough once and for all), 
we find that there exists a number $\tilde{s} > 0$, depending on
the solution $f$ only through $M,q$, with the following
property: if $e_0 (t_0) < \infty$, then $e_{t_0} (s) < \infty$. This is sufficient to conclude
the theorem.
\end{proof}

\section{The local well-posedness theorem}
\label{sec:lwpthm}

In view of the Kaniel-Shinbrot iteration, in order to prove 
Theorem \ref{thm:lwp} it suffices to prove a suitable local
well-posedness theorem for the \emph{gain-only} Boltzmann
equation. This theorem will require $\alpha = \frac{1}{2}+$ 
regularity on $f_0$ but the time of existence will depend only
on the $H^{s,s}$ norm for given $s \in (0, \frac{1}{2} )$. 

Let us define the norms, for $\alpha \in (\frac{1}{2},1)$,
 $0 < \theta < 1$ and $\varepsilon > 0$,
\begin{equation}
\begin{aligned}
& \left\Vert f \right\Vert_{H^{\alpha,\alpha}_{\varepsilon,\theta}} = \\
& \left\Vert \left( 1 + \varepsilon^2 |v|^2 \right)^{\frac{\alpha}{2}(1-\theta)}
\left( 1 + \varepsilon^2 | \xi |^2 \right)^{\frac{\alpha}{2}(1-\theta)}
\left( 1 + |v|^2 \right)^{\frac{\alpha}{2}\theta}
\left( 1 + |\xi|^2 \right)^{\frac{\alpha}{2}\theta}
\mathcal{F}_x f \left( \xi, v \right) \right\Vert_{L^2_{\xi,v}}
\end{aligned}
\end{equation}
The $H^{\alpha,\alpha}_{\varepsilon,\theta}$ norm is equivalent
(up to powers of $\varepsilon$) to the $H^{\alpha,\alpha}$ norm,
but for small $\varepsilon$ the $H^{\alpha,\alpha}_{\varepsilon,\theta}$
norm is nearly equal to the $H^{s,s}$ norm where 
$s = \alpha \theta$. Also note that
$$
\left(
H^{\alpha,\alpha}_\varepsilon,
H^{\alpha,\alpha} \right)_\theta =
H^{\alpha,\alpha}_{\varepsilon,\theta}
$$
with equality of norms.

The following bilinear estimate is proven in
\cite{CDP2017}:

\begin{proposition}
Let $\alpha > \frac{1}{2}$. Then there is a constant
$C = C(\alpha)$ such that, for the constant collision kernel
in dimension $d=2$,
\begin{equation}
\left\Vert
Q^+ \left( T(t) f_0, T(t) g_0 \right) 
\right\Vert_{L^2_{t \in \mathbb{R}} H^{\alpha,\alpha}}
\leq C
\left\Vert f_0 \right\Vert_{H^{\alpha,\alpha}}
\left\Vert g_0 \right\Vert_{H^{\alpha,\alpha}}
\end{equation}
\end{proposition}

On the other hand, from
Proposition \ref{prop:gain-eq-bound-kin-alpha} we know that
\begin{equation}
\left\Vert
Q^+ \left( T(t) f_0, T(t) g_0 \right) 
\right\Vert_{L^1_{t \in \mathbb{R}} 
H^{\alpha,\alpha}_\varepsilon}
\leq C
\left\Vert f_0 \right\Vert_{H^{\alpha,\alpha}_\varepsilon}
\left\Vert g_0 \right\Vert_{H^{\alpha,\alpha}_\varepsilon}
\end{equation}
where the constant $C$ is independent of $\varepsilon$.

Interpolating these two estimates yields
\begin{equation}
\label{eq:interpQplus03}
\left\Vert
Q^+ \left( T(t) f_0, T(t) g_0 \right) 
\right\Vert_{L^{p_\theta}_{t \in \mathbb{R}} 
H^{\alpha,\alpha}_{\varepsilon,\theta}}
\leq C
\left\Vert f_0 \right\Vert_{H^{\alpha,\alpha}_{\varepsilon,\theta}}
\left\Vert g_0 \right\Vert_{H^{\alpha,\alpha}_{\varepsilon,\theta}}
\end{equation}
where $C$ is independent of $\varepsilon$ and 
$1/p_\theta = 1-\frac{1}{2}\theta$; note that
$p_\theta > 1$ for each $\theta \in (0,1)$.

The chain of reasoning is as follows. Let
$\alpha = \frac{1}{2}+$ and fix a desired regularity
$s \in (0,\frac{1}{2})$; then, $\theta$ is fixed so that
$s = \alpha \theta$. Let $f_0 \in H^{\alpha,\alpha}$ be an
arbitrary initial datum. By choosing $\varepsilon$ very small,
we can let the $H^{\alpha,\alpha}_{\varepsilon,\theta}$ norm
approach the $H^{s,s}$ norm of $f_0$, while the constant $C$ remains
fixed. This implies that the local time of existence depends only
on the $H^{s,s}$ norm of $f_0$, by an application of 
Theorem \ref{thm:abswp} (we have localized in time using that
$p_\theta > 1$). Altogether we will be able to conclude:

\begin{theorem}
\label{thm:GOlwp-s}
Let $f_0 \in H^{\frac{1}{2}+,\frac{1}{2}+}$ and fix
$s \in (0,\frac{1}{2})$. The gain-only Boltzmann equation
\begin{equation}
\label{eq:GOlwp-s-eqn}
\left( \partial_t + v \cdot \nabla_x \right)
f = Q^+ ( f,f )
\end{equation}
has a mild solution $f \in C \left( [0,T], H^{\frac{1}{2}+,
\frac{1}{2}+}\right)$
such that
$Q^+ (f,f) \in L^1_{t \in [0,T]} H^{\frac{1}{2}+,\frac{1}{2}+}$
and $f( t=0) = f_0$.
The solution is unique in the class of all mild solutions with
the same initial data 
satisfying $Q^+ (f,f) \in L^1_{t \in [0,T]} 
H^{\frac{1}{2}+,\frac{1}{2}+}$. The existence time $T$ depends
only on $s$ and the $H^{s,s}$ norm of $f_0$.
\end{theorem}

Once we have Theorem \ref{thm:GOlwp-s}, we
repeat the Kaniel-Shinbrot iteration as in 
subsection \ref{sec:KanielShinbrot} to conclude
Theorem \ref{thm:lwp}.

\begin{proof}
(Theorem \ref{thm:GOlwp-s})
Since $s \in (0,\frac{1}{2})$ and $\alpha = \frac{1}{2}+$, we can fix
$\theta \in (0,1)$ so that $s = \alpha \theta$; then, we have
$p_\theta > 1$ where $1 / p_\theta = 1 - \frac{1}{2} \theta$.

Under the change of variables
$$ \tilde{f} (t) = T(-t) f(t) $$
the equation (\ref{eq:GOlwp-s-eqn}) is transformed into
$$
\partial_t \tilde{f} (t) =
T(-t) Q^+ \left( T(t) \tilde{f} (t), T(t) \tilde{f} (t) \right)
$$
Fix a smooth, even function
$\psi (t) : \mathbb{R} \rightarrow \mathbb{R}$,
which is decreasing on $(0,\infty)$, equals $1$ on $(0,T)$, and equals
$0$ on $(2T,\infty)$. Then consider the equation
\begin{equation}
\label{eq:truncGO}
\partial_t \tilde{f} (t) = \mathcal{A} (t, \tilde{f}(t), \tilde{f}(t))
\end{equation}
where
$$
\mathcal{A} (t,f_0,g_0) =
\psi (t) T(-t) Q^+ \left( T(t) f_0, T(t) g_0 \right)
$$
and $\tilde{f} (t=0) = f_0$.
By (\ref{eq:interpQplus03}), the definition of $\psi$, and H{\" o}lder's inequality,
there holds
\begin{equation}
\left\Vert \mathcal{A} \left( t, f_0, g_0 \right) \right\Vert_{L^1_t 
H^{\alpha,\alpha}_{\varepsilon, \theta}} \leq C 
T^{\theta / 2}
\left\Vert f_0 \right\Vert_{H^{\alpha,\alpha}_{\varepsilon,\theta}}
\left\Vert g_0 \right\Vert_{H^{\alpha,\alpha}_{\varepsilon,\theta}}
\end{equation}
where the constant $C$ is independent of $\varepsilon$. By Theorem
\ref{thm:abswp}, equation (\ref{eq:truncGO}) is well-posed as long as
\begin{equation}
\left\Vert f_0 \right\Vert_{H^{\alpha,\alpha}_{\varepsilon,\theta}} \leq
C T^{- \theta  / 2}
\end{equation}
where the constant $C$ is again independent of $\varepsilon$. Letting $\varepsilon$
tend to zero, this condition becomes
\begin{equation}
\left\Vert f_0 \right\Vert_{H^{s,s}} \leq C T^{-\theta / 2}
\end{equation}
which was what we wanted.
\end{proof}

\appendix

\section{An Endpoint Strichartz Estimate}
\label{app:strich}

We recall Theorem 10.1 from \cite{KT1998}:

\begin{theorem}
\label{thm:KT}
\cite{KT1998}
Let $\sigma > 0$, $H$ be a Hilbert space and $B_0, B_1$ be Banach spaces.
Suppose that for each time $t$ we have an operator
$U(t) : H \rightarrow B_0^*$ such that
\begin{equation}
\left\Vert U(t) \right\Vert_{H \rightarrow B_0^*} \lesssim 1
\end{equation}
\begin{equation}
\left\Vert U(t) \left( U(s) \right)^* 
\right\Vert_{B_1 \rightarrow B_1^*} \lesssim |t-s|^{-\sigma}
\end{equation}
Let $B_\theta$ denote the real interpolation space
$\left( B_0, B_1 \right)_{\theta,2}$. Then we have the estimates
\begin{equation}
\left\Vert U(t) f \right\Vert_{L_t^q B_\theta^*} \lesssim
\left\Vert f \right\Vert_H
\end{equation}
\begin{equation}
\left\Vert \int \left( U(s) \right)^* F(s) ds \right\Vert_H
\lesssim \left\Vert F \right\Vert_{L_t^{q^\prime} B_\theta}
\end{equation}
\begin{equation}
\left\Vert
\int_{s < t} U(t) \left( U(s) \right)^* F(s) ds
\right\Vert_{L_t^q B_\theta^*} \lesssim
\left\Vert F \right\Vert_{L_t^{\tilde{q}^\prime} B_{\tilde{\theta}}^*}
\end{equation}
whenever $0 \leq \theta \leq 1$, $2 \leq q = \frac{2}{\sigma \theta}$,
$(q,\theta,\sigma) \neq (2,1,1)$, and similarly for
$( \tilde{q}, \tilde{\theta} )$. If the decay estimate is
strengthened to
\begin{equation}
\left\Vert U(t) \left( U(s) \right)^* \right\Vert_{B_1 \rightarrow B_1^*}
\lesssim \left( 1 + |t-s| \right)^{-\sigma}
\end{equation}
then the requirement $q = \frac{2}{\sigma \theta}$ can be relaxed to
$q \geq \frac{2}{\sigma \theta}$, and similarly for
$( \tilde{q}, \tilde{\theta} )$.
\end{theorem}

For our application, we will need to think of 
$\gamma_0 (x,x^\prime)$ as an arbitrary measurable complex-valued function of
$x,x^\prime \in \mathbb{R}^2$. Let us take
$H = L^2_{x,x^\prime} \left( \mathbb{R}^2 \times \mathbb{R}^2 \right)$, 
$B_0 = L^2_{x,x^\prime} \left( \mathbb{R}^2 \times \mathbb{R}^2 \right)$,
and
$B_1 = L^1_{x,x^\prime} \left( \mathbb{R}^2 \times \mathbb{R}^2 \right)$.
We employ the notation $\Delta_{\pm} = \Delta_{x} - \Delta_{x^\prime}$. The
energy estimate
\begin{equation}
\left\Vert e^{i t \Delta_{\pm}} \gamma_0 \right\Vert_{L^2_{x,x^\prime}}
\lesssim \left\Vert \gamma_0 \right\Vert_{L^2_{x,x^\prime}}
\end{equation}
is immediate. The dispersive estimate
\begin{equation}
\left\Vert e^{i (t-s) \Delta_{\pm}} \gamma_0 
\right\Vert_{L^\infty_{x,x^\prime}} \lesssim
|t-s|^{-2}
\left\Vert \gamma_0 (x,x^\prime) \right\Vert_{L^1_{x,x^\prime}}
\end{equation}
follows from writing the fundamental solution of
$\left( i \partial_t + \Delta_{\pm} \right) \gamma = 0$, that is
$$
\frac{1}{t^2}
 e^{i \left( |x|^2 - |x^\prime |^2 \right) / t}
$$
for initial data $\delta (x) \delta (x^\prime)$,
and applying
Young's inequality. The relevant parameters for Theorem
\ref{thm:KT}
are $q = 2$,
$\theta = \frac{1}{2}$ and $\sigma = 2$. 
The real interpolation space $(B_0, B_1)_{\theta,2}$ is
the Lorentz space $L^{\frac{4}{3},2}_{x,x^\prime}$
(\cite{BeLo1976} Theorem 5.3.1), and its dual is
$L^{4,2}_{x,x^\prime}$ (\cite{Gra2008} Theorem 1.4.17 (vi)), so we obtain
\begin{equation}
\left\Vert e^{i t \Delta_{\pm}} \gamma_0 \right\Vert_{L^2_t
L^{4,2}_{x,x^\prime} \left( \mathbb{R}^2 \times \mathbb{R}^2 \right)}
\lesssim \left\Vert \gamma_0 \right\Vert_{L^2_{x,x^\prime}
\left( \mathbb{R}^2 \times \mathbb{R}^2 \right)}
\end{equation}
which is the desired inequality.

\section{Fractional Leibniz Formulas}
\label{app:FracLeib}

\begin{theorem} 
\label{thm:FracLeibKPV}
Let $s \in \left( 0,1 \right)$ and $n \in \left\{ 2, 3, 4, 5, \dots \right\}$. 
Then if $f(x),g(x) : \mathbb{R}^n \rightarrow \mathbb{R}$ are measurable functions
such that $f \in H^s \left( \mathbb{R}^n \right)$ and
$g \in L^\infty \left( \mathbb{R}^n \right)$, 
then $\left( - \Delta \right)^{\frac{s}{2}} \left( fg \right)$ and
$f \left( -\Delta \right)^{\frac{s}{2}} g$ are
canonically identified with well-defined tempered distributions,
and their difference is in $L^2 \left( \mathbb{R}^n \right)$ and the
following estimate holds:
\begin{equation}
\label{eq:frac-est}
\left\Vert 
\left( -\Delta \right)^{\frac{s}{2}} \left( f g \right)
- f \left( - \Delta \right)^{\frac{s}{2}} g \right\Vert_{L^2 \left( \mathbb{R}^n \right)}
\leq
C \left( n,s \right)
\left\Vert \left( - \Delta \right)^{\frac{s}{2}} f
\right\Vert_{L^2 \left( \mathbb{R}^n \right)}
\left\Vert g \right\Vert_{L^\infty \left( \mathbb{R}^n \right)}
\end{equation}
\end{theorem}
\begin{proof} 
The estimate follows formally from \cite{KPV1993}, Appendix A,
Theorem A.12, in the one-dimensional case, for Schwartz functions $f,g$. 
(Also see 
\cite{2013arXiv1309.3291K} problem 5.1 and pp. 105--110 for the
multidimensional case.)
The objective here is to ensure that the result remains true in suitable
inhomogeneous Sobolev spaces; the argument is broken into three parts.

\emph{(i)}
 For $f,g$ in the Schwartz class, the estimate
(\ref{eq:frac-est}) is true due
to \cite{KPV1993}.

\emph{(ii)}
Keeping $f$ fixed in the Schwartz class, we can pass to the 
\emph{distributional} limit 
$g_n \rightharpoonup g \in L^\infty \left( \mathbb{R}^n \right)$
in (\ref{eq:frac-est}),
where each $g_n$ is Schwartz and uniformly bounded in $L^\infty$.
 Every term makes sense because 
$g$ is a tempered distribution and $f$ is Schwartz.

\emph{(iii)}
We need to pass to the limit $f_n \rightarrow f \in H^s \left(
\mathbb{R}^n \right)$ in (\ref{eq:frac-est}),
where the $f_n$ are Schwartz and uniformly bounded
in $H^s$, but $g \in
L^\infty \left( \mathbb{R}^n \right)$ is now \emph{fixed}.
Now $f_n , f $ are uniformly bounded in $H^s \left( \mathbb{R}^n \right)$,
 hence
uniformly bounded in $L^r \left( \mathbb{R}^n \right)$ where
$\frac{1}{2} - \frac{s}{n} = \frac{1}{r}$, by the Sobolev embedding
theorem. Hence
$f_n g$ and $fg$ are uniformly bounded in $L^r \left( 
\mathbb{R}^n \right)$, so they are well-defined tempered distributions,
as is $\left( -\Delta \right)^{\frac{s}{2}} \left( fg \right)$. 
For any Schwartz function $\psi$, the estimate
\begin{equation}
\begin{aligned}
\int \psi 
\left( -\Delta \right)^{\frac{s}{2}} \left( fg \right)
& \leq C 
\left\Vert
\left( -\Delta \right)^{\frac{s}{2}} \psi  \right\Vert_{L^{r^\prime}
\left( \mathbb{R}^n \right)} \left\Vert
f \right\Vert_{L^r \left(
\mathbb{R}^n \right)}
\left\Vert g \right\Vert_{L^\infty \left( \mathbb{R}^n \right)} \\
& \leq C \left\Vert
\left( -\Delta \right)^{\frac{s}{2}} \psi  \right\Vert_{L^{r^\prime}
\left( \mathbb{R}^n \right)} \left\Vert
\left( -\Delta \right)^{\frac{s}{2}} f \right\Vert_{L^2 \left(
\mathbb{R}^n \right)}
\left\Vert g \right\Vert_{L^\infty \left( \mathbb{R}^n \right)}
\end{aligned}
\end{equation}
where $\frac{1}{2} - \frac{s}{n} = \frac{1}{r}$,
 follows from duality, H{\" o}lder's inequality, and Sobolev's inequality.

Finally we deal with the term $f \left( -\Delta \right)^{\frac{s}{2}} g$. The
idea is to re-write it in the following  way:
\begin{equation}
f \left( -\Delta \right)^{\frac{s}{2}} g
= - \left\{ 
\left( -\Delta \right)^{\frac{s}{2}} \left( f g \right)
- f \left( - \Delta \right)^{\frac{s}{2}} g
\right\}
+ \left( -\Delta \right)^{\frac{s}{2}} \left( fg \right)
\end{equation}
so it is a difference of two things which apparently make sense. 
Using this difference formula and the commutator estimate of
Kenig-Ponce-Vega from the theorem statement, we can prove
the estimate
\begin{equation}
\begin{aligned}
& \int \psi  f
\left( -\Delta \right)^{\frac{s}{2}} g  \leq \\
& \qquad \quad \leq C 
\left(
\left\Vert \psi \right\Vert_{L^2 \left( \mathbb{R}^n \right)} + 
\left\Vert
\left( -\Delta \right)^{\frac{s}{2}} \psi  \right\Vert_{L^{r^\prime}
\left( \mathbb{R}^n \right)}
\right)
 \left\Vert
\left( -\Delta \right)^{\frac{s}{2}} f \right\Vert_{L^2 \left(
\mathbb{R}^n \right)}
\left\Vert g \right\Vert_{L^\infty \left( \mathbb{R}^n \right)}
\end{aligned}
\end{equation}
where $g \in
L^\infty \left( \mathbb{R}^n \right)$ and $f , \psi$ are in the
Schwartz class. 
We conclude (by density of Schwartz functions
in $H^s \left( \mathbb{R}^n \right)$) that $f \left( -\Delta \right)^{\frac{s}{2}} g$ is
canonically identified with
a well-defined tempered distribution whenever
$ f \in H^s \left( \mathbb{R}^n \right)$
and $g \in L^\infty \left( \mathbb{R}^n \right)$; moreover,
 we can take distributional
limits in $f_n$
 where needed (keeping $g \in L^\infty \left( \mathbb{R}^n \right)$
 fixed) to derive the desired estimate in this
class.
\end{proof}

\section{Some general estimates in $L^2 \bigcap L^1$}
\label{app:apriori}

Assume throughout this appendix that $0 \leq f_0 \in L^2_{x,v} \cap L^1_{x,v}$.
As is typical for a kinetic equation, we will consider a suitable
\emph{mollification} (with the same, i.e. \emph{unmollified}, initial
data $f_0$), which takes the following form:
\begin{equation}
\label{eq:moll1}
\left( \partial_t + v \cdot \nabla_x \right) f^n =
\frac{Q(f^n,f^n)}{1+\frac{1}{n} \rho_{f^n}}
\end{equation}
Here $n=1,2,3,\dots$ and $f^n (t=0) = f_0$. Note that we are not allowed
to mollify the data in general, because that would change the profile of
the data, and we are looking for local well-posedness in the critical
space $L^2$ (with an auxiliary $L^1$ estimate). It is well-known that
the mollified equation (\ref{eq:moll1}) is globally well-posed for
initial data $f_0 \in L^1$; the proof is by a Picard iteration and
time-stepping procedure. \cite{DPL1989}

Since $Q = Q^+ - Q^-$ (both non-negative) and $\rho_{f^n} \geq 0$,
we can conclude
\begin{equation}
\left( \partial_t + v \cdot \nabla_x \right) f^n \leq Q^+ (f^n,f^n)
\end{equation}
which implies
\begin{equation}
f^n (t) \leq T(t) f_0 + \int_0^t T(t-t^\prime)
Q^+ ( f^n,f^n) (t^\prime) dt^\prime
\end{equation}
where $T(t) = e^{-t v \cdot \nabla_x}$. In particular, for $0 \leq t \leq T$,
\begin{equation}
f^n (t) \leq T(t) f_0 + \int_0^T T(t-t^\prime)
Q^+ ( f^n,f^n) (t^\prime) dt^\prime
\end{equation}
 Apply $Q^+ (\cdot,\cdot)$ to both
sides of this inequality and apply monotonicity to obtain
\begin{equation}
\begin{aligned}
& Q^+ (f^n,f^n) \leq Q^+ \left( T(t) f_0, T(t) f_0 \right) \\
& + \int_0^T Q^+ \left( T(t) f_0, T(t-t^\prime) 
Q^+ (f^n,f^n) (t^\prime) \right) dt^\prime \\
& + \int_0^T Q^+ \left( T(t-t^\prime) Q^+ (f^n,f^n) (t^\prime), T(t) f_0\right) 
dt^\prime \\
& + \int_0^T \int_0^T Q^+ \left(
T(t-t^\prime) Q^+ (f^n,f^n) (t^\prime),
T(t-t^{\prime \prime}) Q^+ (f^n,f^n) (t^{\prime \prime}) \right) dt^\prime
dt^{\prime \prime}
\end{aligned}
\end{equation}
Now we take the $L^1_{t \in [0,T]} L^2_{x,v}$ norm of both sides (noting that
this quantity might be infinite), and 
apply Minkowski's inequality.
\begin{equation}
\begin{aligned}
& \left\Vert Q^+ (f^n,f^n) \right\Vert_{L^1_{t \in [0,T]} L^2_{x,v}}
 \leq \left\Vert Q^+ \left( T(t) f_0, T(t) f_0 \right)
\right\Vert_{L^1_{t \in [0,T]} L^2_{x,v}} \\
& + \int_0^T \left\Vert Q^+ \left( T(t) f_0, T(t-t^\prime) 
Q^+ (f^n,f^n) (t^\prime) \right)\right\Vert_{L^1_{t \in [0,T]} L^2_{x,v}} dt^\prime \\
& + \int_0^T \left\Vert
Q^+ \left( T(t-t^\prime) Q^+ (f^n,f^n) (t^\prime), T(t) f_0\right)
\right\Vert_{L^1_{t \in [0,T]} L^2_{x,v}} 
dt^\prime \\
& + \int_0^T \int_0^T dt^\prime
dt^{\prime \prime} \times \\
& \left\Vert Q^+ \left(
T(t-t^\prime) Q^+ (f^n,f^n) (t^\prime),
T(t-t^{\prime \prime}) Q^+ (f^n,f^n) (t^{\prime \prime}) \right) 
\right\Vert_{L^1_{t \in [0,T]} L^2_{x,v}} 
\end{aligned}
\end{equation}
Apply Proposition \ref{prop:kin-gain-bound} to the last term,
and Proposition \ref{prop:shorttime} to the first three terms, to obtain:
\begin{equation}
\begin{aligned}
& \left\Vert Q^+ (f^n,f^n) \right\Vert_{L^1_{t \in [0,T]} L^2_{x,v}}
 \leq  \delta_{f_0} (T)  \left\Vert f_0 \right\Vert_{L^2_{x,v}} \\
& + 2 \int_0^T \delta_{f_0} (T) \left\Vert
 T(-t^\prime) Q^+ (f^n,f^n) (t^\prime)
\right\Vert_{ L^2_{x,v}} 
dt^\prime \\
& + \int_0^T \int_0^T  C \left\Vert 
T(-t^{\prime}) Q^+ (f^n,f^n) (t^{\prime}) 
\right\Vert_{ L^2_{x,v}} 
\left\Vert 
T(-t^{\prime \prime}) Q^+ (f^n,f^n) (t^{\prime \prime}) 
\right\Vert_{ L^2_{x,v}} dt^\prime
dt^{\prime \prime} 
\end{aligned}
\end{equation}
where $\limsup_{T \rightarrow 0^+} \delta_{f_0} (T) = 0$ (note that
$\delta_{f_0} (T)$ depends on the profile of the data for any fixed $T>0$).

Overall we conclude
\begin{equation}
\begin{aligned}
 & \left\Vert Q^+ (f^n,f^n) \right\Vert_{L^1_{t \in [0,T]} L^2_{x,v}}
 \leq 
 \delta_{f_0} (T) \left\Vert f_0 \right\Vert_{L^2_{x,v}} \\
& \qquad + 2 \delta_{f_0} (T) 
\left\Vert Q^+ (f^n,f^n) \right\Vert_{L^1_{t \in [0,T]} L^2_{x,v}}
+ C
\left\Vert Q^+ (f^n,f^n) \right\Vert_{L^1_{t \in [0,T]} L^2_{x,v}}^2
\end{aligned}
\end{equation}
where $\limsup_{T \rightarrow 0^+} \delta_{f_0} (T) = 0$.
In the case that $\left\Vert Q^+ (f^n,f^n) \right\Vert_{L^1_{t \in [0,T_0]} L^2_{x,v}}$
is finite for \emph{some} $T_0 > 0$, a standard continuity argument allows
us to bound this quantity \emph{uniformly in $n$} up to some other small
time $T>0$ which depends on $f_0$. We can state this is an alternative:
there are numbers $C(f_0),T(f_0)$ such that, for each $n$, exactly
 one of the following holds:
\begin{enumerate}
\item \emph{Case 1:} 
$\left\Vert Q^+ (f^n,f^n) \right\Vert_{L^1_{t \in [0,\sigma]} L^2_{x,v}} = \infty$ for
every $\sigma >0$

\item \emph{Case 2:}
 $\left\Vert Q^+ (f^n,f^n) \right\Vert_{L^1_{t \in [0,T (f_0)]} L^2_{x,v}}
\leq C (f_0) $
\end{enumerate}
In particular $C(f_0),T(f_0)$ are independent of $n$; hence, as long as Case 2
holds for infinitely many $n$, we can hope for a 
compactness argument. 
Note that once $Q^+ (f^n,f^n)$ is placed uniformly
in $L^1_{t \in [0,T (f_0)]} L^2_{x,v}$, the method of Section \ref{sec:scatteringL2}
implies that 
$$T(-t) \frac{Q^- (f^n,f^n)(t)}{1+\frac{1}{n} \rho_{f^n (t)}}$$
 is uniformly bounded in 
$L^2_{x,v} L^1_{t \in [0,T (f_0)]}$; in particular, 
$\left( \partial_t + v \cdot \nabla_x \right) f^n$
is locally
integrable in $(t,x,v)$, boundedly with respect to $n$. Moreover, on $[0,T(f_0)]$,
 $f^n$ satisfies the full range of
Strichartz estimates expected for $L^2$ solutions of the free transp1ort 
equation, uniformly in $n$.

\begin{remark}
The classical $L^1$
 velocity averaging lemma used in \cite{DPL1989} requires that
\emph{both} $f^n$ and $\left( \partial_t + v \cdot \nabla_x \right) f^n$
 are relatively weakly compact in $L^1 (K)$
for compact sets $K \subset [0,\infty) \times \mathbb{R}^2_x \times
\mathbb{R}^2_v$. 
However, a refinement cited as Lemma 4.1 in \cite{Ar2011} states that,
under the condition that $f^n$ is relatively weakly compact in $L^1 (K)$ for
compact sets $K$, it suffices for
$\left( \partial_t + v \cdot \nabla_x \right) f^n$
 to be uniformly \emph{bounded} in $L^1 (K)$ for compact $K$.
\end{remark}

\input{wigner-strichartz.bbl}

%\bibliography{wigner-strichartz}

\end{document}

%% file: wigner-strichartz.bbl
% \bib, bibdiv, biblist are defined by the amsrefs package.

%% file: wigner-strichartz-arXiv-v2-corrected.bbl
\begin{bibdiv}
\begin{biblist}

\bib{AMUXY2011}{article}{
      author={Alexandre, R.},
      author={Morimoto, Y.},
      author={Ukai, S.},
      author={Xu, C.-J.},
      author={Yang, T.},
       title={Global existence and full regularity of the \uppercase{B}oltzmann
  equation without angular cutoff},
        date={2011},
     journal={Comm. Math. Phys.},
      volume={304},
      number={2},
       pages={513\ndash 581},
}

\bib{Alonso2009}{article}{
      author={Alonso, Ricardo~J.},
      author={Gamba, Irene~M.},
       title={Distributional and classical solutions to the \uppercase{C}auchy
  \uppercase{B}oltzmann problem for soft potentials with integrable angular
  cross section},
        date={2009},
        ISSN={1572-9613},
     journal={Journal of Statistical Physics},
      volume={137},
      number={5},
       pages={1147},
         url={https://doi.org/10.1007/s10955-009-9873-3},
}

\bib{Ar2011}{article}{
      author={Arsenio, D.},
       title={On the global existence of mild solutions to the
  \uppercase{B}oltzmann equation for small data in \uppercase{$L^D$}},
        date={2011},
     journal={Comm. Math. Phys.},
      volume={302},
      number={2},
       pages={453\ndash 476},
}

\bib{Bennettetal2014}{article}{
      author={Bennett, Jonathan},
      author={Bez, Neal},
      author={Gutiérrez, Susana},
      author={Lee, Sanghyuk},
       title={On the \uppercase{S}trichartz estimates for the kinetic transport
  equation},
        date={2014},
     journal={Communications in Partial Differential Equations},
      volume={39},
      number={10},
       pages={1821\ndash 1826},
}

\bib{BeLo1976}{book}{
      author={Bergh, J.},
      author={L{\" o}fstr{\" o}m, J.},
       title={Interpolation spaces},
      series={Grundlehren der mathematischen Wissenschaften},
   publisher={Springer-Verlag Berlin Heidelberg},
        date={1976},
      volume={223},
}

\bib{CIP1994}{book}{
      author={Cercignani, C.},
      author={Illner, R.},
      author={Pulvirenti, M.},
       title={The mathematical theory of dilute gases},
   publisher={Springer Verlag},
        date={1994},
}

\bib{CDP2017}{article}{
      author={{Chen}, T.},
      author={{Denlinger}, R.},
      author={{Pavlovic}, N.},
       title={{Local well-posedness for Boltzmann's equation and the Boltzmann
  hierarchy via Wigner transform}},
        date={2019},
     journal={Communications in Mathematical Physics},
      volume={368},
}

\bib{CDP-DCDS2019}{article}{
      author={{Chen}, T.},
      author={{Denlinger}, R.},
      author={{Pavlovic}, N.},
       title={{Moments and Regularity for a Boltzmann Equation via Wigner
  Transform}},
        date={2019},
     journal={Discrete and Continuous Dynamical Systems A},
      volume={39},
      number={9},
       pages={4979\ndash 5015},
}

\bib{DPL1989}{article}{
      author={DiPerna, R.~J.},
      author={Lions, P.-L.},
       title={On the \uppercase{C}auchy problem for \uppercase{B}oltzmann
  equations: \uppercase{G}lobal existence and weak stability},
        date={1989},
     journal={Ann. Math.},
      volume={130},
      number={2},
       pages={321\ndash 366},
}

\bib{Gra2008}{book}{
      author={Grafakos, L.},
       title={Classical \uppercase{F}ourier analysis},
     edition={2},
      series={Graduate Texts in Mathematics},
   publisher={Springer-Verlag New York},
        date={2008},
      volume={249},
}

\bib{GS2011}{article}{
      author={Gressman, P.~T.},
      author={Strain, R.~M.},
       title={Global classical solutions of the \uppercase{B}oltzmann equation
  without angular cut-off},
        date={2011},
     journal={J. Amer. Math. Soc.},
      volume={24},
      number={3},
       pages={771\ndash 847},
}

\bib{Gu2003}{article}{
      author={Guo, Y.},
       title={Classical solutions to the \uppercase{B}oltzmann equation for
  molecules with an angular cutoff},
        date={2003},
     journal={Archive for Rational Mechanics and Analysis},
      volume={169},
      number={4},
       pages={305\ndash 353},
}

\bib{HeJiang2017}{article}{
      author={He, Lingbing},
      author={Jiang, Jin-Cheng},
       title={Well-posedness and scattering for the \uppercase{B}oltzmann
  equations: Soft potential with cut-off},
        date={2017Jul},
     journal={Journal of Statistical Physics},
      volume={168},
      number={2},
       pages={470\ndash 481},
}

\bib{KS1984}{article}{
      author={Kaniel, S.},
      author={Shinbrot, M.},
       title={The \uppercase{B}oltzmann equation: \uppercase{I}.
  \uppercase{U}niqueness and local existence},
        date={1978},
     journal={Communications in Mathematical Physics},
      volume={58},
      number={1},
       pages={65\ndash 84},
}

\bib{KT1998}{article}{
      author={Keel, M.},
      author={Tao, T.},
       title={Endpoint \uppercase{S}trichartz estimates},
        date={1998},
     journal={Amer. J. Math},
      volume={120},
      number={5},
       pages={955\ndash 980},
}

\bib{2013arXiv1309.3291K}{article}{
      author={{Kenig}, C.},
       title={{The Cauchy problem for the quasilinear Schrodinger equation}},
        date={2013-09},
     journal={ArXiv e-prints},
      eprint={1309.3291},
}

\bib{KPV1993}{article}{
      author={Kenig, Carlos~E.},
      author={Ponce, Gustavo},
      author={Vega, Luis},
       title={Well-posedness and scattering results for the generalized
  \uppercase{K}orteweg-de \uppercase{V}ries equation via the contraction
  principle},
        date={1993},
     journal={Communications on Pure and Applied Mathematics},
      volume={46},
      number={4},
       pages={527\ndash 620},
}

\bib{KM2008}{article}{
      author={Klainerman, S.},
      author={Machedon, M.},
       title={On the uniqueness of solutions to the
  \uppercase{G}ross-\uppercase{P}itaevskii hierarchy},
        date={2008},
     journal={Comm. Math. Phys.},
      volume={279},
      number={1},
       pages={169\ndash 185},
}

\bib{KM1993}{article}{
      author={S., Klainerman},
      author={M., Machedon},
       title={Space‐time estimates for null forms and the local existence
  theorem},
        date={1993},
     journal={Communications on Pure and Applied Mathematics},
      volume={46},
      number={9},
       pages={1221\ndash 1268},
}

\bib{St1970}{book}{
      author={Stein, E.~M.},
       title={Singular integrals and differentiability properties of
  functions},
   publisher={Princeton University Press},
        date={1970},
        ISBN={9780691080796},
}

\bib{To1988}{article}{
      author={Toscani, G.},
       title={Global solution of the initial value problem for the
  \uppercase{B}oltzmann equation near a local \uppercase{M}axwellian},
        date={1988},
     journal={Archive for Rational Mechanics and Analysis},
      volume={102},
      number={3},
       pages={231\ndash 241},
}

\bib{Uk1974}{article}{
      author={Ukai, S.},
       title={On the existence of global solutions of mixed problem for the
  non-linear \uppercase{B}oltzmann equation},
        date={1974},
     journal={Proc. Japan Acad.},
      volume={50},
      number={3},
       pages={179\ndash 184},
}

\end{biblist}
\end{bibdiv}
